\documentclass[11pt, oneside]{article}   	
\usepackage{geometry}                		
\usepackage{graphicx}				
\usepackage{amssymb}
\usepackage{amsmath}
\usepackage{amsthm}
\usepackage{color}					
\usepackage{hyperref}				
\usepackage{mathtools}				
\usepackage{comment}				
\usepackage{amscd}					
\usepackage{enumerate}				
\usepackage{extarrows}				
\usepackage{mathtools}
\usepackage{mathrsfs}				
\usepackage{MnSymbol}				

\usepackage{tikz}
\usetikzlibrary{matrix,arrows}
\tikzset{node distance=2cm, auto}
\usepackage{tikz-cd}

\usepackage[small,nohug,heads=LaTeX]{diagrams}    		
\diagramstyle[labelstyle=\scriptstyle]
\newcommand*{\DashedArrow}[1][]{\mathbin{\tikz [baseline=-0.25ex,-latex, dashed,#1] \draw [#1] (0pt,0.5ex) -- (1.3em,0.5ex);}}

\newtheorem{theorem}{Theorem}[section]
\newtheorem*{theorem*}{Theorem}

\newtheorem{lemma}[theorem]{Lemma}
\newtheorem{proposition}[theorem]{Proposition}
\newtheorem*{proposition*}{Proposition}

\newtheorem{corollary}[theorem]{Corollary}
\newtheorem*{corollary*}{Corollary}

\theoremstyle{definition} 	\newtheorem{definition}[theorem]{Definition}
\newtheorem*{definition*}{Definition}

\theoremstyle{remark}    	\newtheorem{remark}[theorem]{Remark}
\theoremstyle{definition} 	\newtheorem{notation}[theorem]{Notation}
\theoremstyle{definition}	
\theoremstyle{remark}   	
\theoremstyle{definition}	\newtheorem*{conjecture*}{Conjecture}
					\newtheorem{conjecture}[theorem]{Conjecture}
					\newtheorem{conjectureintro}{Conjecture}

\newcommand{\nc}{\newcommand}

\nc{\UCubic}{\calC ub}

\nc{\CC}{\mathbb{C}}
\nc{\HH}{\mathbb{H}}
\nc{\LL}{\mathbb{L}}
\nc{\PP}{\mathbb{P}}
\nc{\QQ}{\mathbb{Q}}
\nc{\RR}{\mathbb{R}}
\nc{\BBS}{\mathbb{S}}
\nc{\ZZ}{\mathbb{Z}}

\nc{\kk}{\mathsf{k}}

\nc{\scrL}{\mathscr{L}}
\nc{\scrM}{\mathscr{M}}

\nc{\Bg}{B}
\nc{\Bs}{B^{\sfs}}
\nc{\Bn}{B^{\sfn}}
\nc{\sfs}{\mathsf{s}}
\nc{\sfn}{\mathsf{n}}

\nc{\Gr}{\mathbb{G}\mathrm{r}}
\nc{\Ver}{\mathrm{Ver}}
\nc{\Bl}{\mathrm{Bl}}
\nc{\Pic}{\mathrm{Pic}}
\nc{\Hilb}{\mathsf{Hilb}}
\nc{\Flag}{\mathbb{F}\mathrm{l}}
\nc{\PGL}{PGL}
\nc{\GL}{GL}
\nc{\PSL}{PSL}
\nc{\SSL}{SL}
\nc{\SL}{SL}
\nc{\LieSL}{\mathfrak{sl}}

\nc{\Ker}{\mathrm{Ker}}
\nc{\Coker}{\mathrm{Coker}}
\nc{\coker}{\mathrm{coker}}
\nc{\Image}{\mathrm{Im}}
\nc{\rank}{\mathrm{rank}}
\nc{\corank}{\mathrm{corank}}
\nc{\codim}{\mathrm{codim}}
\nc{\ch}{\mathrm{ch}}
\nc{\Tor}{\mathrm{Tor}}
\nc{\Lotimes}{\stackrel{L}{\otimes}}
\nc{\Ev}{Ev}
\nc{\Kosz}{\mathsf{Kosz}}
\nc{\cone}{\mathsf{cone}}
\nc{\Proj}{\mathsf{Proj}}
\nc{\ev}{\mathsf{ev}}
\nc{\Sym}{\mathsf{Sym}}

\nc{\calU}{\mathcal{U}}
\nc{\calO}{\mathcal{O}}
\nc{\calI}{\mathcal{I}}
\nc{\calC}{\mathcal{C}}
\nc{\calL}{\mathcal{L}}
\nc{\calM}{\mathcal{M}}
\nc{\calT}{\mathcal{T}}
\nc{\calK}{\mathcal{K}}
\nc{\calHom}{\mathcal{H}\mathit{om}}
\nc{\calExt}{\mathcal{E}\mathit{xt}}

\nc{\restr}[2]{{#1}\vert_{#2}}
\nc{\korth}{\kappa}
\nc{\Dtr}{D_{tri}}
\nc{\Dfat}{D_{fat}}
\nc{\D}{D}
\nc{\Id}{Id}
\nc{\ST}{\mathsf{ST}}
\nc{\BC}{\mathsf{BC}}
\nc{\tr}{\mathsf{tri}}
\nc{\Tr}{\mathsf{Tri}}
\nc{\PhiK}{\Phi_{\Bl}(\CO_{\PP(K)}(2h))}

\nc{\Hom}{\mathop{\mathsf{Hom}}\nolimits}
\nc{\Ext}{\mathop{\mathsf{Ext}}\nolimits}
\nc{\ext}{\mathop{\mathrm{ext}}\nolimits}
\nc{\CRHom}{\mathit{R}\mathcal{H}\mathit{om}}
\nc{\RHom}{\mathsf{RHom}}
\nc{\Cone}{\mathsf{Cone}}
\nc{\Supp}{\mathsf{Supp}}

\nc{\Cm}{m}

\nc{\CA}{\mathcal{A}}
\nc{\CB}{\mathcal{B}}
\nc{\calB}{\mathcal{B}}
\nc{\calD}{\mathcal{D}}
\nc{\CE}{\mathcal{E}}
\nc{\CF}{\mathcal{F}}
\nc{\CG}{\mathcal{G}}
\nc{\CH}{\mathcal{H}}
\nc{\CI}{\mathcal{I}}
\nc{\CJ}{\mathcal{J}}
\nc{\CK}{\mathcal{K}}
\nc{\CL}{\mathcal{L}}
\nc{\CM}{\mathcal{M}}
\nc{\CN}{\mathcal{N}}
\nc{\CO}{\mathcal{O}}
\nc{\CQ}{\mathcal{Q}}
\nc{\CR}{\mathcal{R}}
\nc{\CS}{\mathcal{S}}
\nc{\CT}{\mathcal{T}}
\nc{\CU}{\mathcal{U}}

\nc{\WC}{\widetilde{C}}
\nc{\WE}{\widetilde{E}}
\nc{\WL}{\widetilde{L}}
\nc{\WZ}{\widetilde{Z}}
\nc{\Wpi}{\widetilde{\pi}}
\nc{\Wsigma}{\widetilde{\sigma}}
\nc{\WDtr}{\widetilde{\Dtr}}
\nc{\WDfat}{\widetilde{\Dfat}}
\nc{\WCL}{\widetilde{\CL}}
\nc{\WCN}{\widetilde{\CN}}

\nc{\CMI}{\mathcal{MI}}
\nc{\CMIbar}{\overline{\CMI}}
\nc{\CMItil}{\widetilde{\CMI}}
\nc{\WCMI}{\widetilde{\CMI}}
\nc{\CHom}{\mathcal{H}\mathit{om}}
\nc{\CTaut}{\mathcal{T}}
\nc{\CTang}{\mathcal{T}}
\nc{\CSTaut}{\mathcal{T}^\mathsf{sp}}
\nc{\CRelTaut}{\mathcal{T}^\mathsf{rel}}
\nc{\CKer}{\mathcal{K}\mathrm{er}}
\nc{\JL}{\mathsf{JL}}
\nc{\JC}{\mathsf{JC}}
\nc{\divQ}{Q}

\nc{\Fe}{\mathfrak{e}}
\nc{\Fm}{\mathfrak{m}}
\nc{\Fn}{\mathfrak{n}}
\nc{\FA}{\mathfrak{A}}
\nc{\FMI}{\mathfrak{MI}}

\nc{\SGr}{\Gr^{\mathsf{sp}}}
\nc{\SCL}{\CL^{\mathsf{sp}}}
\nc{\Scoker}{\coker}
\nc{\RelGr}{\Gr_{\Bg}(3,A \otimes \CK)}
\nc{\Gropen}{\mathbb{G}\mathrm{rop}}
\nc{\Mon}{\textsf{M}}
\nc{\cod}{\codim}
\nc{\Spec}{\mathsf{Spec}}
\nc{\length}{length}
\nc{\UL}{\CL}
\nc{\IQ}{IQ}
\nc{\ILC}{\mathsf{ILC}}
\nc{\UIL}{I_\UL}
\nc{\Imsigma}{\sigma(A)}
\nc{\Imkappa}{\kappa(A)}
\nc{\Imtau}{\tau(A)}

\nc{\norm}{\textrm{norm}}

\nc{\eps}{\epsilon}
\nc{\Lvee}{\prescript{\vee}{}}
\nc{\Lperp}{\prescript{\perp}{}}

\nc{\red}[1]{\textcolor{red}{#1}}
\nc{\green}[1]{\textcolor{green}{#1}}
\nc{\blue}[1]{\textcolor{blue}{#1}}
\nc{\purple}[1]{\textcolor{purple}{#1}}
\nc{\magenta}[1]{\textcolor{magenta}{#1}}

\nc{\bb}{b}
\nc{\cc}{c}
\nc{\dd}{d}
\nc{\ee}{e}

\title{Rational curves and instantons on the Fano threefold $Y_5$}
\author{Giangiacomo Sanna}

\begin{document}
\maketitle

\begin{abstract}
This thesis is an investigation of the moduli spaces of instanton bundles on the Fano threefold $Y_5$ (a linear section of $\Gr(2,5)$). It contains new proofs of classical facts about lines, conics and cubics on $Y_5$, and about linear sections of $Y_5$.

The main original results are a Grauert--M\"ulich theorem for the splitting type of instantons on conics, a bound to the splitting type of instantons on lines and an $SL_2$-equivariant description of the moduli space in charge $2$ and $3$.

Using these results we prove the existence of a unique $SL_2$-equivariant instanton of minimal charge and we show that for all instantons of charge $2$ the divisor of jumping lines is smooth. In charge $3$, we provide examples of instantons with reducible divisor of jumping lines. Finally, we construct a natural compactification for the moduli space of instantons of charge $3$, together with a small resolution of singularities for it.
\end{abstract}

\tableofcontents

\section*{Introduction}

The goal of this thesis is an investigation of instanton vector bundles on a Fano threefold of index 2 and degree 5,
a natural generalization of instanton vector bundles on $\PP^3$. 

\begin{definition*}
An instanton on $\PP^3$ is a $\mu$-stable vector bundle $E$ of rank $2$ such that $c_1(E) = 0$ and $H^1(E(-2)) = 0$. Its second Chern class $c_2(E)$ is called the \emph{charge} of $E$.
\end{definition*}

The standard approach \cite{atiyah1978construction} (see also \cite{okonek1980vector}) to the description of instantons on $\PP^3$ and of their moduli space is via the Beilinson spectral sequence. 
Using this approach one can write any instanton $E$ of charge $n$ as the middle cohomology of a complex
\begin{equation}\label{intro:monad for P3}
0 \to \CO_{\PP^3}(-1)^{\oplus n} \to \CO_{\PP^3}^{\oplus 2n+2} \to \CO_{\PP^3}(1)^{\oplus n} \to 0
\end{equation}
such that the first map is fiberwise injective and the last map is surjective. 
Such a complex is usually called a \emph{monad} for $E$ and is determined by $E$ up to the action of a group. 
This leads to a description of the moduli space of instantons as a GIT quotient.
This description turns out to be too complicated for large $n$ to provide answers to the natural questions about the moduli space of instantons: smoothness, irreducibility and rationality.
In the end, some of these questions were solved in recent works of Verbitsky, Jardim, Markushevitch and Tikhomirov by a different technique (see \cite{jardim2011trihyperkahler}, \cite{markushevich2010rationality}, \cite{tikhomirov2012moduli}).

On the other hand, for small values of $n$ one can give a more explicit description of the moduli space of instantons (see \cite{ellingsrud1981stable}) by analyzing the loci of jumping lines.

\begin{definition*}
A line $L$ is jumping for an instanton $E$ if $E$ restricted to $L$ is not trivial. 
We will say that a line $L$ is $k$-jumping for $E$ if $E$ restricted to $L$ is isomorphic to $\CO_L(-k) \oplus \CO_L(k)$.
\end{definition*}

By Grauert--M\"ulich theorem the locus of jumping lines for an instanton $E$ is a hypersurface $S_E$ of degree $n$ in the Grassmannian $\Gr(2,4)$ parameterizing lines in $\PP^3$.
It comes with a rank one torsion free sheaf $\CL_E$ and one can show that an instanton $E$ can be reconstructed from the pair $(S_E,\CL_E)$.
However, it turns out to be quite hard to characterize which pairs $(S_E,\CL_E)$ give rise to instantons, at least for higher values of the charge. 
One of the reasons for this is the high dimension of $S_E$ and its singularities.

Recently, the definition of instanton bundle was extended from $\PP^3$ to other Fano threefolds of Picard number $1$.

\begin{definition*}[\cite{faenzi2011even}]
A rank $2$ vector bundle $E$ with $c_1(E) = 0$ or $1$ on a Fano threefold $X$ is an instanton if there is a twist $E \otimes \CO_X(t) = F$ such that
\[
F \cong F^* \otimes \omega_X, \qquad H^1(X,F) = 0
\]
\end{definition*}

Faenzi proved some general results about instantons on Fano threefolds, such as the existence of a good component of arbitrary charge (see \cite{faenzi2011even}).
On the other hand, Kuznetsov investigated in greater detail the case of Fano threefolds of index 2.
In particular, in \cite{kuznetsov2012instanton} the locus of jumping lines $D_E$ was defined together with a theta-sheaf $\CL_E$ on it. 
Moreover, motivated by the Grauert--M\"ulich theorem, the following conjecture was suggested.

\begin{conjectureintro}[\cite{kuznetsov2012instanton}]\label{conjecture-jumping-lines}
For an instanton $E$ on a Fano threefold of index $2$ the generic line is not jumping.
An instanton $E$ can be reconstructed from its theta-sheaf $\CL_E$. 
\end{conjectureintro}

In the case of Fano threefolds of index $2$ and degree $4$ or $5$, \cite{kuznetsov2012instanton} provides more concrete descriptions of the moduli spaces of instantons.

Namely, in the case of the degree $4$ threefold the moduli spaces of instantons were shown to be related to the moduli spaces of self dual vector bundles on a curve of genus~2. 
Using this, the first part of conjecture \ref{conjecture-jumping-lines} was reinterpreted and the second part was proved.

The degree 5 case is that of a triple linear section of $\Gr(2,5)$.
\begin{definition*}
The variety $Y$ is the transverse intersection of the Pl\"ucker embedding $\Gr(2,5) \subset \PP^9$ with a $\PP^6$.
\end{definition*}
It turns out that $Y$ is similar to $\PP^3$ in many respects.
For example, both admit an action of $\SSL_2$ with an open orbit and both admit a full exceptional collection for their derived category.

Again in \cite{kuznetsov2012instanton}, these similarities with $\PP^3$ were used to show that every instanton on $Y$ can be represented as the middle cohomology of the monad
\begin{equation*}
0 \to \CU^{\oplus n} \to \CO_{Y}^{\oplus 4n+2} \to {\CU^*}^{\oplus n} \to 0,
\end{equation*}
where $\CU$ is the restriction of the tautological bundle to $Y$, in analogy with \eqref{intro:monad for P3}.
As a corollary a GIT description of the moduli space of instantons via special nets of quadrics was found, a reinterpretation of the first part of conjecture \ref{conjecture-jumping-lines} was given and the second part was proved.

\bigskip

The goal of this thesis is an extension of the results of \cite{kuznetsov2012instanton} on moduli spaces
of instantons on $Y$ and a detailed study of moduli spaces of small charge,
namely charge~2 (minimal instantons) and charge 3.

We start with an exposition of various facts about the geometry of $Y$, some of which are known to experts.
The references for these facts are scattered in the literature and some of them do not explicitly refer to the $\SL_2$-structure, so that we prefer providing complete proofs for most of them.
In particular, we describe several useful exceptional collection in $D^b(Y)$, the bounded derived category of coherent sheaves on $Y$, show that the Hilbert scheme of lines on $Y$ identifies with $\PP^2$ and that the Hilbert scheme of conics on $Y$ identifies with $\PP^4$.
We also write down explicit $\SL_2$-equivariant resolutions for the structure sheaves of lines and conics, discuss the natural stratification of the Hilbert scheme of conics and describe the natural incidence correspondences on the products of these Hilbert schemes.
Finally, we apply the same methods to cubic curves in $Y$. 
In all these constructions we pay special attention to the induced action of $\SL_2$.

In section \ref{sec:More on the geometry of Y} we introduce the following construction of $Y$: we show that the blowup of the Veronese surface in $\PP^4$ identifies naturally with the projectivization of the vector bundle $\CU$ on $Y$.
This description turns out to be very useful later, when we discuss the geometry of hyperplane sections of $Y$ and their relation to the Grassmannian of planes in $\PP^4$.
Again, all the results of this section are $\SL_2$-equivariant.

In section \ref{sec:Instantons} we start discussing instantons on $Y$. We remind the necessary definitions
and constructions from \cite{faenzi2011even} and \cite{kuznetsov2012instanton}, in particular the monadic description and some results about jumping lines. 
We state here the first new result of the thesis, which will be of great use in the description of the moduli space of minimal instantons.

\begin{theorem*}[\ref{thm:gamma injective}, \ref{cor:no n-jumps}]
If $E$ is an instanton of charge $n$ then $E$ has no $k$-jumping lines for $k \ge n$.
\end{theorem*}

A direct consequence of theorem \ref{thm:gamma injective} is the following proposition providing an explicit description of instantonic theta-characteristics of degree 2.

\begin{proposition*}[\ref{cor:generic line minimal}]
If $E$ is an instanton of charge $2$ then the curve of jumping lines $D_E$ is a smooth conic and the sheaf $\CL_E$ is isomorphic to $\CO_{D_E}(-1)$. 
In particular,  conjecture~\ref{conjecture-jumping-lines} holds true for instantons of charge $2$.
\end{proposition*}

We also introduce the notion of a jumping conic for an instanton. 
Of course one can just say that a conic $C$ is jumping for $E$ if the restriction of $E$ to $C$ is nontrivial, but it is not clear how to encode nontriviality of $E$ on $C$ in a sheaf on the Hilbert space of conics.
We show that a reducible conic $C = L_1 \cup L_2$ (resp.\ nonreduced conic $C = 2L$) is a jumping conic for an instanton $E$ 
if and only if the restriction of $E$ to either of the lines $L_1$ and $L_2$ (resp.\ to $L$) is nontrivial.
Moreover, we show that

\begin{proposition*}[\ref{prop:equivalence of jumping}]
If $E$ is an instanton and $C$ is a conic then $C$ is a jumping conic if and only if
$H^i(C,E\otimes \restr{\CU}{C}) \ne 0$ for some $i$.
\end{proposition*}

By means of proposition \ref{prop:equivalence of jumping}, we also show that an analogue of the Grauert--M\"ulich theorem holds for conics.

\begin{theorem*}[\ref{cor:generic conic is not jumping}, \ref{prop:degree of jumping conics}]
For an instanton $E$ of charge $n$ the scheme of jumping conics of $E$ is a divisor of degree $n$ in $\PP^4$
which comes with a natural sheaf of rank $2$. In particular, the generic conic is non-jumping.
\end{theorem*}

So, in some sense, conics behave better than lines, but as the dimension of the space of conics is higher,
the information we get is harder to use.

In section \ref{sec:Minimal instantons} we discuss in detail the moduli space of minimal instantons. 
This moduli space is also known in the literature as the moduli space of aCM bundles (see \cite{brambilla2009moduli}). 
We show that the moduli space $\CMI_2$ identifies with an explicit open subset of $\PP^5$.

\begin{theorem*}[\ref{minimalmoduli}]
The moduli space $\CMI_2$ of charge $2$ instantons on $Y$ can be represented as
\begin{equation*}
\CMI_2 \cong \PP(S^2\CC^3) \setminus ( \Delta \cup H ),
\end{equation*}
where $\Delta$ is the symmetric determinantal cubic and $H$ is the unique $\SL_2$-invariant hyperplane.
\end{theorem*}

A consequence of the description of theorem \ref{minimalmoduli} is the following.

\begin{corollary*}[\ref{cor:SL_2 equivariant instanton}]
There is a unique instanton $E_0$ of charge $2$ on $Y$ with an $\SL_2$-equivariant structure.
\end{corollary*}

In section \ref{sec:Instantons of charge 3} we consider the moduli space of instantons of charge 3. We construct a natural map
$\beta:\CMI_3 \to \Bg := \Gr(3,5)$ and show that its flattening stratification on $\Bg$ consists
of only two strata: the closed stratum $\Bs \subset \Bg$ identifies with the image of 
the projective plane parameterizing lines on $Y$ under a special embedding $\kappa:\PP^2 \to B$
constructed in section 3, and the open stratum $\Bn := \Bg \setminus \Bs$ is its complement.

After constructing $\beta$, we also construct a natural compactification $\CMIbar_3$ of $\CMI_3$ and describe it.

\begin{theorem*}[\ref{thm:compactify MI_3}]
There is a natural compactification $\CMIbar_3$ of $\CMI_3$ such that the map $\beta:\CMI_3 \to B$
extends to a regular map $\bar\beta:\CMIbar_3 \to B$ which is a $\PP^3$-fibration over $B^{\sfn}$ and
a $\Gr(2,5)$-fibration over $B^{\sfs}$.
\end{theorem*}

We show that instantons $E$ such that $\beta(E) \in B^{\sfs}$ form a divisorial family and have many special properties
(because of this we call them \emph{special}). 
One of the characterizing properties is the following

\begin{proposition*}[\ref{prop:2-jump again}]
An instanton $E$ is special if and only if it has a $2$-jumping line $L$. 
For a special instanton $E$ there is a unique $2$-jumping line $L$ and $\kappa(L) = \beta(E)$.
\end{proposition*}

In section \ref{sec:An example}, the results of section 6 are used to show how singular the theta-sheaf $\CL_E$ of jumping lines can be for an instanton $E$ of charge $3$.
While in the case of minimal instantons the theta-characteristics admit the simple description of corollary \ref{cor:generic line minimal}, for $c_2(E) = 3$ we prove the following propositions

\begin{proposition*}[\ref{prop:non locally free theta}]
The theta-sheaf of $E$ is not locally free if and only if $E$ is special.
\end{proposition*}

\begin{proposition*}[\ref{prop:reducible theta}]
If the plane $\beta(E) \subset \PP^4$ is a generic tangent plane to the Veronese surface in $\PP^4$, then the theta-sheaf of $E$ is supported on a reducible curve.
\end{proposition*}

It turns out that the compactified moduli space $\CMIbar_3$ is singular on a closed subset of codimension $3$ which is contained in $\CMIbar_3 \setminus \CMI_3$ and in $\beta^{-1}(\Bs)$. 

The last section is devoted to the construction and the description of an explicit desingularization 
\[
\Wpi: \WCMI_3 \to \CMIbar_3
\] 
We construct $\WCMI_3$ by blowing up the (non-Cartier) divisor of special instantons inside $\CMIbar_3$ and we identify $\WCMI_3$ with a $\PP^3$-fibration over the blowup $\widetilde{B}$ of $B$ in $B^{\sfs}$.

\begin{theorem*}[\ref{thm:resolution of singularity}]
There is a small contraction
\[
\Wpi_4: \WCMI_3 \to \CMIbar_3
\]
which is a resolution of the singularity of $\CMIbar_3$.
Moreover, the pushforward of the structure sheaf of $\WCMI_3$ is the structure sheaf of $\CMIbar_3$.
\end{theorem*}

\section*{Acknowledgements}
I would like to thank my PhD advisor A.~Kuznetsov for introducing me to the topic of this thesis, and for teaching me the techniques to play with it.
I am extremely grateful both to him and to my co-advisor U.~Bruzzo for patiently providing me with clear answers to many questions.


\section{Some preliminary facts}\label{sec:Some preliminary facts}

We work over $\CC$. All schemes are assumed to be Noetherian and separated.
In this section we set the notation which we will use for derived categories and we state a few more results (Beilinson spectral sequence, Grauert--M\"ulich theorem, Hoppe's criterion) to which we will refer in the rest of the thesis.

\subsection{Semi-orthogonal decompositions}

In any triangulated category $\CT$ the shift is denoted by $[1]$. 
For any two objects $F, G$ of $\CT$, the space $\Hom(F,G[p])$ is denoted by $\Ext^p(F,G)$ and the cohomology of $\RHom(F,G)$, that is to say $\bigoplus_p \Ext^p(F,G)[-p]$, is denoted by $\Ext^\bullet(F,G)$.

We work with triangulated categories $\CT$ which are $\Ext$-finite, that is to say that for any pair of objects $F,G$ the graded vector space $\Ext^\bullet(F,G)$ is finite dimensional. 
In other words, $\Hom(F,G)$ is always finite dimensional and $\Hom(F,G[i])$ vanishes for almost all integers $i \in \ZZ$.

We usually think of $\Ext^\bullet(F,G)$ as a non-symmetric scalar product with values in the category of graded finite dimensional vector spaces. 
When a subcategory $\CA \subset \CT$ is given, the subcategory
\[
\CA^\perp = \left\{T \in \CT \text{ s.t. } \forall A \in \CA \quad \Ext^\bullet(A,T) = 0 \right\}
\]
is called the right orthogonal of $\CA$ in $\CT$. Respectively, the subcategory 
\[
\Lperp \CA = \left\{T \in \CT \text{ s.t. } \forall A \in \CA \quad \Ext^\bullet(T,A) = 0 \right\}
\]
is called the left orthogonal of $\CA$ in $\CT$.

\begin{definition}
A full triangulated subcategory $\CA$ of $\CT$ is admissible if the inclusion functor has a left and a right adjoint.
\end{definition}
If the inclusion functor of $\CA$ has only a left or right adjoint, the category $\CA$ is called respectively left or right admissible.
By \cite[sec. 2.2/2.3]{kuznetsov2007homological} (which summarizes results and definitions from \cite{bondal1990representation} \cite{bondal1990representable} \cite{bondal2003generators}), in the case $\CT$ is the derived category of a smooth variety, the following theorem holds.
\begin{theorem}
[\cite{bondal1990representable}]
For $X$ smooth projective variety over a field $k$, any left or right admissible subcategory of $\D^b(X)$ is admissible.
\end{theorem}

\begin{definition}\label{def:semi-orthogonal decomposition}
A semiorthogonal decomposition of a triangulated category $\CT$ is an ordered set of full triangulated subcategories $\langle \CA_1, \ldots, \CA_n \rangle$ such that the following two conditions hold:
\begin{itemize}
\item For all $i >j$ and all $A_i \in \CA_i$ and $A_j \in \CA_j$
\[
\Hom(A_i, A_j) = 0
\]
\item
For any object $T \in \CT$ there exist objects $T_i$ and maps $f_i$
\[
0 = T_n \xrightarrow{f_n} T_{n-1} \xrightarrow{f_{n-1}} \ldots \xrightarrow{f_2} T_1 \xrightarrow{f_1} T_0 = T
\]
such that the cones of $f_i: T_i \to T_{i-1}$ lie in $\CA_{i}$.
\end{itemize}
\end{definition}

A sequence of maps as in definition \ref{def:semi-orthogonal decomposition} is usually called a filtration for $T$.
\begin{definition}
Given an object $T$ in $\CT$, a filtration for $T$ is a set of objects $T_i$ and maps $f_i$ such that
\begin{equation}\label{eq:filtration}
0 \cong T_m \xrightarrow{f_m} T_{m-1} \xrightarrow{f_{m-1}} \ldots \xrightarrow{f_{n+2}} T_{n+1} \xrightarrow{f_{n+1}} T_n \cong T
\end{equation}
\end{definition}
If we denote by $L^i[-i]$ the cone of $f_{i+1}$, there is an induced sequence in the opposite direction
\begin{equation}\label{eq:subquotients}
0 \to L^n \xrightarrow{g^n} L^{n+1} \xrightarrow{g^{n+1}} \ldots \xrightarrow{g^{m-2}} L^{m-1} \to 0
\end{equation}
such that $g^{i+1} \circ g^{i} = 0$.
Sequences \eqref{eq:subquotients} and \eqref{eq:filtration} fit in a diagram
\begin{equation}\label{eq:Postnikov}
\begin{diagram}
0 	& \rTo^{f_m} 	&T_{m-1} 	&\rTo^{f_{m-1}} &\ldots	&\rTo^{f_{n+2}} 	&T_{n+1} 	&\rTo^{\;\; \quad f_{n+1}\quad \;\;}	&  T \\
	&\luDashto(1,2) \star \ldTo(1,2) & 	& \luDashto(1,2) \star \ldTo(1,2)	&	&\luDashto(1,2) \star \ldTo(1,2) & 	& \luDashto(1,2) \star \ldTo(1,2) \\
	& L^{m-1}[-m+1]& \lDashto_{g^{m-2}}	& L^{m-2}[-m+2]	&\quad \dots \quad	& L^{n+1}[-n-1] & \lDashto_{g^n} & L^{n}[-n] 
\end{diagram}
\end{equation}
where all bottom triangles are commutative and those marked with $\star$ are distinguished. 
Dashed arrows denote maps of degree one.
The diagram \eqref{eq:Postnikov} is usually called Postnikov tower (or Postnikov system). 
One also usually says that $T$ is the convolution of the Postnikov tower \eqref{eq:Postnikov}.

The following is a simple and classical fact, but it is fundamental.
\begin{theorem}[\cite{bondal1990representation}] \label{thm:functorial decomposition}
The Postnikov tower induced by a semi-orthogonal decomposition is functorial.
\end{theorem}
\begin{proof}
By definition of semi-orthogonal decomposition it is enough to prove the theorem for a decomposition $\langle \CA_1, \CA_2 \rangle$ and then proceed by induction.

Choose a filtration for each object $T$ in $\CT$.
Given a map $\varphi: S \to T$ we will show that it extends uniquely to a map of distinguished triangles
\[
\begin{diagram}
S_1 					& \rTo	& S				& \rTo	& L				& \rTo	&S_1[1] 				\\
\dDashto_{\varphi_1}		&		& \dTo_{\varphi}	& 		& \dDashto_{\psi} 	&  		&\dDashto_{\varphi_1[1]}	\\
T_1					& \rTo	& T				& \rTo	& M				& \rTo	&T_1[1]
\end{diagram}
\]

The fact that $\psi$ commutes with both $\varphi$ and $\varphi_1[1]$ comes from the axiom of triangulated category which says that there is at least one such $\psi$.

Functoriality follows from uniqueness of $\varphi_1$ and $\psi$.
\end{proof}

An object $E \in \CA$ is called exceptional if $\Ext^\bullet(E,E) = \CC$. 
A set of exceptional objects $E_1, \ldots, E_n$ is called an exceptional collection if $\Ext^k(E_i,E_j) = 0$ for $i > j$. 
If moreover $\Ext^k(E_i,E_j) = 0$ whenever $k \neq 0$ the exceptional collection is called a strong exceptional collection. 

A way to induce semiorthogonal decompositions is to look for exceptional objects in $\CT$, as the following theorem shows. 
Note that, when included in a semiorthogonal decomposition, $E_i$ stands for the category generated by $E_i$.
\begin{theorem}
Given an exceptional collection $E_1, \ldots, E_n$ in $\CT$, there are induced semiorthogonal decompositions
\[
\langle E_1, \ldots, E_n, \Lperp \langle E_1, \ldots E_n \rangle \rangle
\]
and
\[
\langle \langle E_1, \ldots E_n \rangle^\perp,  E_1, \ldots, E_n \rangle
\]
\end{theorem}

If the left or the right orthogonal of $\langle E_1, \ldots, E_n \rangle$ (i.e. the minimal triangulated subcategory containing $E_1, \ldots E_n$) is $0$,
the collection is called a full exceptional collection. In this case both the left and the right orthogonal to $\langle E_1, \ldots, E_n \rangle$ vanish.

Given a semiorthogonal decomposition, one can obtain other semiorthogonal decompositions via functors which are called left and right mutations. 

Let $j: \CA \to \CT$ be an admissible subcategory, so that $\CT = \langle \CA^\perp, \CA \rangle$ and $\langle \CA, \Lperp \CA \rangle$ are semiorthogonal decompositions.
Denote the left and right adjoint functors to $j$ by $j^*$ and $j^!$.
Denote the embedding functors of $\Lperp \CA$ and $\CA^\perp$ in $\CT$ by $l$ and $r$, and their (respectively right and left) adjoint functors by $l^!$ and $r^*$.
It is easy to check that for $T \in \CT$ we have 
\[
l^!(T) = \Cone(T \to j j^*T)[-1]
\]
and
\[
r^*(T) = \Cone(j j^! T \to T)
\]

\begin{definition}\label{def:mutation}
The functor $l l^!: \CT \to \CT$ is called right mutation through $\CA$ and is denoted by $\RR_{\CA}$.
The functor $r r^*[-1]: \CT \to \CT$ is called left mutation through $\CA$ and is denoted by $\LL_{\CA}$.
\end{definition}
To spell out definition \ref{def:mutation}, we point out that the left and right mutation functors are characterized by the fact that there are distinguished triangles
\[
\begin{diagram}
\LL_{\CA}(T) 		& \rTo	& j j^! T 	& \rTo	& T 		& \rTo	& \LL_{\CA}(T)[1]
\end{diagram}
\]
and 
\[
\begin{diagram}
\RR_{\CA}(T)[-1]	& \rTo 	&T		& \rTo	& j j^* T 	& \rTo	& \RR_{\CA}(T)	
\end{diagram}
\]
The functors $\RR_{\CA}$ and $\LL_{\CA}$ return zero on $\CA$, but they are isomorphisms when restricted to the (respectively right and left) orthogonals. 
More precisely
\[
l^! r = l^! \RR_{\CA} r: \CA^\perp \to \Lperp \CA
\]
and
\[
r^* l = r^* \LL_{\CA} l: \Lperp \CA \to \CA^\perp
\]
are equivalences of triangulated categories as it is easy to check that the functor $l^! \cdot r \cdot r^* \cdot l$ is equivalent to the identity of $\Lperp \CA$ and that the functor $r^* \cdot l \cdot l^! \cdot r$ is equivalent to the identity of $\CA^\perp$.

A special case of the above setting is when $\CA$ is $D^b(k)$, with $k$ a field.
In this case $j$ sends $k$ to an exceptional object $E$, and the functors $j^*$ and $j^!$ are respectively $\Ext^\bullet( - ,E)^*$ and $\Ext^\bullet(E, - )$.
In this case the unit and the counit of the adjunction are respectively the evaluation and coevaluation maps for $E$, i.e. the universal
\[
\Ext^\bullet(E,T) \otimes E \to T
\]
and
\[
T \to \Ext^\bullet(T,E)^* \otimes E
\]

Now, let $\CT = \langle \CA_1, \ldots , \CA_n \rangle$ be a semiorthogonal decomposition into admissible subcategories.
Denote the collection $(\CA_1, \ldots, \CA_n)$ by $\FA$; we are going to construct two collections which we will denote by $\RR_i \FA$ and $\LL_i \FA$.

\begin{definition}[Lemma 2.3 \cite{kuznetsov2006homological}] \label{def:mutated decomposition}
Given $\FA = (\CA_1, \ldots, \CA_n)$ semiorthogonal decomposition, assume that each $\CA_i$ is admissible. Then for each $i \in [1,n-1]$ there is a semiorthogonal decomposition
\[
(\RR_i \FA)_j =
\left\{
\begin{matrix*}[l]
\CA_j 	& \mathrm{if \quad} j \neq i, i+1	\\
\CA_{i+1} 	& \mathrm{if  \quad} j = i		\\
\RR_{\CA_{i+1}}\CA_i = \Lperp \langle \CA_1, \ldots, \CA_{i-1}, \CA_{i+1} \rangle \cap \langle \CA_{i+2}, \ldots, \CA_n \rangle^\perp & \mathrm{if  \quad} j = i+1.
\end{matrix*}
\right.
\]
called the $i$-th right mutation of $\FA$ and a semi-orthogonal decomposition
\[
(\LL_i \FA)_j =
\left\{
\begin{matrix*}[l]
\CA_j & \mathrm{if \quad} j \neq i, i+1 	\\
\LL_{\CA_i} \CA_{i+1} = \Lperp \langle \CA_1, \ldots, \CA_{i-1} \rangle \cap \langle \CA_{i}, \CA_{i+2}, \ldots, \CA_{n} \rangle^\perp & \mathrm{if \quad} j = i \\ 
\CA_i & \mathrm{if \quad} j = i+1
\end{matrix*}
\right.
\]
called the $i$-th left mutation of $\FA$.
\end{definition}

We are now going to introduce two semiorthogonal decompositions related to $\FA$: they are called the left dual and the right dual to $\FA$ due to the vanishing of the entries of the matrices $\RHom(\Lvee \CA_i, \CA_j)$ and $\RHom(\CA_i, \CA_j^\vee)$ such that $i+j \neq n+1$.

\begin{definition}\label{def:dual decomposition}
The right dual decomposition of $\FA$ is
\[
\FA^\vee = (\LL_1 \ldots \LL_{n-1}) \circ \ldots \circ (\LL_1 \LL_2) \circ (\LL_1) (\FA)
\]
The left dual decomposition of $\FA$ is 
\[
\Lvee \FA = (\RR_{n-1} \ldots \RR_{1})\circ \ldots \circ (\RR_{n-1} \RR_{n-2})\circ (\RR_{n-1}) (\FA)
\]
\end{definition}


In the case the categories $\CA_i$ are generated by an exceptional element, one can refine the idea of dual decomposition to that of dual collection.

\begin{definition} \label{def:dual collection}
Let $\CT = \langle E_1, \ldots, E_n \rangle$ be a full exceptional collection for $\CT$.
Set
\[
E_i^{\vee} = \LL_{E_1} \LL_{E_2} \cdots \LL_{E_{n-1}} E_{n-i+1}
\]
and
\[ 
\Lvee E_i = \RR_{E_n} \RR_{E_{n-1}} \cdots \RR_{E_{n-i+2}} E_{n-i+1}
\]
The full exceptional collections $\langle E_n^{\vee}, \ldots, E_1^\vee \rangle$ and $\langle \Lvee E_n, \ldots, \Lvee E_1 \rangle$ are called respectively right dual collection and left dual collection.
\end{definition}

With the above choices, the right and left dual collections are characterized by the following property (see \cite[sec.2.6]{gorodentsev2004helix}).
\begin{equation}\label{eq:dual characterization}
\Ext^k(\Lvee E_i, E_j) = \Ext^k(E_i, E_j^\vee) = \left\{ 
\begin{matrix*}[l]
\CC & \textrm{\quad if $i+j = n+1$ and $i = k + 1$} \\
0 & \textrm{\quad otherwise}
\end{matrix*}
\right.
\end{equation}



\subsection{Beilinson spectral sequence}

Given a filtration for an object $A$ in $\D^b(X)$, one can construct a spectral sequence analogous to the one induced by a filtered complex. 
Recall that the filtration induces a Postnikov tower whose convolution is $A$. In the notation of diagram \eqref{eq:Postnikov}, we have the following proposition. 

\begin{proposition}[e.g. \cite{gelfand2003methods} ex. IV.2.2]\label{prop:spectral sequence}
Given a filtration $T_i, f_i$ for an object $A$ in $\D^b(X)$, 
there is a natural spectral sequence 
\[
E_1^{p,q} = H^q(L^p) 
\]
with differential $d^p = H^q(g^p)$ and converging to $H^{p+q}(A)$.
\end{proposition}
\begin{proof} 
Replace the $L^i$ to reduce to the case where $f_i$ are actual morphisms of complexes, then use the double complex spectral sequence.
\end{proof}
As we will often use proposition \ref{prop:spectral sequence}, let us clarify in the following remarks what we mean by natural and what the direction of differentials in the spectral sequence is.
\begin{remark}
By naturality of the induced spectral sequence we mean that given two Postnikov towers whose convolution is $A_1$ and $A_2$, and given a map between the two towers (i.e. a collection of maps between the objects in the tower commuting with all maps in the tower), then there is an induced map of spectral sequences.
Moreover, such maps of spectral sequences depend functorially on the maps between Postnikov towers.
\end{remark}
\begin{remark}
The spectral sequence can be thought of as the spectral sequence induced by a bicomplex whose columns are $L^i$ and whose horizontal arrows are $g^i$.
Entries $E_1^{p,q}$ with $p \notin [n, m-1]$ vanish.
\[
\begin{diagram}
H^{j+1}(L^n)	& \rTo^{H^{j+1}(g^n)}	& H^{j+1}(L^{n+1}) 	& \rTo_{H^{j+1}(g^{n+1})}	& \ldots	&  \rTo_{H^{j+1}(g^{m-2})}	& H^{j+1}(L^{m-1})	\\
	& \rdTo(4,2)^{d_2}	\\
H^{j}(L^n)	&\rTo^{H^j(g^n)}	&H^j(L^{n+1}) 	& \rTo_{H^j(g^{n+1})}	& \ldots	& \rTo_{H^j(g^{m-2})}	& H^j(L^{m-1})	\\
	& \rdTo(4,2)^{d_2}	\\
H^{j-1}(L^n)	& \rTo^{H^{j-1}(g^n)}	& H^{j-1}(L^{n+1}) 	& \rTo_{H^{j-1}(g^{n+1})}	& \ldots	&  \rTo_{H^{j-1}(g^{m-2})}	& H^{j-1}(L^{m-1})	\\	
\end{diagram}
\]
It follows that the last non trivial differential is at most at page $E_{n - m +1}$.
\end{remark}

A common situation where a filtration for an object in $\D^b(X)$ is available is when there is a semi-orthogonal decomposition of $\D^b(X)$.
The archetype of such a situation is $X = \PP^n$: its derived category admits several full exceptional collections, the most notable being $\langle \CO_{\PP^n}(-n), \ldots , \CO_{\PP^n} \rangle$. 
Together with its dual collection, it gives rise to a resolution of the diagonal inside $\PP^n \times \PP^n$.

The point in having a resolution for the diagonal of a scheme $X$ is that one can always rewrite the identity functor of $\D^b(X)$ as the Fourier--Mukai transform with kernel the structure sheaf of the diagonal.
As on $\PP^n$ there is such a resolution, one can take any object $A$ in $\D^b(\PP^n)$, pull it back to $\PP^n \times \PP^n$, tensor it by the resolution of the diagonal, push it forward to the other copy of $\PP^n$ and finally recover the initial $A$.
The Grothendieck spectral sequence is in this case a spectral sequence converging to $H^\bullet (A)$, which is known as the Beilinson spectral sequence.
\begin{theorem}[\cite{okonek1980vector}]
For any $A$ in $D^b(\PP^n)$ there are spectral sequences
\[
\prescript{'}{}E_1^{pq} = H^q(A \otimes \Omega^{-p}(-p)) \otimes \CO(p)
\]
and
\[
\prescript{''}{}E_1^{pq} = H^q\left(A(p)\right) \otimes \Omega^{-p}(-p)
\]
converging to $H^{p+q} (A)$.
\end{theorem}

The existence of a Beilinson-type spectral sequence on a variety $X$ is a more general fact, depending only on the choice of a full exceptional collection on $X$.
As we will need to decompose families of sheaves with respect to a full exceptional collection, we will state a relative version of the Beilinson spectral sequence.


Assume a scheme $X$ 
has a full exceptional collection $E_1, \ldots, E_n$. Then there is an induced left dual collection $\Lvee E_1, \ldots , \Lvee E_n$ and a full exceptional collection for $X \times X$
\[
\D^b(X \times X) = \langle E_i \boxtimes (\Lvee E_j)^* \rangle_{i =1 \ldots n, j = 1 \ldots n}
\]
The filtration of the diagonal of $X$ with respect to the above exceptional collection takes a very simple form.
\begin{theorem} [\cite{kuznetsov2009hochschild}]\label{thm:filtration of the diagonal}
The structure sheaf of the diagonal $\Delta_*\CO_X$ has a filtration whose $i$-th shifted successive cone $L^i$ is $E_{i+1} \boxtimes (\Lvee E_{n-i})^*[n - 1]$. 
\end{theorem}

\begin{remark}
The statement in \cite[Prop. 3.8]{kuznetsov2009hochschild} only says that the subquotients belong to the subcategory generated by $E_{i+1} \boxtimes \Lvee (E_{n-i})^*$.
To check that the cones are the ones in the statement of theorem \ref{thm:filtration of the diagonal}, use the filtration of the diagonal to induce a filtration of $E_i$, and afterwards use the spectral sequence of proposition \ref{prop:spectral sequence} on the filtration for $E_i$.
\end{remark}
One of the consequences of theorem \ref{thm:filtration of the diagonal} is the following result, which can be seen as an analogue of the Beilinson spectral sequence in a category where a full exceptional decomposition is available.
\begin{theorem}\label{thm:Beilinson}
Let $X$ be a smooth projective variety and let $D^b(X)$ be generated by an exceptional sequence of pure objects $E_1, \ldots, E_n$. Then for any $A$ in $D^b(X)$ there is a spectral sequence
\[
E_1^{p,q} = Ext^{n+q-1}(^{\vee}E_{n-p}, A) \otimes E_{p+1} 
\]
converging to $H^{p+q}(A)$.
\end{theorem}
The statement and its proof can be found both in \cite[sec 2.7.3]{gorodentsev2004helix} and in \cite[cor. 3.3.2]{rudakov1990helices}. A more general form can be found in \cite[thm 2.1.14]{bohning2006derived}.

With theorem \ref{thm:Beilinson} in mind, we write a spectral sequence for a family of objects in $\D^b(X)$ parametrized by a scheme $S$.
First, choose a notation for the projections between products of $X$ and $S$
\[
\begin{diagram}
X \times X \times S	& \rTo^{\quad p_2 \quad }	& X \times S	\\
\dTo^{p_1}		& 					& \dTo_{q}	\\
X \times S			& \rTo_{q}				& S
\end{diagram}
\]
Next, construct the diagonal of $X \times X \times S$ relative to $S$, namely $\Delta_X \times S$. 
Finally, by pulling back the filtration for the diagonal of $X \times X$ to $X \times X \times S$, we obtain a filtration for $\Delta_X \times S$ whose successive cones are $E_i \boxtimes (\Lvee E_{n-i+1})^* \boxtimes \CO_S$.

\begin{theorem}\label{thm:relative Beilinson}
Let $X$ be a projective variety having a full exceptional collection. Let $S$ be any scheme. For any $\CF \in \D^b(X \times S)$ there is a spectral sequence
\begin{equation}\label{eq:relative Beilinson}
E_1^{p,q} = E_{p+1} \boxtimes R^{n+q-1} q_{*}\CRHom_{X \times S}(\Lvee E_{n-p} \boxtimes \CO_S, \CF)
\end{equation}
converging to $H^{p+q}(\CF)$ which is functorial in $\CF$.
\end{theorem}
\begin{proof}
The main step in the proof is rewriting the identity functor of $\D^b(X \times S)$ as the composition of the following functors: first pullback via $p_2$, then tensor by the structure sheaf of the relative diagonal $\Delta_X \times S$, finally pushforward via $p_1$.
This is the same as saying that the identity of $\D^b(X \times S)$ is the Fourier--Mukai transform with kernel the structure sheaf of $\Delta_X \times S$.

Pullback the filtration for $\CO_{\Delta_X}$ of theorem \ref{thm:filtration of the diagonal} to get a filtration of the relative diagonal $\CO_{\Delta_X \times S}$ with successive cones $E_{i+1} \boxtimes  (\Lvee E_{n-i})^*\boxtimes \CO_S $. 
Then, for any $\CF \in \D^b(X \times S)$ we can tensor the filtration of the relative diagonal by $p_2^* \CF$ to get a filtration of $p_2^* \CF \otimes \CO_{\Delta_X \times S}$.

Next, push the filtration for $p_2^* \CF \otimes \CO_{\Delta_X \times S}$ forward via $p_1$. 
By the first paragraph of this proof, the result is a filtration for $\CF$.
Finally, the spectral sequence \eqref{eq:relative Beilinson} is induced by using \ref{prop:spectral sequence} for the filtration of $\CF$ which we have just obtained.

The functoriality in $\CF$ follows from the fact that the Postnikov tower \eqref{eq:Postnikov} associated with a semiorthogonal decomposition is functorial in $\CF$ by \ref{thm:functorial decomposition} and from the fact that the spectral sequence associated with a Postnikov tower depends functorially on the Postnikov tower by \ref{prop:spectral sequence}.
\end{proof}

\subsection{Flatness and families of sheaves}

We will often need to check flatness of families of sheaves. Most of the times such sheaves will not be coherent over the base of the family, so that it is useful to state the following classical local criterion for flatness \cite[thm. 6.8]{eisenbud1995commutative}, where $M$ is not assumed to be finitely generated over $R$.

\begin{theorem}[{\cite[thm 6.8]{eisenbud1995commutative}}]\label{thm:flatness eisenbud}
Suppose $(R,\mathfrak{m})$ is a local noetherian ring and let $(S, \mathfrak{n})$ be a local noetherian $R$-algebra such that $\Fm S \subset \Fn$. If $M$ is a finitely generated $S$-module, then $M$ is flat over $R$ if and only if $\Tor_1^R(R/\Fm, M) = 0$.
\end{theorem}

We will use the following 
generalization of criterion \ref{thm:flatness eisenbud}. Let 
\[
f: X \to Y
\] 
be a finite type flat map of noetherian schemes. Let
\[
\begin{diagram}
X_y	& \rTo^{j}		& X		\\
\dTo	& \square		& \dTo_{f}	\\
y	& \rTo		& Y
\end{diagram}
\]
be a cartesian diagram where $y$ is a closed point in $Y$.

\begin{corollary}\label{cor:qcoh flat}
Let $\CF \in \D^-_{\mathrm{coh}}(X)$, fix $m \in \ZZ$ and assume that for any $y \to Y$ the pullback $Lj^*\CF$ is a sheaf shifted by $m$. 
Then $\CF$ is the shift by $m$ of a coherent sheaf on $X$ which is flat over $Y$. 
\end{corollary}
\begin{proof}
The statement is invariant under shift, so we assume that $m = 0$. 
It is local on the stalk both in $X$ and $Y$ because taking the cohomology of a complex of sheaves and flatness are so.

We want to check that $H^i(\CF) = 0$ for $i \neq 0$: this can be done locally around closed points of $X$ and $Y$.
We also want to check that $H^0(\CF)$ is flat over $Y$: by flatness of $f$, this can be checked locally around closed points of $X$. As $f$ is of finite type, it sends closed points of $X$ into closed points of $Y$, so that one can localize at closed points of $Y$ and the $\Fm S \subset \Fn$ assumption in theorem \ref{thm:flatness eisenbud} holds.




In this way the statement reduces to a statement about $(S, \Fn)$ local ring of a closed point of $X$, $(R, \Fm)$ local ring of a closed point of $Y$ and $f$ a map between them sending the closed point of $X$ into the closed point of $Y$. 
The complex of coherent sheaves $\CF$ becomes a complex of modules $\CF$ which are finitely generated over $S$.

There is a spectral sequence whose entries are $L^q j^*(H^p(\CF))$ and converging to $L^{p+q}j^*(\CF)$. 
Note that $H^{>0}(\CF) = 0$, as otherwise the top nonvanishing cohomology cannot be killed by any of the next differentials and would contribute non-trivially to $L^{>0}j^*(\CF)$, which we assumed is zero.
As a consequence, the spectral sequence has a page
\begin{equation}\label{eq:qcoh flat 2}
\begin{diagram}
\ldots& \quad		&H^{-1}(\CF) \otimes_{S} S/\Fm S 	& \qquad		& H^{0}(\CF) \otimes_{S} S/\Fm S & & 0\\
	& \quad		&								& \luTo(2,4)	&	&	\\
\ldots& \quad		&\Tor_1^S(H^{-1}(\CF), S/ \Fm S)	& \qquad		& \Tor_1^S(H^{0}(\CF), S/ \Fm S) & & 0 \\
	& \quad		&								& \qquad		&	&	\\
\ldots& \quad		&\Tor_2^S(H^{-1}(\CF), S/ \Fm S)	& \qquad		& \Tor_2^S(H^{0}(\CF), S/ \Fm S) & & 0 \\
\end{diagram}
\end{equation}
where we have substituted $H^{>0}(\CF) = 0$.

As we assumed that only $L^0j^{*}(\CF)$ is non zero, $\Tor_1^S(H^{0}(\CF), S/ \Fm S) = 0$. As $S$ is flat over $R$, there is an isomorphism
\begin{equation}\label{eq:qcoh flat}
\Tor_i^S(H^{0}(\CF), S/ \Fm S) = \Tor_i^R(H^{0}(\CF), R/ \Fm R)
\end{equation}
so that by \ref{thm:flatness eisenbud} for $i=1$, $H^{0}(\CF)$ is flat over $R$.

Using \eqref{eq:qcoh flat} in the other direction, we get that $\Tor_i^S(H^{0}(\CF), S/ \Fm S) = 0$ for all $i > 0$.
This in turn implies by spectral sequence \eqref{eq:qcoh flat 2} that $H^{-1}(\CF) \otimes_{S} S/\Fm S = 0$. 
As $\CF$ has coherent cohomology, we can use Nakayama's lemma and obtain $H^{-1}(\CF) = 0$.

Finally, note that if $H^{i}(\CF) = 0$ for some $i \in [i_0, -1]$, by spectral sequence \eqref{eq:qcoh flat 2} also $H^{i_0-1}(\CF) = 0$, so that by induction we obtain $\CF \cong H^0(\CF)$.
\end{proof}


\subsection{Stability of vector bundles} \label{sec:stability}
Throughout section \ref{sec:stability}, $X$ will be an integral locally factorial scheme whose Picard group is $\ZZ$. 
Moreover we always require that the unique ample generator $\CO_X(1)$ has global sections.
As there is a unique ample generator of $\Pic(X)$, we will talk about the degree of a line bundle, meaning the degree with respect to $\CO_X(1)$.

Let us recall the definition of normalization of a vector bundle $E$.
\begin{definition}
The normalization $E_{\norm}$ of a vector bundle $E$ of rank $r$ is its unique twist such that $-r < c_1(E_{\norm}(k))  \leq 0$. 
\end{definition}
Equivalently, one can also say that $E_{\norm}$ is the unique twist of $E$ with slope
\[
\mu(E_{\norm}) = \frac{\deg c_1(E)}{\rank(E)}
\]
in the interval $(-1,0]$. 
The degree of $c_1(E)$ is taken with respect to the unique ample generator of $\Pic(X)$.

The following criterion for Mumford stability goes under the name of Hoppe's criterion. It is proved in \cite{hoppe1984generischer} for vector bundles on projective spaces, but the same proof works more generally.

\begin{theorem}[Hoppe's criterion, \cite{hoppe1984generischer}] \label{thm:hoppe}
Assume $\CO_X(1)$ has global sections and let $E$ be a vector bundle of rank $r$ on $X$ such that for each $1 \leq s \leq r-1$ the vector bundle $\left( \Lambda^s E\right)_{\textrm{norm}}$ has no global sections. Then $E$ is $\mu$-stable.
\end{theorem}
\begin{proof}
Assume $E$ is not Mumford stable, i.e. that there is a sheaf $F$ of rank $s < r$ and a sequence
\[
0 \to F \to E \to Q \to 0
\]
with $\mu(F) \leq \mu(E)$ and $Q$ pure.
As $X$ is integral, $E$ and $Q$ are torsion free, 
 so that $F$ is reflexive \cite[prop 1.1]{hartshorne1980stable}.

As $X$ is locally factorial, we can define $\Lambda^s F$ by removing the non-locally free locus of $F$ (which has at least codimension 2), taking the determinant on the remaining open subset of $X$ and finally extending the line bundle back to $X$.

As a consequence, we have a section
\[
\CO_X \to \Lambda^s E \otimes \CO_X(-c_1(F))
\]
As $\CO_X(1)$ has sections and $X$ is integral, such a section induces non zero sections of $\Lambda^s E \otimes \CO_X(-c_1(F) + k)$ for any $k \geq 0$.

Now we check that
\[
\left( \Lambda^s E \right)_{\textrm{norm}} \cong \Lambda^s E \otimes \CO_X(-c_1(F) + k)
\]
for some $k \geq 0$.
To prove it, it is enough to check that 
\[
\mu \left(\Lambda^s E \otimes \CO_X(-c_1(F))\right) \leq 0
\]
The left hand side of the above inequality is
\begin{equation} \label{eq:hoppe}
s (\mu(E) - \mu(F))
\end{equation}
As $F$ destabilizes $E$, \eqref{eq:hoppe} is not positive, so that we have constructed sections for a twist of $E$ with non-positive slope. 
As $\CO_X(1)$ has sections, also the normalized $s$-th exterior power of $E$ has sections.
\end{proof}

Again, let $X$ be an integral locally factorial scheme whose Picard is generated by an ample line bundle $\CO_X(1)$. 
Under these assumptions, the following lemma shows the equivalence, in the case of rank 2 bundles, of Gieseker--Maruyama stability and Mumford--Takemoto stability.

\begin{lemma}\label{lemma:mumford vs Gieseker}
A vector bundle $E$ of rank $2$ on $X$ is $\mu$-stable if and only if it is Gieseker stable.
\end{lemma}
\begin{proof}
Assuming that $E$ is Gieseker stable, we prove it is $\mu$-stable.
It is enough to prove the statement for normalized rank $2$ bundles, so we replace $E$ by $E_\norm$. As $E$ has rank 2, if it is not $\mu$-stable there is a reflexive rank 1 sheaf which destabilizes it.
Assume $E$ has a section
\[
0 \to \CO_X \to E \to Q \to 0
\]
If the section vanishes on a codimension 1 subset, as the scheme is locally factorial it vanishes on a Cartier divisor. As $\CO_X(1)$ generates the Picard group, we can factor the section as
\[
0 \to \CO_X(n) \to E
\]
for some positive $n$, which is impossible by $\mu$-semistability of $E$.

It follows that we have an exact sequence
\[
0 \to \CO_X \to E \to Q^{**} \to F \to 0
\]
with $F$ supported in codimension at least $2$, showing that $c_1(Q) = 0$.
Finally, the reduced Hilbert polynomial of $E$ satisfies
\[
p_E(n) = p_{\CO_X}(n) - P_F(n)/2
\]
which contradicts the Gieseker stability of $E$.

The other direction is standard.
\end{proof}
\begin{remark}
As explained in \cite[sec.3]{hartshorne1980stable}, we have actually proved that the only strictly $\mu$-semistable rank $2$ vector bundles with $\mu = 0$ are extensions of $\CO_X$ by itself.
As $H^1(\CO_X)$ is the tangent space to $\Pic(X)$ at the identity and as we assumed that $\Pic(X) = \ZZ$, the only extension of $\CO_X$ by itself is $\CO_X \oplus \CO_X$, which is therefore the only strictly $\mu$-semistable rank $2$ vector bundle with $\mu = 0$ on $X$.
\end{remark}

\subsection{The Grauert--M\"ulich theorem}
When a scheme $X$ contains enough rational curves, one can try to recover a stable vector bundle $E$ on $X$ by looking at its behaviour on such curves. 
As all vector bundles on $\PP^1$ split into a direct sum of line bundles, there is the following standard definition (e.g. \cite{okonek1980vector}).
\begin{definition}\label{def:splitting type}
The splitting type of a vector bundle $E$ on a rational curve $L \cong \PP^1$ is the unique sequence of integers $a_1 \leq \ldots \leq a_r$ such that
\[
\restr{E}{L} \cong \calO_L(a_1) \oplus \ldots \oplus \calO_L(a_r)
\]
In this case, we will also write
\[
\ST_{L}(E) \cong (a_1, \ldots, a_r)
\]
\end{definition}

The classical Grauert--M\"ulich theorem \cite{grauert1975vektorbundel} says that for a stable vector bundle on $\PP^n$ the generic splitting type has no gaps. 
The proof mainly relies on the irreducibility of the family of lines through a point. 

Among the generalizations of the Grauert--M\"ulich theorem, we are going to state and use the one proved by Hirschowitz in \cite{hirschowitz1981restriction}.
Let $X$ be a smooth projective variety of dimension $n$ with a fixed ample line bundle $\CO_X(1)$.
Let $Z \subset X \times S$ be a non-empty proper flat family of subvarieties of $X$ of dimension $m$. 
\begin{equation*}
\begin{diagram}
	&			& Z			&			&	\\
	& \ldTo^{p}	&\dInto		& \rdTo^{q}	&	\\
X	& \lTo		&X \times S	& \rTo		& S
\end{diagram}
\end{equation*}
Denote by $\CT_{Z/X}$ the relative tangent bundle of $Z$ over $X$.
Denote by $[Z_s]$ the class of the general fiber of $q$ in $H^{2n-2m}(X, \QQ)$, i.e. $p(q^{-1}(s))$ for a general point $s$ in $S$.

\begin{theorem}[\cite{hirschowitz1981restriction}]\label{thm:GM}
Assume the following conditions hold:
\begin{enumerate}
\item \label{itm:flat projection} $p$ is flat out of some set of codimension at least $2$ in $Z$.
\item \label{itm:irreducible fiber} the generic fiber of $p$ is irreducible.
\item \label{itm:proportionality} the classes $c_1(\CO_X(1))^{n-1}$ and $[Z_s].c_1(\CO_X(1))^{m-1}$ are proportional in 
$H^{2n-2}(X, \QQ)$.
\end{enumerate}
Let $E \neq 0$ be a torsion free sheaf on $X$ and $(\mu_1 \geq \ldots  \geq \mu_k)$ the slopes of the Harder--Narasimhan filtration of $p^*E$ restricted to the general fiber of $q$.

If $E$ is semistable, then 
\[
\mu_i - \mu_{i+1} \leq - \mu_{min}(\CT_{Z/X}) \qquad \text{for } i=1, \ldots, k-1
\]
where $\mu_{min}(\CT_{Z/X})$ is the minimal slope in the Harder--Narasimhan filtration of $\CT_{Z/X}$ relative to $S$. 
\end{theorem}

Note that in theorem \ref{thm:GM} the role of the splitting type is played by the slopes in the Harder-Narasimhan filtration of $E$ restricted to subvarieties of $X$.

\subsection{Koszul complexes}
In this section $X$ is a Cohen-Macaulay scheme. 
Given a section of a vector bundle $\CE^*$, the same latter will be used to denote it and to denote the dual map from $\CE$ to $\CO_X$.

Let $s$ be a section of a vector bundle $\CE^*$ on a scheme $X$. 
There is a complex
\begin{equation}\label{def:Koszul}
\det(\CE) \to \cdots \to \Lambda^2 \CE \to \CE \xrightarrow{s} \CO_X
\end{equation}
where all maps are contractions between exterior powers induced by the section $s$.
Complex \eqref{def:Koszul} is called the Koszul complex and is denoted by $\Kosz(\CE, s)$. 
\begin{remark}
The Koszul complex $\Kosz(\CE,s)$ is a DG-algebra. 
It follows that also its cohomology has the structure of a DG-algebra.
\end{remark}

When $X$ is Cohen-Macaulay it makes sense to talk about codimension, as all the local rings of $X$ and all associated points of $X$ have the same dimension. 
In this case, if the zero-locus $Z$ of $s$ has codimension equal to the rank $r$ of $\CE$, there is an exact sequence
\[
0 \to \det(\CE) \to \cdots \to \Lambda^2 \CE \to \CE \xrightarrow{s} \CO_X \to \CO_Z \to 0
\]
showing that $\Kosz(\CE, s)$ is quasi isomorphic to $\CO_Z$.


In the case the codimension of $Z$ is not the expected one, $\Kosz(\CE,s)$ is not anymore quasi isomorphic to a sheaf.
It is nonetheless possible sometimes to describe its cohomology in terms of exterior powers of an excess bundle related to the conormal bundle of $Z$ in $X$. Proposition \ref{prop:unexpected Koszul} shows how to do it.

Assume that $Z$ is a local complete intersection of codimension $c$ inside $X$, i.e. that $\dim(X) - \dim(Z) = c$ and that locally around any point of $Z$ the ideal $\CI_Z$ of $Z$ in $X$ is generated by $c$ elements. 
Then, as $X$ is Cohen-Macaulay, a minimal set of local generators for $\CI_Z$ is a regular sequence, so that the conormal sheaf $\CI_Z/\CI_Z^2$ is locally free (see for example \cite[A.6.1, A.7.1]{fulton1998intersection}).
We will denote it by $\CN^*_{Z/X}$. 
\begin{proposition} \label{prop:unexpected Koszul} Let $X$ be a Cohen-Macaulay scheme and $Z$ a local complete intersection of codimension $c$ in $X$.

There is a vector bundle $\CL_Z$ on $Z$ sitting in an exact sequence
\[
0 \to \CL_Z \to \CE_Z \to \CN_{Z/X}^* \to 0
\]
such that the cohomology of the Koszul complex $\Kosz(\CE,s)$ is isomorphic as a DG-algebra to the exterior algebra $\Lambda^\bullet \CL_Z$.

\end{proposition}
\begin{proof}
As the schematic zero-locus of $s$ is $Z$, the map $s$ factors as
\[
\begin{diagram}
\CE 	&	\rOnto	&  \CI_Z	& \rTo	& \CO_X
\end{diagram}
\]
Restricting to $Z$ surjectivity is preserved and $\CI_Z$ becomes the conormal bundle, so that we find an exact sequence 
\begin{equation}\label{eq:unexpected Koszul}
0 \to \CL_Z \to \CE_Z \to \CN_{Z/X}^* \to 0
\end{equation}
defining a sheaf $\CL_Z$.
As $Z$ is l.c.i. inside $X$, $\CN_{Z/X}^*$ is a vector bundle on $Z$, so that $\CL_Z$ is a vector bundle as well.

Trivialize $\CE$ on an open affine subset $U \subset X$.
Possibly after restricting to a smaller open subset (still denoted by $U$), we can split sequence \eqref{eq:unexpected Koszul} and lift the splitting to a commutative diagram
\begin{equation}\label{eq:lift splitting}
\begin{diagram}
\CL_{Z\cap U}		& \rTo	&\CE_{Z\cap U}	& \pile{\lTo \\ \rTo}	&\CN^*_{Z\cap U/U}	\\
\uOnto			& 		& \uOnto			&				& \uOnto			\\
\WCL_{Z\cap U}	& \rTo	&\CE_{U}			& \pile{\lTo \\ \rTo}	&\WCN^*_{Z\cap U/U}
\end{diagram}
\end{equation}
where both $\WCL_{Z \cap U}$ and $\WCN^*_{Z \cap U /U}$ are trivial bundles on $U$.
This is possible as the obstruction to all the liftings we want to find depends only on $H^1(\CI_{Z \cap U / U})$, which vanishes as $U$ is affine.

Denote the components of $s$ with respect to the splitting \eqref{eq:lift splitting} by $s_L$ and $s_N$.
In this notation the splitting \eqref{eq:lift splitting} induces an isomorphism of DG-algebras
\begin{equation*}
\Kosz\left(\CE_U, s \right) \cong \Kosz\left(\WCL_Z, s_L \right) \otimes \Kosz\left(\WCN^{*}_{Z/X}, s_N \right)
\end{equation*}
As $X$ is Cohen-Macaulay and $s_N$ is locally around any point of $Z$ a system of parameters for $Z$, 
the Koszul complex of $s_N$ is quasi isomorphic to $\CO_{Z \cap U}$. 
As $s$ and $s_N$ restrict to $0$ on $Z \cap U$, so does $s_L$.
It follows that there is a quasi-isomorphism of DG-algebras
\begin{equation} \label{eq:unexpected Koszul 2}
\Kosz\left(\CE_U, s \right) \cong \Kosz\left(\CL_{Z \cap U}, 0 \right) \cong \Lambda^\bullet \CL_{Z \cap U}
\end{equation}

Now we prove the global statement.
We will set 
\[
\CH_i := H^{-i}(\Kosz(\CE,s))
\]
First, we prove that $\CH_1 \cong \CL_Z$ by constructing a morphism and checking locally that it is an isomorphism.

Denote by $K$ the kernel of $s$, i.e.
\[
0 \to K \to \CE \xrightarrow{s} \CI_{Z} \to 0
\]
is exact.
The canonical projection from $K$ to $\restr{\CE}{Z}$ factors via $\CL_Z$.
Moreover, by definition of $K$ and of $\CL_Z$, the map from $K$ to $\CL_Z$ is surjective.
By the exact sequence
\[
\Lambda^2 \restr{\CE}{Z} \to \restr{K}{Z} \to \restr{\CH_1}{Z} \to 0
\]
and as the differentials of $\Kosz(\CE,s)$ restrict to 0 on $Z$, the surjective map from $\restr{K}{Z}$ to $\CL_Z$ factors surjectively via $\restr{\CH_1}{Z}$.
We have checked locally that the cohomology of $\Kosz(\CE, s)$ is supported on $Z$, so that $\restr{\CH_1}{Z} \cong \CH_1$.
We have therefore constructed a surjective map from $\CH_1$ to $\CL_Z$. Note that we know that locally both $\CH_1$ and $\CL_Z$ are the pushforward of vector bundles on $Z$ of rank equal to $\rank(\CE) - \codim(Z)$. It follows the map
\[
\CH_1 \to \CL_Z
\]
is an isomorphism as it is surjective.

Finally, we prove the isomorphism of the rest of the cohomology.
The diagram
\begin{equation} \label{eq:unexpected Koszul 3}
\begin{diagram}
\CH_1 \otimes \ldots \otimes \CH_1 	& \rTo 	& \CH_k 		\\
\uTo_{\cong}					&		& \uDashto	\\
\CL_Z \otimes \ldots \otimes \CL_Z	& \rTo	& \Lambda^k{\CL_Z}
\end{diagram}
\end{equation}
can be completed to a commutative square uniquely by the universal property of exterior powers, as the multiplication of the cohomology of a DG-algebra is graded commutative.
When we restrict to an arbitrary affine subscheme $U$, the diagram \eqref{eq:unexpected Koszul 3} still enjoys the property that there is a unique arrow from $\Lambda^k\restr{\CL_Z}{U}$ to $\restr{\CH_k}{U}$ making it commutative.

The quasi-isomorphism \eqref{eq:unexpected Koszul 2} is a map of DG-algebras, so that the induced map in cohomology is compatible with multiplication. 
As a consequence it coincides with the restriction of the dashed arrow of \eqref{eq:unexpected Koszul 3} to $U$, finally proving that
\[
\Lambda^\bullet{\CL_Z} \cong \CH_\bullet
\]
as DG-algebras on the whole $X$.
\end{proof}

\subsection{Base change}

Given a commutative square $\sigma$
\begin{equation}\label{diag:sigma}
\begin{diagram}
X'		& \rTo^{v}	& X		\\
\dTo^{g}	& \sigma	& \dTo_{f}	\\
Y' 		& \rTo_{u}	& Y
\end{diagram}
\end{equation}
of schemes, there is an induced natural transformation
\begin{equation}\label{eq:adjunction}
\BC_\sigma: Lu^* Rf_* \to Rg_* Lv^*
\end{equation}
called the base change map.
It can be defined as the composition
\[
Lu^* Rf_* \xrightarrow{Lu^* Rf_*(\eta_v)} Lu^* Rf_* Rv_* Lv^* \xrightarrow{\cong} Lu^* Ru_* Rg_* Lv^* 
\xrightarrow{\epsilon_u(Rg_* Lv^*)} Rg_* Lv^*
\]
where $\eta_v$ is the counit of the adjunction for $Rv_*$, $Lv^*$, while $\epsilon_u$ is the unit for the adjunction $Ru_*$, $Lu^*$.

By interchanging the roles of $f$ and $u$ (and of $g$ and $v$) we get another natural transformation
\[
\BC_\sigma ' : Lf^* Ru_* \to Rv_* Lg^*
\]
defined analogously.

We are interested in diagrams $\sigma$ such that $\BC_\sigma$ is an isomorphism.
Theorem \ref{thm:base change} describes such cartesian diagrams.
In order to state it, we will need the following definition.
\begin{definition}
A cartesian diagram $\sigma$ is $\Tor$-independent if for any pair of points $y', x$ such that $u(y') = f(x) = y$ one has
\[
\Tor_i^{\CO_{y,Y}}\left(  \CO_{x,X}, \CO_{y',Y'} \right) = 0 
\]
for all $i > 0$.
\end{definition}
\begin{remark} \label{rmk:flat is tor-independent}
Any cartesian diagram where $u$ or $f$ is flat is $\Tor$-independent.
\end{remark} 
\begin{theorem}[{\cite[thm 3.10.3] {lipman2009foundations}}] \label{thm:base change}
Assume a cartesian diagram $\sigma$ is given. Then the following properties are equivalent.
\begin{enumerate}
\item \label{itm:independent} The natural transformation $\BC_\sigma$ is a functorial isomorphism.
\item \label{itm:independent prime} The natural transformation $\BC_\sigma '$ is a functorial isomorphism.
\item \label{itm:Tor independent} The diagram $\sigma$ is $\Tor$-independent.
\end{enumerate} 
\end{theorem}

\begin{remark}\label{rmk:adjacent cartesian}
Note that given two adjacent cartesian squares $\sigma$ and $\tau$
\[
\begin{diagram}
X''		& \rTo	& X'		& \rTo	& X		\\
\dTo		& \sigma	& \dTo	&\tau		& \dTo	\\
Y'' 		& \rTo	& Y'		& \rTo	& Y	
\end{diagram}
\]
satisfying any of the equivalent conditions of theorem \ref{thm:base change}, it is easy to see that the outer square also satisfies them.
\end{remark}

Next, we provide a class of $\Tor$-independent squares.
\begin{lemma}\label{lemma:smooth expected dimension}
Assume $X,Y,Y'$ are smooth and irreducible and that the fiber product $X'$ has expected dimension. Then diagram \eqref{diag:sigma} is $\Tor$-independent.

\end{lemma}
\begin{proof}
First, we will reduce to the case where $f$ and $u$ are closed embeddings. 
Consider the diagram
\begin{equation}\label{diag:smooth expected dimension}
\begin{diagram}
X'		& \rInto^{(g,v)}			& Y' \times X		& \rTo^{p_X}	& X		\\
\dTo^{g}	& 					& \dTo_{Y' \times f}	&			& \dTo_{f}	\\
Y' 		& \rInto_{\Gamma_u}	& Y' \times Y		& \rTo_{p_Y}	& Y	
\end{diagram}
\end{equation}
where $\Gamma_u$ is the graph of $u$.
By remark \ref{rmk:flat is tor-independent} and as $p_Y$ is smooth, the square on the right is $\Tor$-independent.
It follows by remark \ref{rmk:adjacent cartesian} that in order to show that $\sigma$ is $\Tor$-independent it is enough to show that the left square in \eqref{diag:smooth expected dimension} is so.

Note that
\[
\begin{diagram}
X'		& \rInto^{(g,v)}			& Y' \times X			\\
\dTo^{g}	& 					& \dTo_{Y' \times f}		\\
Y' 		& \rInto_{\Gamma_u}	& Y' \times Y		
\end{diagram}
\]
has expected dimension if and only if
\[
\begin{diagram}
X'		& \rInto^{v}			&  X			\\
\dTo^{g}	& \sigma				& \dTo_{f}		\\
Y' 		& \rInto_{u}			&  Y		
\end{diagram}
\]
has expected dimension.

Repeating the construction with the graph of $f$ we reduce to the case of $f,u$ closed embeddings of smooth irreducible varieties with intersection of expected dimension.
We now proceed to the computation of
\begin{equation}\label{eq:smooth expected dimension 1}
\Tor_i^{\CO_{y,Y}}\left(  \CO_{x,X}, \CO_{y',Y'} \right)
\end{equation}
which is clearly local on the stalk of $X,Y,Y'$.

As $Y'$ is smooth and as $X$ is integral, 
the structure sheaf of $\CO_X$ has a Koszul resolution $\Kosz_{X/Y}$ by free $\CO_Y$-modules which can be used to compute \eqref{eq:smooth expected dimension 1}.
The tensor product
\[
\Kosz_{X/Y} \otimes \CO_{Y'}
\]
is itself a Koszul complex with $Y'$, which we denote by $\Kosz_{X'/Y'}$ as by definition of fiber product it has $0$-th cohomology isomorphic to $\CO_{X'}$.

Restrict around a point $x'$ of $X'$.
As the fiber product $X'$ has expected dimension
\[
\dim(Y) - \dim(X) = \dim(Y') - \dim(X')
\]
we have that the $\dim(Y) - \dim(X)$ generators of the Koszul complex are a system of parameters for $X'$ in $Y'$, which is smooth. 
It follows that $\Kosz_{X'/Y'}$ has non-trivial cohomology only in degree $0$, that is to say that all higher $\Tor$ groups \eqref{eq:smooth expected dimension 1} vanish. 
\end{proof}

\subsection{Families of subschemes of codimension at least $2$}

As we will work with families of curves inside threefolds, we will need lemma \ref{lemma:no ext^1 from codimension 2}. Note that $S$ is not assumed to be reduced.

Let $X$ be smooth and projective of dimension $d_X$. 
Let $Z,W$ be families of subschemes of $X$
\[
\begin{diagram}
Z	& \rInto^{\varphi}	& S \times X 	\\
	& \rdTo_{\mu_S}	& \dTo_{\pi_S} 	\\
	&				& S
\end{diagram}
\qquad \qquad \qquad
\begin{diagram}
W	& \rInto^{\psi}		& S \times X 	\\
	& \rdTo_{\nu_S}		& \dTo_{\pi_S} 	\\
	&				& S
\end{diagram}
\]
of dimension $d_Z, d_W$ and flat over $S$.
\begin{lemma}\label{lemma:no ext^1 from codimension 2}
Let $Z,W$ be flat families of subschemes of $X$. If the ideal sheaves $\CI, \CJ$ of $Z, W$ in $S \times X$ are isomorphic to each other as $\CO_{S \times X}$-modules, then $Z$ and $W$ are isomorphic families.
\end{lemma}
\begin{proof}
We begin by proving that given a diagram
\begin{equation}\label{diag:no ext^1 from codimension 2}
\begin{diagram}
0 	& \rTo	& \CI			& \rInto	& \CO_{S \times X} 	& \rTo	& \CO_{Z}				& \rTo	& 0	\\
	& 		& \dTo^{\cong}	&		& \dDashto^{t}		& 		& \dDashto^{\overline{t}}	&		&	\\
0 	& \rTo	& \CJ		& \rInto	& \CO_{S \times X}	& \rTo	& \CO_{W}			& \rTo	& 0	\\
\end{diagram}
\end{equation}
one can choose the dashed arrows to make it into a commutative diagram. Note that for this purpose it is enough to find $t$.
 
By long exact sequence, it is enough to show that
\[
\Ext^i(\CO_{Z}, \CO_{S \times X}) = 0 
\]
for $i = 0,1$.
For this to happen, it is enough that
\[
R^{\leq 1}\pi_{S*}\CRHom_{S \times X}\left(\varphi_* \CO_Z, \CO_{S \times X}\right) = 0 
\]
As $\varphi$ is a proper morphism, we can rewrite the above condition as
\begin{equation}\label{eq:no ext^1 from codimension 2}
R^{\leq 1}\mu_{S*}\CRHom_{Z}\left(\CO_Z, \varphi^! \CO_{S \times X}\right) = R^{\leq 1}\pi_{S*}\varphi_* \CRHom_{Z}\left(\CO_Z, \varphi^! \CO_{S \times X}\right) = 0 
\end{equation}
via Grothendieck duality.

As $X$ is smooth, $\pi_S^! \CO_S = \omega_X[d_X]$ is a shifted line bundle. As $X$ is smooth and $\mu_S$ is flat, $\varphi$ has finite $\Tor$-dimension, so that $\varphi^!$ commutes with tensor product (see \cite[Ex.\ 5.2]{neeman1996grothendieck}). 
As a consequence, we can rewrite condition \eqref{eq:no ext^1 from codimension 2} as
\[
R^{\leq 1 - d_X}\mu_{S*}\CRHom_{Z}\left(\varphi^*\omega_X, \mu_S^! \CO_{S}\right) = 0
\]
which by Grothendieck duality for the proper map $\mu_S$ is equivalent to
\[
R^{\leq 1 - d_X}\CHom_{Z}(R\mu_{S*}\left(\varphi^*\omega_X), \CO_{S}\right) = 0
\]
The hypercohomology spectral sequence yields for dimensional reasons that this last condition is true whenever $d_X - d_Z > 1$, which is one of our hypotheses.

Now that we have lifted the isomorphism between $\CI$ and $\CJ$ to a morphism $t$ in diagram \eqref{diag:no ext^1 from codimension 2}, we want to show that $t$ is an isomorphism. 
If this is true, then it induces an isomorphism $\overline{t}$ from $\CO_Z$ to $\CO_W$ and the lemma is proved.

By snake lemma, the kernel and the cokernel of 
\[
\overline{t}: \CO_Z \to \CO_W
\]
are isomorphic respectively to $\ker(t)$ and $\coker(t)$.
We want to check that there are no non-zero morphisms from $\CO_W$ to $\coker(t)$.

Assume on the contrary that there is a morphism
\begin{equation}\label{eq:ext^1 L to Y 1}
\CO_{W} \to \coker(t)
\end{equation}
which is non-zero. 
Then there is an associated point of $\coker(t)$ where it is non-zero. 
This contradicts the fact that all associated points of $W$ have height at least $2$,
while the associated primes of $\coker(t)$ have height at most $1$.

As $\coker(t) = 0$, also $\ker(t)$ has to vanish (as $\CO_{S \times X}$ is locally free), so that $t$ is an isomorphism.
\end{proof}

\subsection{Closed embeddings}

Given a proper map of schemes
\[
f : Y \to X
\]
we will often use the following criterion in order to determine whether $f$ is a closed embedding or not.

\begin{lemma}\label{lemma:closed embedding and pushforward}
Assume $f$ is proper of relative dimension $0$ and assume that the natural $\CO_X \to f_*\CO_Y$ is surjective.
Then $f$ is a closed embedding.
\end{lemma}
\begin{proof}
If $f$ is proper of relative dimension $0$, then it is affine by \cite[prop. 4.4.2]{grothendieck1967elements}. 
As $f$ is affine and as being a closed embedding is local in $Y$, we reduce to the case of
\[
f: \Spec (B) \to \Spec( A)
\]
In this case the natural $\CO_X \to f_*\CO_Y$ becomes $A \to B$.
As we assume that $\CO_X \to f_*\CO_Y$ is surjective, so is $A \to B$. 
It follows that $f$ is a closed embedding.
\end{proof}

\subsection{A classical fact about (possibly nonreduced) curves in $\PP^n$}

Let $Z \subset \PP^n$ be a closed subscheme such that the Hilbert polynomial $P_Z(t)$ of $\CO_Z$ is $d_Z t + c$.
Then $Z$ has dimension $1$.
We want to say that we can find a projection
\[
\pi_Z: \PP^n \to \PP^d
\]
which is an isomorphism out of a finite number of points of $Z$.
If we denote by $X$ the scheme-theoretic image of $Z$, then this is equivalent to the fact that the adjunction map
\[
\CO_X \to \pi_* \CO_Z
\]
extends to an exact sequence
\[
0 \to \CO_X \to \pi_{Z*} \CO_Z \to M \to 0
\]
where $M$ is supported in dimension 0.

The following proposition is analogous to the classical fact that any reduced curve is birational to an affine plane curve, but works for possibly non-reduced curves.
\begin{proposition}\label{prop:project curves}
Let $Z \subset \PP^n$ be a closed subscheme. Assume that at each generic point of $Z$ the tangent space has dimension less or equal than $d$, with $d > 1$. Then for the generic linear projection $\pi_Z$ to $\PP^d$ there is an exact sequence
\[
0 \to \CO_X \to \pi_{Z*} \CO_Z \to M \to 0
\]
where $M$ is supported in dimension 0.
\end{proposition}
\begin{remark}
If the Hilbert polynomial of $Z$ is $P_Z(t) = d_Zt + c$ with $d_Z \leq d$, then the assumptions of proposition \ref{prop:project curves} are satisfied. 
\end{remark}
\begin{proof}
As the statement we want to prove is up to finitely many points, we can assume that $Z$ has no components of dimension $0$, i.e. it is locally Cohen-Macaulay.
We will show that if $n>d$ then we can find a projection to $\PP^{n-1}$ as in the statement.
We will denote by $Z_i$ the components of $Z$, each of degree $d_i$.

First note that there for the general point $P$ in $\PP^n$ the projection from $P$ satisfies the following properties.
\begin{enumerate}
\item \label{itm:closed points} For each component $Z_i$ there is a point $Q_i$ such that the line through $P$ and $Q_i$ does not intersect $Z$ in any other point.
\item \label{itm:tangents} For each component $Z_i$ there is a point $Q'_i$ such that the line through $P$ and $Q'_i$ does not lie in the tangent space to $Z$ at $Q'_i$. 
\end{enumerate}
%
As for property \eqref{itm:closed points}, note that for any point $Q_i$ the set of lines through $Q_i$ which intersect $Z$ in another point forms a surface. 
As the components are finitely many and $n > d \geq 2$, property \eqref{itm:closed points} holds for the generic point in $\PP^n$. 

As property \eqref{itm:closed points} holds for any $Q_i \in Z_i$, we will look for $Q'_i$ among the $Q_i$ and we will suppress the prime in $Q'_i$.
Note that the tangent space $T_{Q_i}Z_i$ to $Z_i$ at a general point $Q_i$ has dimension at most $d_i$.
It follows that the general $d_i$-codimensional linear space $H_i \subset \PP^n$ containing $Q_i$ is transversal to $T_{Q_i}Z_i$ inside $T_{Q_i}\PP^n$, and that through the general point $P \in \PP^n$ there is such an $H_i$.
Any line contained in $H_i$ is not contained in $T_{Q_i}Z_i$ and each $H_i$ contains lines as $d_i < n$.
As a consequence, property \eqref{itm:tangents} holds for the general point $P \in \PP^n$.

Let $\pi$ be the linear projection from $P$ satisfying properties \eqref{itm:closed points} and \eqref{itm:tangents}. It restricts to a regular map $Z \to X$ which in affine coordinates in suitable neighborhoods of each $\pi(Q_i)$ becomes
\begin{equation}\label{eq:project curves}
A_X \to A_Z
\end{equation}
Note that \eqref{eq:project curves} is injective by definition of schematic image. 
Moreover, possibly after restricting further the affine chart in $X$, the map $Z \to X$ is injective on closed points by property \eqref{itm:closed points}. 
Finally, that \eqref{eq:project curves} is surjective can be checked on cotangent spaces at each closed point. 
By property \eqref{itm:tangents} we know that the map is injective on tangent spaces at each closed point, which is equivalent.
It follows that \eqref{eq:project curves} is an isomorphism.

As $\pi$ is affine, the map \eqref{eq:project curves} is the restriction of
\[
0 \to \CO_X \to \pi_{Z*} \CO_Z \to M \to 0
\]
around general points of each component of $X$.
It follows that $M$ is not supported on any of the components of $X$, so that it is supported on points.
\end{proof}

\begin{corollary}\label{cor:d=1,2,3 in P^n}
Let $Z \subset \PP^n$ be a closed subscheme with $P_Z(t) = d_Zt + c$.
Denote by $\delta$ the length of $M$ and by $\epsilon$ the sum of the lengths of the embedded points of $X$.
\begin{enumerate}
\item \label{itm:d=1} There is no $Z$ such that $d_Z =1$ and $c < 1$. If $c=1$, $Z$ is a line.
\item \label{itm:d=2} There is no $Z$ such that $d_Z = 2$ and $ c < 1$. If $c = 1$, $Z$ is a plane conic.
\item \label{itm:d=3} There is no $Z$ such that $d_Z = 3$ and $c < 0$. If $c = 0$, $Z$ is a plane cubic. If $c = 1$, the following are equivalent
\begin{itemize}
\item $Z$ contains a plane cubic.
\item $Z$ has an embedded point.
\item $h^1(\CO_Z) \neq 0$.
\end{itemize}
\end{enumerate}
\end{corollary}
\begin{proof}
Note that $\delta = 0$ if and only if $\pi$ is an isomorphism between $Z$ and $X$.

Case \eqref{itm:d=1} is clear as proposition \ref{prop:project curves} gives $P_Z = t + 1 + \epsilon + \delta$, so that $c \geq 1$ and equality holds if and only if $\delta = \epsilon = 0$.

Case \eqref{itm:d=2} is clear as proposition \ref{prop:project curves} gives $P_Z = %
2t + 1 + \epsilon + \delta$.
It follows that $c \geq 1$ and equality holds if and only if $\epsilon = \delta = 0$.

In case \eqref{itm:d=3} again by \ref{prop:project curves} one has $P_Z = P_X + \delta$ for $X \subset \PP^3$.
If the dimension of the tangent space to $X$ is greater than $2$ at each point, than $X$ is, up to finitely many points, a triple spatial line (i.e. cut by the square of the ideal of a line). In this case $P_Z = 3t + 1 + \epsilon + \delta$, so that $c \geq 1$ and equality holds if and only $\epsilon = \delta = 0$, i.e. if and only if both $X$ and $Z$ are triple spatial lines.

If the dimension of the tangent space to $X$ is generically less than $3$, project to $\PP^2$ and obtain $P_Z = 3t + \epsilon + \delta$, so that $c \geq 0$ with equality if and only if $Z$ is a plane cubic.


Finally, assume $c = 1$.
Note that by the above part of the corollary, $Z$ contains an embedded point if and only if it contains a plane elliptic curve.
If $Z$ contains a plane elliptic curve, then $h^1(\CO_Z) = 1$ by long exact sequence. 
If $h^1(\CO_Z) \neq 0$, then $h^0(\CO_Z) > 1$, so that there is a non-surjective non-zero section of $\CO_Z$, which implies that there is a non surjective 
\[
\varphi: \CO_Z \to \CO_Z
\]
with cokernel $\CO_{W_1}$. 
The ideal sheaf of $W_1$ in $Z$ is a quotient of $\CO_Z$, so that it is also the structure sheaf of a closed subscheme $W_2$ of $Z$.
By parts \eqref{itm:d=1} and \eqref{itm:d=2} and additivity of the Hilbert polynomial, the degree of $W_1$ cannot be $1$ or $2$.
It follows that either $W_1$ or $W_2$ has Hilbert polynomial $3t$ and is therefore an elliptic plane curve in $Z$.
\end{proof}

Let $Y_5$ be the Fano threefold introduced in the next section. 
As we are interested in $\Hilb^{dt + 1}(Y_5)$, we also give the following corollary.
\begin{corollary}\label{cor:d=1,2,3 in Y}
Let $d \in \{1,2,3\}$, then $\Hilb^{dt}(Y_5)$ is empty. Any curve $Z \subset Y_5$ representing a point in $\Hilb^{dt+1}(Y_5)$ has  no embedded points and satisfies $h^1(\CO_Z) = 0$.
\end{corollary}
\begin{proof}
As for $\Hilb^{dt}(Y_5)$, the only non-trivial case is $d= 3$. By part \eqref{itm:d=3} of \ref{cor:d=1,2,3 in P^n}, all subschemes of $\PP^6$ with Hilbert polynomial $3t$ are plane cubics. As $Y_5$ is cut by quadrics, if it contains a plane cubic it contains the whole plane which it spans, against Lefschetz hyperplane theorem.

As for $\Hilb^{dt+1}(Y_5)$, it follows directly from part \eqref{itm:d=3} of \ref{cor:d=1,2,3 in P^n} and from the fact that $Y_5$ contains no plane cubics, as we have just pointed out.
\end{proof}


\section{The setting} \label{sec:The setting}


Let $V$ be a complex vector space of dimension $5$. Let $A \subset \Lambda^2 V^*$ be a 3-dimensional space of $2$-forms on $V$, and let $Y_A$ denote the triple linear section of $\Gr(2,V)$ cut by $A$.


The following lemma 
is classical and its proof can be found for example in \cite[ch.II, thm 1.1]{iskovskikh1980anticanonical}, where one can find the full classification of Fano threefolds of index $2$.
The proof which we give here is ad hoc for triple linear sections of $\Gr(2,V)$, but it is simpler.
\begin{lemma}\label{lemma:HPD of Gr}
$Y_A$ is a smooth threefold if and only if all forms in $A$ have rank $4$. Moreover all smooth threefolds $Y_A$ are in the same $\GL(V)$-orbit.
\end{lemma}
\begin{proof}
For the first part, it is a classical fact that $\Gr(2,V^*) \subset \PP(\Lambda^2 V^*)$ is the projective dual of $\Gr(2,V) \subset \PP(\Lambda^2 V)$.
A very short proof is for example in \cite[thm 2.1]{tevelev2005projective} and is the following one.

The action of $GL(V)$ on $\PP(\Lambda^2V)$ and $\PP(\Lambda^2V^*)$ has only one closed orbit, i.e. respectively $\Gr(2,V)$ and $\Gr(2,V^*)$. As the projective dual is closed and preserved by the action of $GL(V)$, the two orbits are projectively dual to each other.

Once we know that $\Gr(2,V^*)$ is the projective dual of $\Gr(2,V)$, it is enough to note that $\Gr(2,V^*)$ is the locus of forms of rank $2$ and that $Y_A$ is a smooth threefold if and only if it has $3$-dimensional tangent space at any of its points, that is to say if and only if all linear sections in $A$ are smooth.

For the second part, note that $\PSL(V)$ has dimension $24$, while $\Gr(3,\Lambda^2V^*)$ has dimension $21$, so that for a general point we expect a $3$-dimensional stabilizer.
Denote by $\SSL_A(V)$ the stabilizer of $A$ under the action of $\SSL(V)$ on $\Gr(3,\Lambda^2V)$.
As the Grassmannian is irreducible, if there are points with a $3$-dimensional stabilizer, then they all belong to a unique open orbit under the action of $\PSL(V)$.
We will prove that, for any smooth $Y_A$, $\SSL_A(V)$ is $3$-dimensional. 
From this fact, it follows that all smooth $3$-dimensional $Y_A$ are projectively equivalent.

By its definition $\SSL_A(V)$ acts on $A$, so that there is a map
\begin{equation}\label{eq:HPD of Gr}
\SSL_A(V) \to \SSL(A)
\end{equation}
As all forms in $A$ have rank $4$, define a $\SSL_A(V)$-equivariant map
\[
A \xrightarrow{\Ver_2} S^2A \xrightarrow{S^2} S^2\Lambda^2V^* \xrightarrow{\wedge} \Lambda^4V^* \cong V
\]
The kernel of the projection from $S^2A$ to $V$ is generated by a non-degenerate $\SSL_A(V)$-invariant quadratic form on $A^*$, as one can show as follows.

Assume the kernel contains a degenerate quadratic form  $a_1 \cdot a_2 \in S^2A$, i.e. $a_1 \wedge a_2 = 0$. 
Fix a basis $v_1, \ldots, v_5$ for $V^*$ such that $a_1 = v_1 \wedge v_2 + v_3 \wedge v_4$ and denote the span of $v_1, \ldots, v_4$ by $V_4^*$.
It follows that 
\[
a_2 \in \Lambda^2 V_4^*
\]
As a consequence, the pencil of $2$-forms spanned by $a_1$ and $a_2$ inside $\PP(\Lambda^2 V_4^*)$ intersects $\Gr(2,V_4^*)$, against the assumption that all forms in $A$ have rank $4$.

Assume now that the kernel has dimension greater than $1$: then there would be a degenerate quadric in it, as the set of degenerate quadrics is a divisor in $\PP(S^2A)$.
We have therefore proved that the kernel of the projection from $S^2A$ to $V$ is generated by a non-degenerate quadric.

As a consequence, map \eqref{eq:HPD of Gr} factors as
\begin{equation}\label{eq:HPD to Gr}
\SSL_A(V) \to \SSL_q(A) \to \SSL(A)
\end{equation}
where $\SSL_q(A)$ is the subgroup of $\SSL(A)$ of transformations fixing a smooth conic, that is to say $\SSL(2, \CC)$.

Finally, note that, as the kernel of the $\SSL_A(V)$-equivariant projection from $S^2A$ to $V$ is non-degenerate, there is a well defined regular map from $\PP(A)$ to $\PP(V)$ factoring via $\PP(S^2A)$. 
The image of this map is not contained in any hyperplane of $\PP(V)$ as the image of $\Ver_2$ is not contained in any hyperplane of $\PP(S^2A)$.
As the connected components of the fixed locus of an automorphism of $\PP(V)$ are linear subspaces, an element of $\PSL_A(V)$ fixing all points in the image of $\PP(A)$ in $\PP(V)$ has to be the identity of $\PP(V)$.
It follows that the kernel of \eqref{eq:HPD to Gr} is at most discrete, 
so that $\SSL_A(V)$ is at most $3$-dimensional.
\end{proof}

Lemma \ref{lemma:HPD of Gr}, together with the fact that the generic triple hyperplane section of $\Gr(2,V)$ is smooth by Bertini, justifies the following definition.
\begin{definition}
$Y$ is the unique (up to projective equivalence) smooth threefold obtained as the intersection of $\Gr(2,V)$ and $\PP(A^\perp)$ inside $\PP(\Lambda^2V^*)$.
\end{definition}

While proving lemma \ref{lemma:HPD of Gr} we have introduced a map which plays an important role in the geometry of $Y$.
We point it out in the next definition, and explain in a remark that it is an embedding.
\begin{definition}\label{def:sigma}
The map $\sigma$ is the compostion
\[
\begin{diagram}
\PP(A) &  \rTo^{\Ver_2} 	& \PP(S^2A)	 & \rDashto	& \PP(V)
\end{diagram}
\]
where the first map is the Veronese embedding of degree 2 and the second one is projectivization of the composition
\[
\begin{diagram}
S^2A	& \rTo	& S^2 \Lambda^2 V^*	& \rTo^{\wedge} 	& \Lambda^4 V^* \cong V  
\end{diagram}
\]
\end{definition}
\begin{remark}\label{rmk:sigma embedding}
In \ref{lemma:HPD of Gr} we have proved that the kernel of the projection from $S^2A$ to $V$ is a non-degenerate form.
As the variety of bisecants to the image of $\Ver_2$ is the variety of degenerate quadratic forms, $\sigma$ is a closed embedding.
\end{remark}

The following lemma describes the (co-)homology of $Y$. 
\begin{lemma}\label{lemma:homology of Y}
Let $M$ an abelian group, then
\[
H^{\mathrm{odd}}(Y,M) = 0
\]
and
\[
H^0(Y, M) = H^2(Y,M) = H^4(Y,M) = H^6(Y,M) = M
\]
and the same holds for homology.
\end{lemma}
\begin{proof}
The statement is true for $H^i(Y, \ZZ)$ and $H_i(Y,\ZZ)$ when $i \neq 3$ by Lefschetz hyperplane theorem and by Poincar\'e duality over $\ZZ$. 
For $i = 3$, note that by Euler characteristic $H_i$ is torsion and by universal coefficients a non-trivial $H_i$ would give a torsion subgroup of $H^4(Y,\ZZ)$.

For arbitrary $M$ the statement follows from universal coefficients and from the result for $M = \ZZ$. 
\end{proof}

As a consequence of lemma \ref{lemma:homology of Y}, we will often refer to classes of the Chow group of $Y$ as to integers. For example, we will usually say that Chern classes on $Y$ are integers.
We will denote the generators of cohomology groups of $Y$ respectively by $H$ (the class of a hyperplane section), $L$ (the class of a line, which is Poincar\'e dual to $H$), $P$ (the class of a point).

On a Fano variety $X$, Kodaira vanishing implies
\[
H^{i}(X, \CO_X) = 0
\]
for $i > 0$.
It follows that the exponential sequence induces an isomorphism
\[
c_1 : \mathrm{Pic}(X) \to H^2(X,\ZZ)
\]
As a consequence,
$\Pic(Y)$ has a unique ample generator $\CO_Y(1)$ induced by the unique ample generator of $\Gr(2,V)$. 
As all ample line bundles are positive multiples of $\CO_Y(1)$, stability of a sheaf on $Y$ will always mean stability with respect to $\CO_Y(1)$.

\begin{remark}\label{rmk:hoppeY}
As $\CO_Y(1)$ comes from the ambient $\PP^6$ via pullback, it has global sections. Moreover, we have just shown that $\Pic(Y) = \ZZ$ and $Y$ is smooth by definition. Summing up, we have just said that $Y$ satisfies the hypotheses of theorem \ref{thm:hoppe} and lemma \ref{lemma:mumford vs Gieseker}.
\end{remark}


Vector bundles on $\Gr(2,V)$ induce vector bundles on $Y$ by restriction: we will denote the rank 2 tautological subbundle by $\calU$, so that we have the following exact sequence
\begin{equation} \label{eq:definingsequence}
0 \rightarrow \calU \rightarrow V \otimes \calO_Y \rightarrow V/\calU \rightarrow 0
\end{equation}
and its dual
\begin{equation} \label{eq:definingsequencedual}
0 \rightarrow \calU^\perp \rightarrow V^* \otimes \calO_Y \rightarrow \calU^* \rightarrow 0
\end{equation}

The decomposition of the cohomology of $Y$ is actually a shadow of a more general decomposition on $Y$, that of $\mathcal{D}^b(Y)$ into subcategories generated by exceptional vector bundles. More precisely, there is a semiorthogonal decomposition induced by a full exceptional collection (see \cite{orlov1991exceptional}) of the form 
\begin{equation}\label{eq:full exceptional}
\mathcal{D}^b(Y) = \langle \calU, \calU^\perp, \calO_Y, \calO_Y(1) \rangle
\end{equation}
We will see in proposition \ref{prop:full exceptional} that triangles \eqref{eq:definingsequence} and \eqref{eq:definingsequencedual} compute some of the mutations of the above collection, yielding other useful collections. 

We will need the Chern polynomials and characters of these bundles, so we write them in the following lemma for future reference.

\begin{lemma}
The following equalities hold in the cohomology of $Y$.
\begin{equation}\label{Cherntable}
\begin{aligned}
\ch(\calU) &= 2 - H + \frac{L}{2} + \frac{P}{6} \\
\ch(\calU^\perp) &=  3 - H - \frac{L}{2} + \frac{P}{6}  \\
\ch(S^2\CU^*) &= 3 + 3H + \frac{9L}{2} - \frac{P}{6} \\
\mathrm{c}(\calU) &= 1 - H + 2L \\
\mathrm{c}(\calU^\perp) &= 1 - H + 3L - P
\end{aligned}
\end{equation}
Moreover, $H^3 = 5P$ and $\omega_Y = \calO_Y(-2)$.
\end{lemma}
\begin{proof}
Once the Chern character of $\CU$ is known, all the other equalities follow from additivity of the Chern character or by multiplicativity of the Chern class with respect to exact sequences \eqref{eq:definingsequence} and \eqref{eq:definingsequencedual}.
On the ambient $\Gr(2,V)$ we have that $-c_1(\CU)$ is the ample generator of the Picard group, so that the same holds by Lefschetz on $Y$.
Moreover, on $\Gr(2,V)$ we can easily compute $c_1(\CU)^4 c_2(\CU) = 2 P_{\Gr}$, where $P_{\Gr}$ is the class of a point on $\Gr(2,V)$, so that, by functoriality of Chern classes, on $Y$ we have $c_1(\CU) c_2(\CU) = 2P$, which implies $c_2(\CU) = 2L$.

The fact that $H^3 = 5P$ follows by the computation of $H_{\Gr}^6$ on $\Gr(2,V)$. The fact that $\omega_Y = \CO_Y(-2)$ follows from adjunction formula and from $\omega_{\Gr} = \CO_{\Gr}(-5)$.
\end{proof}

The cohomology of the equivariant bundles in sequences \eqref{eq:definingsequence} and \eqref{eq:definingsequencedual} and the maps between them can be computed using Borel--Bott--Weil on the ambient $\Gr(2,V)$ and the resolution for the structure sheaf of $Y$ inside it.
\begin{equation}\label{eq:resolution of Y}
0 \to \CO_{\Gr}(-3) \to A^* \otimes \CO_{\Gr}(-2) \to A \otimes \CO_{\Gr}(-1) \to \CO_{\Gr} \to \CO_Y \to 0
\end{equation}
We will compute all of them in lemma \ref{lemma:cohomologies on Y}.

The natural bundles on $Y$ which we have just introduced are all $\mu$-stable. The proof is an easy application of Hoppe's criterion \ref{thm:hoppe} and is carried out in \ref{lemma:stability on Y}.

\subsection{The $SL_2$-action on Y} \label{sec:SL2action}

When working on $Y$, one should be aware that there is an action of $SL_2$ on $Y$ with an open orbit, as explained for example in \cite{mukai1983minimal}. The $SL_2$-action can be induced by choosing a vector space $W$ with $\dim W = 2$, and identifying $V$ with $S^4W$. Then, we get a canonical decomposition
\[
\Lambda^2V^* \cong \mathrm{S}^2W \oplus \mathrm{S}^6W
\]
and we can choose $A = S^2W$. With this choice $\Gr(2,V) \cap \PP(A^\perp)$ is smooth and the natural action of $SL_2(W)$ on $\Lambda^2V$ restricts to an action on $\Gr(2,V) \cap \PP(A^\perp)$, as it is defined as the intersection of two $SL_2(W)$ invariant varieties.
\begin{lemma}\label{lemma:SL_2-equivariant projection}
Let $Y_A = \Gr(2,V) \cap \PP(A^\perp)$ where $A \subset \Lambda^2V^*$ is $SL_2$-equivariant. Then $Y_A$ is smooth.
\end{lemma}
\begin{proof}
To see that the above $SL_2$-equivariant construction returns a smooth variety, we will show that all forms in $A$ have rank $4$ and then we will use lemma \ref{lemma:HPD of Gr}.

To show that all forms in $A$ have rank $4$, consider the $SL_2$-equivariant projection
\begin{equation}\label{eq:SL_2-equivariant projection}
S^2A\xrightarrow{S^2} S^2\Lambda^2V^* \xrightarrow{\wedge} \Lambda^4 V^* \cong V
\end{equation}
Note that a form in $A$ has rank $2$ if and only if it is in the kernel of \eqref{eq:SL_2-equivariant projection}.

By decomposing it into $SL_2$-irreducibles, we find
\[
S^4W \oplus \CC \to S^4W
\]
It follows that the kernel of \eqref{eq:SL_2-equivariant projection} is either a line or the whole $S^2A$.

If it is a line, then there is at most one degenerate form in $\PP(A)$, which gives an irreducible non-trivial $SL_2$-subrepresentation of $A$.

If it is the whole $S^2A$, then all forms in $A$ are degenerate, so that they share a common kernel, which gives a non-trivial $SL_2$-subrepresentation of $V$.

As $V$ and $A$ are by definition $SL_2$-irreducible, all forms in $A$ have rank $4$, so that $Y_A$ is smooth.
\end{proof}

\begin{remark}
We have proved in \ref{lemma:HPD of Gr} that for a smooth $\Gr(2,V) \cap \PP(A^\perp)$ the kernel of \eqref{eq:SL_2-equivariant projection} is a non-degenerate conic.
By lemma \ref{lemma:SL_2-equivariant projection}, this holds also for $Y$ constructed in the $SL_2$-equivariant way.

A more direct argument is the following: if the kernel of \eqref{eq:SL_2-equivariant projection} is degenerate as a symmetric form on $A^*$, then its kernel (as a form on $A^*$) is $SL_2$-invariant, against the fact that $A$ is $SL_2$-irreducible.
\end{remark}

\begin{remark} 
As $SL_2$ acts as a subgroup of $GL(V)$, sequences \eqref{eq:definingsequence}, \eqref{eq:definingsequencedual} and \eqref{eq:resolution of Y} are $SL_2$-equivariant.
\end{remark}

One can describe explicitly the vector space $A \subset \Lambda^2 S^4 W$.
Let us fix the notation for the $\SSL_2$ (and $\LieSL_2$) action.
Denote by $x,y$ a basis for $W$ and by $x^{n-i}y^i$ a basis for $S^n W$.
Then the linear operators
\[
X = x \partial_y \qquad Y = y \partial_x \qquad H = [X, Y] = x \partial_x - y \partial_y
\] 
are the standard basis for $\LieSL_2$.
In this notation $A$ is spanned by
\[
3x^3y \wedge x^2y^2 - x^4 \wedge xy^3, \quad 2x^3y \wedge x y^3 - x^4 \wedge y^4, \quad 3x^2y^2 \wedge xy^3 - x^3y \wedge y^4
\]
and $A^\perp$ is spanned by
\begin{gather*}
x^4 \wedge x^3y, \qquad x^4 \wedge x^2y^2, \qquad 2x^3y \wedge x^2y^2 + x^4\wedge xy^3, \qquad 8x^3y \wedge xy^3 + x^4 \wedge y^4, 	\\ 
 2x^2y^2 \wedge xy^3 + x^3y \wedge y^4, \qquad x^2y^2 \wedge y^4, \qquad xy^3 \wedge y^4
\end{gather*}
The form 
\begin{equation}\label{eq:generator of Y}
x^4 \wedge x^2y^2 + x^2y^2 \wedge y^4
\end{equation}
is decomposable and belongs to $A^\perp$. It follows that it belongs to $Y$.
Under the identification of $A^\perp$ with $S^6W$, the form \eqref{eq:generator of Y} corresponds to 
\[
xy(x^4 - y^4)
\]
Note that $xy(x^4 - y^4)$ has $6$ distinct roots and that its stabilizer under the action of $SL_2$ is therefore discrete (it is a discrete group of order $48$ known as the binary octahedral group).
The well-known proposition \ref{prop:orbit structure} is a specialization of the following lemma to the case $f = xy(x^4 - y^4)$.
\begin{lemma}[{\cite[lemma 1.5]{mukai1983minimal}}]\label{lemma:mukai umemura 1}
Let $f$ be a homogeneous polynomial of degree $n$ in two variables.
If the stabilizer of $[f] \in \PP(S^n W)$ is a finite group and if all the roots of $f = 0$ are distinct, then the closure of the $SL_2$-orbit generated by $[f]$ is decomposed into the disjoint union of $SL_2$-orbits
\[
\left( SL_2\cdot f \right) \cup \left( SL_2 \cdot x^5y \right) \cup \left( SL_2 \cdot x^6 \right)
\]
\end{lemma}
\begin{remark}\label{rmk:mukai umemura}
The $2$-dimensional orbit is the span of tangent lines to the rational normal curve of degree $6$. One can check it at the point $x^6$ and then argue by $SL_2$-equivariance.
\end{remark}
\begin{proposition}\label{prop:orbit structure}
The action of $SL_2$ on $Y$ has the following orbit decomposition:
\begin{itemize}
\item A unique open orbit, isomorphic to $\SSL_2/ \Gamma$, where $\Gamma$ is the binary octahedral group. 
\item A unique $2$-dimensional orbit, whose closure is linearly equivalent to $2H$. Moreover, it is the union of the tangent lines to the $1$-dimensional orbit.
\item A unique $1$-dimensional orbit, which is a rational normal curve of degree $6$.
\end{itemize}
\end{proposition}
\begin{proof}
Using lemma \ref{lemma:mukai umemura 1} and the fact that $xy(x^4 - y^4)$ has $6$ distinct roots one immediately gets the decomposition into orbits,  while remark \ref{rmk:mukai umemura} provides the description in terms of tangent lines.

That the 2-dimensional orbit is linearly equivalent to $2H$, it follows from the decomposition into $SL_2$-irreducible representations of $H^0(\CO_Y(2))$.
The reason is that the $SL_2$-equivariant Koszul complex \eqref{eq:resolution of Y} provides us with an exact sequence
\[
0 \to A^* \to A \otimes \Lambda^2 V^* \to H^0(\CO_Y(2)) \to 0
\]
which is $SL_2$ isomorphic to
\[
0 \to S^2W \to (\ldots) \oplus S^0W \to H^0 (\CO_Y(2)) \to 0
\]
and implies the existence of an $SL_2$-invariant divisor linearly equivalent to $2H$.
As there are no $SL_2$-invariant divisors linearly equivalent to $H$, the zero locus of the $SL_2$-invariant section of $\CO_Y(2)$ is the unique $2$-dimensional orbit. 
\end{proof}

Via Borel--Bott--Weil, Serre duality and the sequences \eqref{eq:definingsequence}, \eqref{eq:definingsequencedual} and \eqref{eq:resolution of Y} one can describe thoroughly the cohomology of the bundles on $Y$ which come from the ambient Grassmannian (see also \cite{orlov1991exceptional}).
\begin{lemma}\label{lemma:cohomologies on Y}
There are canonical $SL_2$-equivariant isomorphisms
\begin{itemize}
\item $\CU(1) = \CU^*$
\item $H^\bullet(\CU) = H^\bullet (\CU(-1)) = H^\bullet(\CU^\perp) = H^\bullet(\CU^\perp(-1)) = H^\bullet(\CO_Y(-1)) = 0$;
\item $H^\bullet(V/\CU (-1)) =  \Hom(\CU^\perp, V/\CU(-1)) = 0$;
\item $H^\bullet(V/\CU) = V$ and $H^\bullet(\CU^*) = V^*$;
\item $\Hom(\CU, \CU^\perp) = A$;
\item $H^\bullet(\CO_Y(1)) = \Lambda^2 V^* / A$.
\end{itemize}
Moreover, under the above identifications the composition pairing
\[
A \otimes V \cong \Hom(\CU,\CU^\perp) \otimes \Hom(\CU^\perp,\CO_Y) \to \Hom(\CU, \CO_Y) \cong V^*
\]
is the natural evaluation of a skew form in $A$ on a vector of $V$.
\end{lemma}
\begin{proof}
The isomorphism of $\CU(1) = \CU^*$ follows from $\rank(\CU) = 2$ and $\Lambda^2\CU = \CO(-1)$.
The vanishing of $H^\bullet(\CU)$, $H^\bullet(\CU(-1))$, $H^\bullet(\CU^\perp)$, $H^\bullet(\CU^\perp(-1))$ and $H^\bullet(\CO_Y(-1))$ comes from Borel--Bott--Weil and from the Koszul resolution \eqref{eq:resolution of Y}.
The vanishing of $H^\bullet(V/\CU (-1))$ and $\Hom(\CU^\perp, V/\CU(-1))$ follows from the above vanishings and from the tautological sequences \eqref{eq:definingsequence} and \eqref{eq:definingsequencedual}.

Again by Borel--Bott--Weil there is an $SL_2$-equivariant isomorphism 
\[
\Hom_Y(\CU, \CU^\perp) = A \otimes \Ext^1(\CU(1), \CU^\perp) = A
\]

The long exact sequences of \eqref{eq:definingsequence} and \eqref{eq:definingsequencedual} induce the other two isomorphisms, together with their $SL_2$-equivariance. In the same way, one checks that $H^\bullet(\CO_Y(1)) = \Lambda^2 V^* / A$.

To compute the composition, note that there is a unique non-zero $SL_2$-equivariant map from $A \otimes V$ to $V^*$ and that the natural pairing is $SL_2$-equivariant.
To understand why the pairing is non-trivial, note that it can be induced by the diagram
\begin{equation} \label{eq:HomUUperp}
\begin{diagram}
0	&\rTo	&A \otimes \CU			&\rTo	^{\qquad}	&A \otimes V \otimes \CO_Y	&\rTo	&A \otimes V/\CU		&\rTo	& 0	\\
	&	& \dDashto^{\ev_{\CU}}	&			& \dTo^{a(-,-)}				&	&\dDashto	^{\ev_{V/\CU}}	&	&	\\
0 	&\rTo	&\CU^\perp			&\rTo			&V^*\otimes\CO_Y			&\rTo	&\CU^*				&\rTo	& 0
\end{diagram}
\end{equation}
where it is always possible to draw the dashed arrow as by definition of $Y$ the pairing $a(u_1,u_2)$ vanishes at any point $[U]$ and for any $u_1,u_2 \in U$.
\end{proof}

In the next proposition we use lemma \ref{lemma:cohomologies on Y} to compute some mutations of the exceptional collection \eqref{eq:full exceptional}.

\begin{proposition} \label{prop:full exceptional}
The following equalities hold
\begin{equation} \label{eq:mutations}
\RR_{\CO_Y}(\CU) = V/\CU, \qquad \RR_{\CO_Y}(\CU^\perp) = \CU^*, \qquad \RR_{\CU}(V/\CU(-1)) = \CU^\perp.
\end{equation}
The collections
\begin{equation}\label{eq:coll 2}
\langle \calU, \calO_Y, \calU(1), \calO_Y(1) \rangle, \qquad \langle \CO_Y, V/ \CU, \CU^*, \CO_Y(1) \rangle
\end{equation}
are full and exceptional.

The collection
\begin{equation} \label{eq:coll}
\langle \CO_Y(-1), \CU, \CU^\perp, \CO_Y \rangle
\end{equation}
is full and exceptional. Its left dual collection is
\begin{equation} \label{eq:coll dual}
\langle \CO_Y, \CU(1), \CU^\perp(1), \CO_Y(1) \rangle
\end{equation}
\end{proposition}
\begin{proof}
By Borel--Bott--Weil and resolution \eqref{eq:resolution of Y}, $(V/\CU)^* = \CU^\perp$ 
 is acyclic on $Y$, that is to say that $V/\CU \in \Lperp \CO_Y$.
As a consequence, sequence \eqref{eq:definingsequence} is the mutation sequence of $\CU$ through $\CO_Y$, so that $\RR_{\CO_Y}(\CU) = V/\CU$.
The dual of a mutation triangle is again a mutation triangle, so that $\RR_{\CO_Y}(\CU^\perp) = \CU^*$.

As for $\RR_{\CU}(V/\CU(-1)) = \CU^\perp$, by lemma \ref{lemma:cohomologies on Y} there is an evaluation map
\[
V / \CU (-1) \to A \otimes \CU
\]
We are now going to show that the cone of the evaluation map is $\CU^\perp$.
Note that by Borel--Bott--Weil and Serre duality the cone lies in the right orthogonal to $\CO_Y$ and $\CO_Y(1)$. 
Moreover, by definition of the evaluation map, it lies in the left orthogonal to $\CU$, so that by fullness of the collection \eqref{eq:full exceptional} it is in $\langle \CU^\perp \rangle$.
It follows that the cone of the evaluation map is a direct sum of shifted copies of $\CU^\perp$.
As there are no maps from $\CU^\perp$ to $V/\CU(-1)$, the evaluation map is injective and its cokernel is $\CU^\perp$, that is to say that there is an exact sequence
\begin{equation}\label{eq:mutation U Uperp}
0 \to V/\CU(-1) \to A \otimes \CU \to \CU^\perp \to 0
\end{equation}
showing that $\RR_{\CU}(V/\CU(-1)) = \CU^\perp$.

As a consequence of \eqref{eq:mutations}, the collections in \eqref{eq:coll 2} are full and exceptional because they are mutations of \eqref{eq:full exceptional}.

The fact that \eqref{eq:coll} is exceptional follows from the more general fact that every time that there is a Serre functor $\BBS$ for a triangulated category $\CT$ and an admissible subcategory $\CA$ is given, there is an isomorphism 
\[
\BBS(\Lperp \CA) = \CA^\perp
\]
In our setting,
\[
\langle \CU, \CU^\perp, \CO_Y \rangle^\perp = \BBS  (\Lperp \langle \CU, \CU^\perp, \CO_Y \rangle)  = \BBS ( \langle \CO_Y(1) \rangle ) = \langle \CO_Y(-1) \rangle
\]

The left dual of collection \eqref{eq:coll} is characterized by equations \eqref{eq:dual characterization}, which hold for collections \eqref{eq:coll} and \eqref{eq:coll dual}, as one can easily check with the aid of lemma \ref{lemma:cohomologies on Y}.
\end{proof}

Via Hoppe's criterion \ref{thm:hoppe} and lemma \ref{lemma:cohomologies on Y} we can easily check that all tautological bundles we have introduced so far are $\mu$-stable.
\begin{lemma} \label{lemma:stability on Y}
The tautological vector bundles $\CU, \CU^*, \CU^\perp, V/ \CU$ are $\mu$-stable.
\end{lemma}
\begin{proof}
By remark \ref{rmk:hoppeY} we can use Hoppe's criterion \ref{thm:hoppe}. As $\rank (\CU) = 2$ and $c_1(\CU) = -1$, $\CU$ is a normalized bundle. To prove that $\CU$ is stable it is enough to check that $H^0 (\CU) = 0$, which is true by lemma \ref{lemma:cohomologies on Y}.

The bundle $\CU^\perp$ has rank $3$, so that there is an isomorphism $\Lambda ^2 \CU^\perp \cong \CO(-1) \otimes V/ \CU$. By Hoppe's criterion, it is enough to check $H^0(\CU^ \perp) \cong H^0(V/ \CU(-1)) = 0$, which holds true again by lemma \ref{lemma:cohomologies on Y}.

As for the other two bundles, it is enough to point out that $V/\CU = (\CU^\perp)^*$ and that the dual of a $\mu$-stable bundle is $\mu$-stable. 
\end{proof}

\subsection{Lines in $Y$} \label{sec:lines}

As $Y$ has a canonical choice of a very ample line bundle $\CO_Y(1)$, we will call line any subscheme $L$ with Hilbert polynomial $h_L(t) = 1 + t$.
It is well known that $L$ is an actual line linearly embedded in $\PP^n$, without any embedded points. 
We will need analogous statements for conics and cubics in $Y$: they are collected in corollary \ref{cor:d=1,2,3 in Y}.

The Fano scheme of lines in Y is canonically identified with $\PP(A)$. Before describing the identification, we characterize the restriction to any line of the tautological bundles coming from $\Gr(2,V)$. 

In order to do it, recall that any vector bundle on $\PP^1$ splits as a direct sum of line bundles and that we have already defined the splitting type of a vector bundle $E$ on a line $L$ in \ref{def:splitting type}.

\begin{lemma} \label{lemma:splitting on lines}
The splitting types of $\calU, V / \calU, \calU^\perp, \calU^*$ on any line $L$ are as follows:
\[
\begin{matrix*}[l]
\ST_L(\CU) = (-1,0) \\
\ST_L (V/\CU) = (0,0,1) \\
\ST_L (\CU^\perp) = (-1,0,0) \\
\ST_L (\CU^*) = (0,1)
\end{matrix*}
\]
\end{lemma}
\begin{proof}
Both $\calU$ and $\calU^{\perp}$ embed in a trivial bundle by sequence \eqref{eq:definingsequence}, so that their splitting type contains no $\calO_L(n)$ with $n > 0$. As $c_1(\calU) = c_1(\calU^\perp) = -1$ we have the splitting types in the thesis. 
The other two splitting types follow by dualization.
\end{proof}

Now we are going to describe the identification of the Fano scheme of lines in $Y$ with $\PP(Hom(\calU,\calU^\perp))$, which by \ref{lemma:cohomologies on Y} is canonically isomorphic to $\PP(A)$.

\begin{lemma}\label{zerolocusV}
The schematic zero locus of $v \in H^0(V/\calU) \cong V$ is either a point or a line. It is a line if and only if $v = \ker(a)$ for some $a \in A$.
\end{lemma}
\begin{proof}
On the ambient $\Gr(2,V)$, consider the dual map
\[
v: \CU^\perp \to V^* \otimes \CO_{\Gr} \xrightarrow{\ev_v}  \CO_{\Gr}
\]
The zero locus $Z$ of $v$ is isomorphic to $\PP(V/v)$, where the embedding of $\PP(V/v)$ into $\Gr(2,V)$ sends $w \in V$ to $w \wedge v \in \Gr(2,V)$.

Denote by $Z_Y$ the zero locus of $v$ on $Y$, and note that it is the intersection of $\PP(V/v)$ and $\PP(A^\perp)$ inside $\PP(\Lambda^2V^*)$, namely 
\[
Z_Y = \PP(V/v) \cap \PP(A^\perp) = \PP(A(v,-)^\perp/v)
\]
It follows that if $v= \ker(a)$ for some $a \in A$, then $A(v,-)^\perp$ is $3$-dimensional and $Z_Y$ is a line, otherwise $A(v,-)^\perp$ is $2$-dimensional and $Z_Y$ is a point.
\end{proof}

\begin{corollary}\label{cor:ideal of a line}
Any map $a:\CU \to \CU^\perp$ is injective and has cokernel isomorphic to the ideal of a line.
\end{corollary}
\begin{proof}
Any non-zero map from $\CU$ to $\CU^\perp$ is injective. 
Indeed, the two bundles are stable by lemma \ref{lemma:stability on Y} and have slope respectively $-1/2$ and $-1/3$, so that, if a non-zero map is not injective, either the kernel destabilizes $\CU$ or the image destabilizes $\CU^\perp$.

Extend the map $a$ to a complex
\begin{equation} \label{eq:zerolocusV}
\begin{diagram} 
\{ \CU	& \rTo^{a}	& \CU^\perp	& \rTo^{\ker(a)}	& \CO_Y & \rTo & \CO_L \}
\end{diagram}
\end{equation}
By lemma \ref{zerolocusV}
\[
\begin{diagram} 
\CU^\perp	& \rTo^{\ker(a)}	& \CO_Y & \rTo & \CO_L & \rTo & 0
\end{diagram}
\]
is exact. 
As $a: \CU \to \CU^\perp$ is injective, the complex \eqref{eq:zerolocusV} is equivalent to the shift of a sheaf.
By table \eqref{Cherntable}, the Chern character of \eqref{eq:zerolocusV} vanishes, so that it is an exact sequence.
\end{proof}

Not only we can associate a line with any map from $\CU$ to $\CU^\perp$, but also the converse is true. 
Moreover, the correspondence works in families and gives rise to the universal family of lines in $Y$.

\begin{proposition}\label{prop:fano of lines}
$\PP(A)$, $\PP(\Hom(\CU,\CU^\perp))$ and the Hilbert scheme of lines $\Hilb_L(Y)$ on $Y$ are isomorphic. Moreover, the exact sequence
\begin{equation}\label{eq:fano of lines 4}
0 \to \CO_{\PP(A)}(-3) \boxtimes \CU \to \CO_{\PP(A)}(-2) \boxtimes \CU^{\perp} \to \CO_{\PP(A) \times Y} \to \CO_\UL \to 0
\end{equation}
is a resolution of the structure sheaf of the universal family $\UL$ of lines on $Y$.
\end{proposition}
\begin{proof}
First, we construct a family of lines on $\PP(A) \times Y$, later on we will prove it is the universal family.

By lemma \ref{lemma:cohomologies on Y}, over $\PP(A) \times Y$ there is a universal map from $\CU$ to $\CU^\perp$
\[
\CO_{\PP(A)}(-3) \boxtimes \CU \to \CO_{\PP(A)}(-2) \boxtimes \CU^{\perp}
\]
The twist by $\CO_{\PP(A)}(-3)$ is chosen for compatibility with future notation. 
The above map extends to an $SL_2$-equivariant complex
\begin{equation}\label{eq:fano of lines 2}
\CO_{\PP(A)}(-3) \boxtimes \CU \to \CO_{\PP(A)}(-2) \boxtimes \CU^{\perp} \to \CO_{\PP(A) \times Y}
\end{equation}
By corollary \ref{cor:ideal of a line} the restriction of \eqref{eq:fano of lines 2} to any point of $\PP(A)$ is a pure object on $Y$, so that by criterion \ref{cor:qcoh flat} the complex \eqref{eq:fano of lines 2} is equivalent to a pure object flat over $\PP(A)$.

As a consequence we have an exact sequence
\begin{equation}\label{eq:fano of lines 3}
0 \to \CO_{\PP(A)}(-3) \boxtimes \CU \to \CO(-2)_{\PP(A)} \boxtimes \CU^{\perp} \to \CO_{\PP(A) \times Y} \to \CO_{\CL} \to 0
\end{equation}
where $\CL$ is a closed subscheme in $\PP(A) \times Y$ flat over $\PP(A)$. By $\eqref{eq:zerolocusV}$ it is a family of lines.


Now, assume a flat family of lines $\pi: L \to S$ over a base $S$ is given. We want to construct a map $f: S \to \PP(A)$ associated to $L$.
By theorem \ref{thm:relative Beilinson} there is a relative Beilinson spectral sequence
\[
R^q\pi_*(\CO_L \otimes (\Lvee E_{n-p})^*)\boxtimes E_{p+1} \xRightarrow{} \CO_L
\]
with respect to the full exceptional collection $\CO_Y(-1), \CU, \CU^\perp, \CO_Y$ and its left dual collection $\CO_Y, \CU^*, \CU^\perp(1), \CO_Y(1)$ (see proposition \ref{prop:full exceptional}).

We will now discuss the vanishing of most of the entries of the relative Beilinson spectral sequence.
Note that we can substitute $R^{>1}\pi_*(\CO_L \otimes -) = 0$ as the relative dimension of $L$ over $S$ is $1$.
It follows that the first page of the spectral sequence is
\begin{equation}\label{diag:Beilinson for lines}
\begin{diagram}
R^1\pi_*\CO_L(-1) \otimes \CO_Y(-1)	&\rTo	& R^1\pi_*\restr{V/\CU(-1)}{L} \otimes \CU	& \rTo	& R^1\pi_*\restr{\CU}{L}\otimes \CU^{\perp} 	& \rTo	& R^1\pi_*\CO_L \otimes \CO_Y 	\\
	&\rdTo(4,2)^{d_2}	&	&\rdTo(4,2)^{d_2}	&	&	& 	\\
\pi_*\CO_L(-1) \otimes \CO_Y(-1)	&\rTo	& \pi_*\restr{V/\CU(-1)}{L} \otimes \CU	& \rTo	& \pi_*\restr{\CU}{L}\otimes \CU^{\perp} 	& \rTo	& \pi_*\CO_L \otimes \CO_Y 	
\end{diagram}
\end{equation}
As $\CO_L(-1)$ is acyclic, the column $p = 0$ vanishes.
By base change and lemma \ref{lemma:splitting on lines}, we can use theorem \ref{cor:qcoh flat} on $R\pi_*(\CO_L \otimes V/ \CU (-1))$ and deduce that it is equivalent to a line bundle $\scrL_{\CU}$. 
Analogously, $R\pi_*(\CO_L \otimes \CU) \cong \scrL_{\CU^\perp}$ and $R\pi_*(\CO_L) \cong \CO_S$.
The spectral sequence \eqref{diag:Beilinson for lines} becomes
\[
\begin{diagram}
0	&\qquad \qquad	& 0 					& 	& 0 								& 	& 0	\\
	&		&								&	&								&	& 	\\
0	&		& \pi_*\restr{V/\CU(-1)}{L} \otimes \CU	&\rTo	& \pi_*\restr{\CU}{L}\otimes \CU^{\perp} 	&\rTo	& \pi_*\CO_L \otimes \CO_Y 	
\end{diagram}
\]

Altogether, we have an exact sequence
\begin{equation} \label{eq:fano of lines}
0 \to \scrL_\CU \boxtimes \CU \to \scrL_{\CU^\perp} \boxtimes \CU^\perp \to \CO_{S \times Y} \to \CO_L \to 0
\end{equation}
on $S \times Y$.

Take $\CHom(\CU, -)$ of the first map of sequence \eqref{eq:fano of lines} and push it forward to $S$. The result is a map
\begin{equation} \label{eq:fano of lines 5}
\scrL_\CU \to A \otimes \scrL_{\CU^\perp}
\end{equation}
As $\CO_L$ is flat over $S$, when we base change to $s \times Y$ the sequence \eqref{eq:fano of lines} is still exact, so that the map \eqref{eq:fano of lines 4} is an injection of vector bundles, and therefore defines a unique map $f: S \to \PP(A)$ such that the pullback of the universal
\[
\CO_{\PP(A)} (-1) \to A \otimes \CO_{\PP(A)}
\]
is \eqref{eq:fano of lines 5} twisted by $\scrL_{\CU^\perp}^*$.

So far we have constructed two correspondences: one associates a map $S \to \PP(A)$ with a family $\pi: L \to S$, the other associates a family to a map $S \to \PP(A)$ by pulling back $\CL$.
We need to check that they are natural and that they are mutually inverse.

The naturality follows easily from their definitions and from the functoriality of the Beilinson spectral sequence, which is discussed in theorem \ref{thm:relative Beilinson}.

The fact that if we start from a map from $S \to \PP(A)$, pullback the universal $\CL$ and use it to induce a map from $S \to \PP(A)$, then we recover the initial map is also a consequence of the functoriality of the Beilinson spectral sequence.

The fact that if we start from a family of lines $\pi: L \to S$, induce a map $S \to \PP(A)$ and pullback the universal $\CL$ via this map, then we recover the initial map can be explained in the following way. 
By definition of the two correspondences we recover the ideal sheaf of the family $L \subset S \times Y$. 
As $\codim(L) = 2$ and by lemma \ref{lemma:no ext^1 from codimension 2}, the ideal sheaf determines the subscheme $L$.
\end{proof}

\begin{remark}In proposition \ref{prop:fano of lines} we have established an $SL_2$-equivariant isomorphism between the Hilbert scheme of lines on $Y$ and $\PP(A)$.
Under the action of $SL_2$, $\PP(A)$ has two orbits: the open one and the conic $\CQ$ which is the zero locus of unique invariant form $q \in S^2A^*$. 

The lines corresponding to points on $\CQ$ are known in the literature as special lines, and we will see in proposition \ref{prop:special lines and normal bundle} that they are characterized by the fact that their normal bundle in $Y$ is not trivial.
\end{remark}

The following definition sets the notation which we are going to use for the identification of the Hilbert scheme of lines and of $\PP(A)$.
\begin{definition}\label{def:notation for lines}
Given an element $a \in A$, denote by $L_a$ the corresponding line in $Y$. Given a line $L$ in $Y$, denote by $a_L$ any element in $A$ such that $L$ corresponds to $[a_L] \in \PP(A)$.
\end{definition}

We are now going to state and prove a few consequences of proposition \ref{prop:fano of lines}.
We will denote the projections from the universal line $\UL$ to $\PP(A)$ and $Y$ by $r_A$ and by $r_Y$, 
as in the following diagram
\[
\begin{diagram}
		&			& \CL	&			&	\\
		& \ldTo^{r_A}	&		& \rdTo^{r_Y}	&	\\
\PP(A)	&			&		&			& Y	
\end{diagram}
\]

In what follows, given a point $P \in Y$, $U_P$ is the fiber of $\CU$ at $P$. All lines in $Y$ arise by the following simple construction.
\begin{corollary}\label{cor:lines through P and Q}
Given two points $P, Q \in Y$, there is a line connecting them if and only if $U_P$ and $U_Q$ have non-trivial intersection. 
In this case the line is given by
\[
\PP\left( \frac{U_P + U_Q}{ \ker a} \right) \subset \Gr(2,V)
\]
\end{corollary}
\begin{proof}
By proposition \ref{prop:fano of lines}, for $P$ and $Q$ to belong to a line $L$ it is equivalent that there exists $a \in A$ such that $\ker a = 0$ in both $V/U_P$ and $V/U_Q$. Therefore, if $P$ and $Q$ belong to $L$, the intersection $U_P \cap U_Q$ is not trivial.
In the other direction, if there exists $v \in U_P \cap U_Q$, then $A(v, U_P + U_Q) = 0$ inside $V^*$, so that $\PP((U_P + U_Q) / v)$ is a line which embeds in $Y$ and contains $P, Q$.
\end{proof}

\begin{corollary}\label{cor:3 to 1 cover}
The map $r_Y:\UL \to Y$ is a 3 to 1 cover. 
The multiplicity of preimages is $(1,1,1)$ over the open orbit, $(2,1)$ over the divisorial orbit and $(3)$ over the closed orbit.
\end{corollary}
\begin{proof}
It is enough to prove that the family of lines through any point is finite and that the rank of $r_{A*} \CO_\UL$ is 3.

The family of lines through any point is finite by the following argument. 
Assume there are infinitely many lines through $P \in Y$.  
By proposition \ref{prop:fano of lines} all lines have a resolution of the form \eqref{eq:zerolocusV}, so that each vector in $U_P$ is the kernel of some $a \in A$.
It follows that $\PP(U_P) \subset \sigma(A)$, against the fact that the $\sigma(A)$, being a surface of degree $2$, does not contain lines.

As for $r_{A*} \CO_\UL$, it is easily computed by pushing forward the resolution \eqref{eq:fano of lines 4}. 
The resulting spectral sequence shows that
$r_{A*} \CO_\UL \cong \CO \oplus \CU$.

By proposition \ref{prop:fano of lines}, the ramification divisor of $r_Y$ is linearly equivalent to $r_A^*\CO_{\PP(A)}(2)$. 
As it is $\SSL_2$-invariant,
it has to be the preimage of special lines $r_A^{-1}(\CQ)$.
For the same reason the ramification type over the $2$-dimensional orbit in $Y$ is $(2,1)$.

Finally, lines intersecting the closed orbit span an $\SSL_2$-invariant divisor, so that by lemma \ref{lemma:mukai umemura 1} and remark \ref{rmk:mukai umemura} they are all tangent to the closed orbit.
\end{proof}
\begin{remark}
By corollary \ref{cor:3 to 1 cover}, the family of lines through a general point in $Y$ is not connected, so that condition \eqref{itm:irreducible fiber} of theorem \ref{thm:GM} is not verified.
\end{remark}

We will now describe the intersection pairing between lines in $Y$.
Given two distinct lines $L_1$ and $L_2$ with ideals $I_1$ and $I_2$ there are several equivalent ways to check whether they intersect or not. 
We look for a condition which is minimal among the ones which cut the closure of the locus of intersecting distinct lines inside $\PP(A) \times \PP(A)$.

Let $SL_2$ act on $\PP(A) \times \PP(A)$ componentwise. 
Then the space $H^0(\CO(1,1))$ contains a unique $SL_2$-invariant line.
Indeed, by Littlewood--Richardson it is $SL_2$-isomorphic to $A^* \otimes A^* = V \oplus A \oplus \CC$. 
Under the decomposition $A^* \otimes A^* \cong \Lambda^2 A^* \oplus S^2 A^*$, the $SL_2$-invariant element is symmetric as $\Lambda^2 A^* \cong A$.

As a consequence, we have selected an $SL_2$-invariant effective divisor inside $\PP(A) \times \PP(A)$. 
We will denote it by $\IQ$. 
Note that, by means of the $SL_2$-invariant quadric $q$, we can identify $A$ and $A^*$, so that we find an $SL_2$-equivariantly embedded flag variety $\Flag(1,2;A)$ inside $\PP(A) \times \PP(A)$. 
Moreover the flag variety is an effective divisor of type $(1,1)$, so that by uniqueness of such a divisor we have
\begin{equation}\label{eq:IQ is flag}
IQ = \Flag(1,2;A) \subset \PP(A) \times \PP(A)
\end{equation}

\begin{proposition} \label{prop:intersecting lines}
$\IQ$ is the closure of the locus of distinct intersecting lines.
\end{proposition}
\begin{proof}
The $\SSL_2$-invariant conic in $S^2A^*$ induces an $\SSL_2$-invariant section
\begin{equation}\label{eq:alternative IQ}
\CO_{\PP(A)}(-2) \boxtimes \CO_{\PP(A)}(-2) \to V \boxtimes V \to A^* \otimes \CO_{\PP(A) \times \PP(A)}
\end{equation}
over $\PP(A) \times \PP(A)$.
The composition of \eqref{eq:alternative IQ} with the $\SSL_2$-invariant surjection
\begin{equation} \label{eq:alternative IQ 1}
A^* \otimes \CO_{\PP(A) \times \PP(A)} \to \CO_{\PP(A)}(1) \boxtimes \CO_{\PP(A)} \oplus \CO_{\PP(A)} \boxtimes \CO_{\PP(A)}(1)
\end{equation}
vanishes: one can check it fiberwise at $(a_1,a_2)$, where the composition of \eqref{eq:alternative IQ} and \eqref{eq:alternative IQ 1} becomes
\[
a_1^2 \otimes a_2^2 \mapsto \ker a_1 \otimes \ker a_2 \mapsto \ker a_1  \wedge \ker a_2 \mapsto a_1(\ker a_1 , \ker a_2 ) \oplus a_2(\ker a_1 , \ker a_2 ) = 0
\]

As a consequence, map \eqref{eq:alternative IQ} factors via the kernel of \eqref{eq:alternative IQ 1}. 
This provides us with an $\SSL_2$-invariant
\[
\CO_{\PP(A)}(-2) \boxtimes \CO_{\PP(A)}(-2) \to \CO_{\PP(A)}(-1) \boxtimes \CO_{\PP(A)}(-1)
\]
which vanishes at $(a_1, a_2)$ if and only if $A( \ker a_1, \ker a_2) = 0$, that is to say if and only if $\ker a_1 \wedge \ker a_2 \in Y$.
By description \ref{cor:lines through P and Q} and for $a_1 \neq a_2$, this happens if and only if $L_{a_1}$ and $L_{a_2}$ intersect.
\end{proof}

\begin{proposition}\label{prop:special lines and normal bundle}
A line $L = L_a$  
has normal bundle $\CO_L(-1) \oplus \CO_L(1)$ if and only if $(a,a) \in \IQ$, otherwise its normal bundle is trivial.
\end{proposition}
\begin{proof}
The derived tensor product
\[
\CO_L \Lotimes \CO_L
\]
is equivalent to $\CO_L[2] \oplus \CN_{L/Y}^*[1] \oplus \CO_L$.
By resolution \eqref{eq:fano of lines}, it is also equivalent to
\[
\CO_L(-1) \oplus \CO_L \to \CO_L(-1) \oplus \CO_L  \oplus \CO_L \to \CO_L
\]
By removing the $\Tor_0$ and $\Tor_2$ components, we are left with an exact sequence
\[
0 \to \CO_L(-1) \to  \CO_L(-1) \oplus \CO_L  \oplus \CO_L \to \CN_{L/Y}^*[1] \to 0 
\]
which shows that the normal bundle is nontrivial if and only if the $\CO_L(-1)$ component of the leftmost map in the above sequence vanishes.
By the description of lemma \ref{lemma:cohomologies on Y} for the maps from $\CU$ to $\CU^\perp$, this condition is equivalent to the fact that
\[
A(\kappa(a), \kappa(a)) = 0
\]
Note that this condition is equivalent to the vanishing of the composition
\begin{equation}\label{eq:special and normal}
\CO_{\PP(A)}(-1) \to A \otimes \CO_{\PP(A)} \to \Lambda^2V^* \otimes \CO_{\PP(A)} \to \Lambda^2 \kappa^*\CK^*
\end{equation}
Moreover, as $A(\sigma(a), \kappa(a)) = 0$ and $\sigma(a) \in \kappa(a)$, the composition \eqref{eq:special and normal} factors via
\[
\CO_{\PP(A)}(1) \cong \Lambda^2(\kappa^*\CK/\sigma^*\CO_{\PP(V)}(-1))^* \to \Lambda^2 \kappa^* \CK^*
\]
where the isomorphism on the left follows from lemma \ref{lemma:deg korth}.
As a consequence, the set of lines with nontrivial normal bundle correspond to the zero-locus of a section of $\CO_{\PP(A)}(2)$.
As the unique $\SSL_2$-invariant quadric is $\IQ$, it follows that the normal bundle to $L_a$ is non-trivial if and only if $(a,a)$ is in $\IQ$.
\end{proof}

\begin{remark}
In other words, two distinct lines intersect if and only if they are polar with respect to the unique $SL_2$-invariant conic $\CQ \subset \PP(A)$.
Proposition \ref{prop:special lines and normal bundle} says that $(a,a)$ lies on $\IQ$ if and only if there is an infinitesimal deformation of $L_a$ fixing a point of $L_a$.
%
\end{remark}

The resolution \eqref{eq:fano of lines} for the structure sheaf of the universal line allows us to compute the action of the Fourier--Mukai functor
\[
\Phi_{\CL}: \D^b(Y) \to \D^b(\PP(A))
\]
with kernel the structure sheaf of $\CL$.
We compute the transforms of a few bundles which we will need later.
\begin{lemma}\label{lemma:FM transform lines}
The following equalities hold
\begin{equation}\label{eq:FM transform lines}
\Phi_{\CL}(\CU(-1)) = \CO_{\PP(A)}(-3)[1], \qquad \Phi_{\CL}(\CU^\perp(-1)) = \CO_{\PP(A)}(-2)[1]
\end{equation}
Moreover,
\begin{equation}\label{eq:FM transform lines 2}
\Phi_{\CL}: A \cong \Hom(\CU(-1), \CU^\perp(-1)) \to \Hom(\CO_{\PP(A)}(-3), \CO_{\PP(A)}(-2)) \cong A^*
\end{equation}
is the unique (up to scalar) $SL_2$-equivariant isomorphism.
\end{lemma}
\begin{proof}
Both equalities in \eqref{eq:FM transform lines} are proved by means of the Grothendieck spectral sequence and using the resolution \eqref{eq:fano of lines}.
The entries in the spectral sequence can be computed via K\"unneth formula and Borel--Bott--Weil formula and in both cases there is only one non-vanishing entry.

%
As for the isomorphism \eqref{eq:FM transform lines 2}, apply $\Phi_{\CL}$ on the mutation triangle \eqref{eq:mutation U Uperp}
\[
0 \to V/\CU(-1) \to A \otimes \CU \xrightarrow{\ev_{\CU}} \CU^\perp \to 0
\]
By \eqref{eq:FM transform lines} we find a distinguished triangle
\[
\Phi_{\CL}(V/\CU(-1))[-1] \to A \otimes \CO_{\PP(A)}(-3) \to \CO_{\PP(A)}(-2)
\]
If \eqref{eq:FM transform lines 2} is zero, then $\Phi_{\CL}(V/\CU(-1))$ is the direct sum of two line bundles with different shifts.
As we know by  lemma \ref{lemma:splitting on lines} and by base change that $\Phi_{\CL}(V/\CU(-1))$ restricts to any point $a \in \PP(A)$ as $\CC \oplus \CC$, this is not possible, so that \eqref{eq:FM transform lines 2} has to be an isomorphism.
%
%
%
\end{proof}
\begin{remark}
Actually, we have also proved that $\Phi_{\CL} (V/\CU(-1))$ is $\Omega_{\PP(A)}(-2)[1]$.
\end{remark}

\subsection{Conics in $Y$} \label{sec:conics}
We will prove in proposition \ref{prop:universal conic} that the Hilbert scheme of conics in Y is canonically identified with $\PP(V^*)$. 
By conic on $Y$ we mean any subscheme $C \subset Y$ such that $\chi(\calO_C(n)) = 2t+1$.
\begin{remark}\label{rmk:h1 of a conic}
Note that the only such subschemes in projective spaces are plane conics (corollary \ref{cor:d=1,2,3 in P^n} part \eqref{itm:d=2}).
It follows that they satisfy $h^1(\CO_C) = 0$ and they have no embedded points.
\end{remark}

When the conic $C$ is smooth, it makes sense to talk about splitting type of a vector bundle on it. 
Still, when $C$ is singular the splitting type on components (or on the line supporting the non-reduced conic) does not characterize uniquely the vector bundle. 
In these cases it is better to have cohomological properties, which in the smooth case allow to recover the splitting type.
We pack such properties in the decomposition of $\calO_C$ with respect to the full exceptional collection \eqref{eq:coll}. 

With some help from lemmas \ref{lemma:canonical of a conic} -- \ref{resolutionofconicaux1}, we prove proposition \ref{prop:universal conic}, which describes the universal conic $\calC \subset \PP(V^*) \times Y$ as the zero locus of a regular section of $\CO_{\PP(V^*)}(1) \boxtimes \CU^*$, as $\PP_Y(\CU^\perp)$ and as the restriction of $\Flag(2,4; V) \subset \PP(V^*) \times \Gr(2,V)$ to $\PP(V^*) \times Y$.


\begin{proposition} \label{prop:universal conic}
The Hilbert scheme of conics $\Hilb^{2t+1}_Y$ is canonically isomorphic to $\PP(V^*) \cong \PP(H^0(\CU^*))$. The universal conic $\calC$ has a resolution
\begin{equation} \label{eq:universal conic}
0 \to \CO_{\PP(V^*)}(-2) \boxtimes \CO_Y(-1) \to \CO_{\PP(V^*)}(-1) \boxtimes \CU \to \CO_{\PP(V^*) \times Y} \to \CO_\calC \to 0
\end{equation}
on $\PP(V^*) \times Y$.
The schemes
\begin{equation}\label{prop:universal conic is proj}
\calC \cong \PP(\CU^\perp) \cong \restr{\Flag(2,4,V)}{\PP(V^*) \times Y} \subset \Flag(2,4,V) \subset \PP(V^*) \times \Gr(2,V)
\end{equation}
are isomorphic as subschemes of $\PP(V^*) \times Y$.
\end{proposition}
\begin{proof}
Given a flat family $C \to S \times Y$ of conics in $Y$, take the relative Beilinson spectral sequence \ref{thm:relative Beilinson} with respect to the full exceptional collection \eqref{eq:coll}, i.e. $\CO(-1), \CU, \CU^\perp, \CO $.

Denote by $\pi_S$ the projection from $C \to S$ and by $\pi_Y$ the projection $C \to Y$. As the relative dimension of $\pi_S$ is 1, when computing the entries of the relative Beilinson spectral sequence for $\CO_C$ we will restrict to $R^0\pi_{S*}$ and $R^1\pi_{S*}$.

We start from $R\pi_{S*}(\CO_C)$: by remark \ref{rmk:h1 of a conic}, $H^1(\CO_{C_s}) = 0$ 
for any $s \in S$, so that by base change $R^1\pi_{S*}(\CO_C) = 0$. As a consequence, $R^0\pi_{S*}(\CO_C)$ is locally free of rank 1. As it has a never vanishing section induced by the $1$ section of $\CO_C$, it is $\CO_S$.

The entries in the second column are the cohomology of $R\pi_{S*}(\CO_C \otimes \pi_Y^*\CU)$. 
By lemma \ref{cohomologyonconic2}, for any conic $C_s$ the cohomology $H^\bullet(C_s, \CU)$ vanishes, so that also $R\pi_{S*}(\CO_C \otimes \pi_Y^*\CU)$ vanishes.

The entries in the third column are the cohomology of $R\pi_{S*}(\CO_C \otimes \pi_Y^*V/\CU(-1))$. 
By lemma \ref{lemma:canonical of a conic} and by Serre duality, for any conic $C_s$
\[
H^\bullet(\restr{V/\CU(-1)}{C_s}) = H^\bullet(\restr{\CU^\perp}{C_s}[3])^*
\]
By lemma \ref{cohomologyonconic}, for any conic $C_s$ we find $H^\bullet(C_s, V/\CU(-1)) =  \CC[-1]$. 
By lemma \ref{cor:qcoh flat}, $R\pi_{S*}(\CO_C \otimes \pi_Y^*V/\CU(-1))$ is quasi isomorphic to the shift of a line bundle. Denote it by $\scrL_{\CU}$.

The entries of the last column are the cohomology of $R\pi_{S*}(\CO_C(-1))$.
By lemma \ref{lemma:canonical of a conic} and remark \ref{rmk:h1 of a conic}, for any conic $C_s$ we find $H^\bullet(C_s, \CO_Y(-1)) = \CC[-1]$.
Again, by lemma \ref{cor:qcoh flat}, $R\pi_{S*}(\CO_C(-1))$ is quasi isomorphic to the shift of a line bundle.
Denote it by $\scrL_{\CO(-1)}$.

As the Beilinson spectral sequence \ref{thm:relative Beilinson} converges to $\CO_C$, there is an exact sequence
\begin{equation}\label{eq:universal conic 5}
0 \to \scrL_{\CO(-1)} \boxtimes \CO_Y(-1) \to \scrL_{\CU} \boxtimes \CU \to \CO_{S \times Y} \to \CO_C \to 0
\end{equation}
Twisting by $\CO_Y(1)$ and pushing forward to $S$, 
we get an exact sequence
\begin{equation}\label{eq:universal conic thief}
0 \to \scrL_{\CO(-1)} \to \scrL_{\CU} \otimes V^* \to \CO_S \otimes \Lambda^2V^* / A \to \pi_*(\CO_C(1)) \to 0
\end{equation}
To get a map from $S$ to $\PP(V^*)$ we still need to prove that the first map in \eqref{eq:universal conic thief} is an injection of vector bundles. 
To prove it, we base change to $s \in S$ by first twisting \eqref{eq:universal conic 5} by $\CO_Y(1)$, then restricting the resulting complex to $s \times Y$ and finally pushing forward to $s$. 
By flatness of $\CO_C$ over $S$, we find an exact sequence
\[
0 \to \CC \to V^* \to \Lambda^2V^* / A \to H^0(\restr{\CO_Y(1)}{C}) \to 0
\]
which shows that
\[
\scrL_{\CO(-1)} \to \scrL_{\CU} \otimes V^*
\]
is an injection of vector bundles.
In the above way, given a family of conics $C$, we have constructed a map from $S$ to $\PP(V^*)$

We will now construct a family which will be the universal family of conics. The structure sheaf of the subscheme 
\[
\PP_Y(\CU^\perp) \subset \PP(V^*) \times Y
\]
is cut by the Koszul complex
\[
0 \to \CO_{\PP(V^*)}(-2) \boxtimes \CO_Y(-1) \to \CO_{\PP(V^*)}(-1) \boxtimes \CU \to \CO_{\PP(V^*) \times Y}
\]
of the $SL_2$-invariant section of $\CO_{\PP(V^*)}(1) \boxtimes \CU^*$.

All fibers of the projection
\[
\PP_Y(\CU^\perp) \to \PP(V^*)
\]
have dimension $1$ as there is no map from $\CU$ to $\CO_Y$ vanishing on a divisor. 
It follows that $\PP_Y(\CU^\perp)$ is flat over $\PP(V^*)$ and that therefore
\[
0 \to \CO_Y(-1) \to \CU \to \CO_Y
\]
is exact for each $v \times Y$ inside $\PP(V^*) \times Y$.
By table \eqref{Cherntable}, $\PP_Y(\CU^\perp)$ is a family of conics.

What is still left to prove is that the two constructions are functorial and that they are inverse to each other: both facts can be proved in the same way they were proved for lines in proposition \ref{prop:fano of lines}, again by means of lemma \ref{lemma:no ext^1 from codimension 2}.

As for the isomorphism of $\PP_Y(\CU^\perp)$ and the restriction of the flag variety $\Flag(2,4,V)$ to $\PP(V^*) \times Y$, this is a direct consequence of the fact that $\Flag(2,4,V)$ embeds into $\PP(V^*) \times \Gr(2,V)$ as $\PP_{\Gr}(\CU^\perp)$ and that restriction commutes with projectivization.
\end{proof}

Note that by proving proposition \ref{prop:universal conic} we have also proved the following corollary.

\begin{corollary} \label{cor:sections of U*}
Any section of $\CU^*$ is regular and fits in a Koszul complex
\begin{equation} \label{eq:sections of U*}
0 \to \CO_Y(-1) \to \CU \to \CO_Y \to \CO_C \to 0
\end{equation}
where $C$ is a conic in $Y$.
\end{corollary}
\begin{proof}
In \ref{prop:universal conic} we have shown that given $\CU \to \CO_Y$ the cokernel is always a conic, so that any section of $\CU^*$ is regular. Then the Koszul complex is exact.
\end{proof}

The following definition sets the notation which we are going to use for the identification of the Hilbert scheme of conics and of $\PP(V^*)$.
\begin{definition}\label{def:notation for conics}
Given an element $w \in V^*$, denote by $L_w$ the corresponding conic in $Y$. Given a conic $C$ in $Y$, denote by $w_C$ any element in $V^*$ such that $C$ corresponds to $[w_C] \in \PP(V^*)$.
\end{definition}

Actually, proposition \ref{prop:universal conic} also says that any conic is obtained by choosing a non zero vector $v \in V^*$ and taking $\Gr(2,\ker v) \cap \PP(A^\perp)$ inside $\PP(\Lambda^2 V)$. Viceversa, every time we choose $v \in V^*$ and intersect $\Gr(2,\ker v)$ embedded via Pl\"ucker with $\PP(A^\perp)$, we find a conic in $Y$, just because $\Gr(2,4)$ itself is a quadric in $\PP^5$.
\begin{corollary} \label{cor:boh conic}
The zero locus of any section of $\CU^*$ is a conic.
\end{corollary}
\begin{proof}
The last map in the resolution \eqref{eq:universal conic} is the universal map from $\CU$ to $\CO_Y$.
\end{proof}


Here are the lemmas we have used in the proof of \ref{prop:universal conic}. Recall that a conic is any subscheme $C$ of $Y$ whose Chern character is $2L-P$, or equivalently whose Hilbert polynomial is $2t + 1$. 

\begin{lemma}\label{lemma:canonical of a conic}
The canonical bundle of any conic $C$ is the restriction of $\CO_Y(-1)$ to $C$.
\end{lemma}
\begin{proof}
As $\CO_Y(-1)$ is the restriction of $\CO(-1)$ from a projective space, it is enough to prove the same statement for a conic $C$ embedded in $\PP^n$.
Moreover, any $C \in \PP^n$ factors via a $\PP^2$ spanned by $H^0(\CO_C(1))^*$ by remark \ref{rmk:h1 of a conic}. 

A conic in $\PP^2$ is a Cartier divisor (as it has no embedded points by remark \ref{rmk:h1 of a conic}), so that by adjunction
\[
\omega_C \cong \restr{\omega_{\PP^2} \otimes \CO_{\PP^2}(2)}{C} \cong \restr{\CO_{\PP^2}(-1)}{C}
\]
which is the thesis.
\end{proof}

\begin{lemma}\label{cohomologyonconic}
The following relations hold on $Y$ for any conic $C$:
\[
\begin{aligned}
\chi(\restr{\calU^\perp}{C}) & = 1 \\
h^0(\restr{\calU^\perp}{C}) & = 1 \\
h^1(\restr{\calU^\perp}{C}) & = 0
\end{aligned} 
\qquad \qquad
\begin{aligned}
\chi(\restr{\calU^*}{C}) & = 4 \\
h^0(\restr{\calU^*}{C}) & = 4 \\
h^1(\restr{\calU^*}{C}) & = 0
\end{aligned}
\]
\end{lemma}
\begin{proof}
To compute $\chi(\restr{\calU^\perp}{C})$ use Riemann--Roch.

To check $h^0(\restr{\calU^\perp}{C})  = 1$ it is then enough to check $h^0(\restr{\calU^\perp}{C}) \leq 1$. Assume this is not the case, i.e. that there is $B \subset H^0(\restr{\calU^\perp}{C}), \dim(B) = 2$. 
This implies that $C$ is schematically contained in $Y' = \Gr(2,B^\perp) \cap Y \subset \Gr(2,V)$. 
As $Y$ is a linear section of $\Gr(2,V)$, $Y' \subset Y$ is a linear section of $\Gr(2,B^\perp) = \PP^2$ linearly embedded in $\PP(\Lambda^2 V^*)$. 
Moreover, we know by Lefschetz that $Y$ contains no $\PP^2$, so that $Y'$ is either a line or a point. As $C$ is a conic and $Y'$ is a line or a point, $C$ can't be a subscheme of $Y'$, so that it must be $h^0(\restr{\calU^\perp}{C}) \leq 1$.

As for $h^1(\restr{\calU^\perp}{C}) = 0$, it follows by subtracting the other two results and using $\dim(C) =1$ to get the vanishing of higher cohomology.

The results for $\restr{\calU^*}{C}$ follow using the defining sequence \eqref{eq:definingsequencedual} restricted to $C$ and the results for $\restr{\calU^\perp}{C}$ which we have just proved.
\end{proof}

\begin{lemma}\label{cohomologyonconic2}
The following relations hold on $Y$ for any conic $C$:
\[
\begin{aligned}
\chi(\restr{\calU}{C}) & = 0  \\
h^0(\restr{\calU}{C}) & = 0 \\
h^1(\restr{\calU}{C}) & = 0
\end{aligned} 
\qquad \qquad
\begin{aligned}
\chi(\restr{V/\calU}{C}) & = 5  \\
h^0(\restr{V/\calU}{C}) & = 5 \\
h^1(\restr{V/\calU}{C}) & = 0
\end{aligned}
\]
\end{lemma}
\begin{proof}
To compute $\chi(\restr{\calU}{C})$ use Riemann--Roch.

To check $h^0(\restr{\calU}{C})  = 0$, assume $\restr{\calU}{C}$ has a section. By remark \ref{rmk:h1 of a conic} we can think of it as $v \in V$. Then $C$ embeds in the zero locus of some $v \in H^0(Y, V/\calU)$, which by \ref{zerolocusV} is either a line or a point.

The other results follow as in \ref{cohomologyonconic}.
\end{proof}

\begin{lemma}\label{resolutionofconicaux1}
The following relations hold on $Y$ for any conic $C$:
\begin{align*}
ext^1(\calU,\calI_C) & = 0 \\
hom(\calU,\calI_C) & = 1
\end{align*}
\end{lemma}
\begin{proof}
In order to prove $ext^1(\calU,\calI_C) = 0$, we will prove the surjectivity of
\[
H^0(\calU^*) \rightarrow H^0(\restr{\calU^*}{C})
\]
Take the diagram
\begin{equation}
\begin{tikzcd}
0 	\arrow{r}  	& \calU^\perp 	\arrow{r} \arrow{d}	&V^* \otimes \calO_Y 	\arrow{r} \arrow{d}	& \calU^*	\arrow{r} \arrow{d}  & 0 \\
0 	\arrow{r} 	& \restr{\calU^\perp}{C} \arrow{r}	&V^* \otimes \calO_C	\arrow{r}			& \restr{\calU^{*}}{C}	\arrow{r}	& 0
\end{tikzcd}
\end{equation}
After applying $H^0$ it becomes
\[
\begin{diagram}
0 	& \rTo	& 0 						& \rTo	&V^*			& \rTo^{\cong}	& H^0(\CU^*)			&\rTo	& 0 	\\
	&		& \dTo					& 		& \dTo^{1_{V*}}	&			& \dTo				& 	&	\\
0 	& \rTo	& H^0(\restr{\CU^{\perp}}{C})	& \rTo	& V^*		& \rOnto		& H^0(\restr{\CU^*}{C})	&\rTo	& 0	\\
\end{diagram}
\]
where we have used $0 = H^1(\restr{\calU^\perp}{C})$ from \ref{cohomologyonconic}, and the rows are exact. Surjectivity follows by commutativity of the diagram.

To check $hom(\calU,\calI_C) = 1$ use the long exact sequence of
\[
0 \rightarrow \restr{\calU^*}{C} \otimes \calI_C \rightarrow \calU^* \rightarrow \restr{\calU^*}{C} \rightarrow 0
\]
together with $h^0(\restr{\calU^*}{C}) = 4$ from \ref{cohomologyonconic} and with $ext^1(\calU,\calI_C) = 0$, which we have just proved.
\end{proof}

Now we want to use the above cohomological properties to recover the splitting type of the tautological bundles on conics.
\begin{corollary}\label{splittingonconic}
The splitting types of $\calU, V / \calU, \calU^\perp, \calU^*$ on any smooth conic $C$ are as follows:
\begin{gather*}
\ST_C(\CU) = (-1, -1) \\
\ST_C(V/\CU) = (0, 1, 1) \\
\ST_C(\CU^\perp) = (-1, -1, 0) \\
\ST_C(\CU^*) = (1,1)
\end{gather*}
\end{corollary}
\begin{proof}
The vector bundle embedding of $\CU^\perp$ into $V^* \otimes \CO_Y$ implies that $\restr{\calU^\perp}{C}$ cannot have direct summands of strictly positive degree.
Moreover, by lemma \ref{cohomologyonconic}, we have $h^1(\restr{\calU^\perp}{C}) = 0$, which yields
\[
\ST_C(\CU^\perp) = (-1, -1, 0)
\]
as only possible splitting. Dualizing we also obtain the splitting type of $V/\calU$.

To compute the splitting of $\calU$, recall that by lemma \ref{cohomologyonconic2} $\restr{\calU}{C}$ has no sections, which implies
\[
\ST_C(\CU) = (-1, -1)
\]
Dualize to get the splitting type of $\calU^*$.
\end{proof}

For any pair of points in $\PP^3$ there is a line connecting them, and the space of lines through each point is irreducible. 
These properties are used in \cite{okonek1980vector} to prove that the splitting type of a stable bundle on a generic line has no gaps. 
As we would like to recover analogous results for instantons on $Y$, we prove now some properties for conics in $Y$ which are analogous to properties of lines in $\PP^3$ that we have just mentioned.

Recall from proposition \ref{prop:universal conic} that the universal conic $\calC$ has two natural projections
\[
\begin{diagram}
		&			& \calC	&			&	\\
		& \ldTo^{q_A}	&		& \rdTo^{q_Y}	&	\\
\PP(V^*)	&			&		&			& Y	
\end{diagram}
\]
which are naturally identified with the projections from $\PP_Y(\CU^\perp)$ and from $\restr{\Flag(2,4,V)}{Y}$ to $\PP(V^*)$ and $Y$.

This last fact has a few straightforward consequences.
Denote by $\Hilb^{2t+1}_{Y,P}$ the Hilbert scheme of conics passing through a point $P \in Y$ and by $\calC_P$ the universal conic passing through $P \in Y$.
There is a cartesian square
\[
\begin{diagram}
\Hilb^{2t+1}_{Y,P}	& \rTo	& \calC	\\
\dInto			& \square	& \dInto	\\
\PP(V^*) \times P	& \rTo	& \PP(V^*) \times Y
\end{diagram}
\]
as one can check by comparing the universal properties of the fiber product 
\[
\calC \times_{\PP(V^*) \times Y} \left( \PP(V^*) \times P \right)
\] 
and of $\Hilb^{2t+1}_{Y,P}$.
\begin{corollary}\label{cor:conics through P}
The Hilbert scheme of conics passing through a point $P \in Y_5$ is canonically isomorphic to $\PP(U_P^\perp) \to \PP(V^*)$ under the identification of $\PP(V^*)$ and $\Hilb^{2t+1}_Y$.
\end{corollary}
\begin{proof}
By isomorphism \eqref{prop:universal conic is proj} and by its definition, $\Hilb^{2t+1}_{Y,P}$ is the fiber of $\PP(\CU^\perp)$, that is to say $\PP(U_P^\perp) \to \PP(V^*)$.
\end{proof}
\begin{remark}\label{rmk:GM for conics}
By corollary \ref{cor:conics through P}, condition \eqref{itm:irreducible fiber} of theorem \ref{thm:GM} holds for the universal family of conics in $Y$.
\end{remark}

Another consequence of \ref{prop:universal conic} is the following description of the space of conics through two points.
\begin{corollary}
Given any two points $P,Q \in Y$ there is at least one conic $C$ connecting them. Moreover the conic $C$ is unique if and only if $P,Q$ do not lie on a line.
\end{corollary}
\begin{proof}
The Hilbert scheme of conics through $P$ and $Q$ is by definition the fiber product of $\PP(U_P^\perp) \to \PP(V^*)$ and $\PP(U_Q^\perp) \to \PP(V^*)$, that is to say $\PP((U_P + U_Q)^\perp)$.

By corollary \ref{cor:lines through P and Q}, $\PP((U_P + U_Q)^\perp)$ is a point if and only if $P$ and $Q$ do not line on a line.
In the case $P$ and $Q$ lie on a line, by the same corollary, there is a $\PP^1$ of conics via $P$ and $Q$. 
All of them contain the line through $P$ and $Q$ (and are therefore singular) as otherwise there would be a plane elliptic curve. 
As $Y$ is cut by quadrics, if it contains a plane elliptic curve, it contains the whole plane, which is against Lefschetz hyperplane section theorem.
\end{proof}

The following proposition describes the locus of singular conics.
\begin{proposition}\label{prop:singular conics}
The locus of singular conics in $Y$ is the image of a degree $2$ map from $\IQ$ to $\PP(V)$.
It is a divisor of degree $3$ which is singular along a rational normal curve of degree $4$ representing double lines.
\end{proposition}
\begin{proof}
Over the space 
\[
\IQ \subset \PP(A) \times \PP(A)
\]
of intersecting pairs of lines there is a family of conics. 
This is clearly true for points corresponding to distinct lines, while it follows from the normal bundle description \ref{prop:special lines and normal bundle} in the case of points on the diagonal of $\PP(A) \times \PP(A)$.

This family of lines induces a regular map
\begin{equation}\label{eq:singular conics}
\IQ \to \PP(V^*)
\end{equation}
which by proposition \ref{prop:intersecting lines} and corollary \ref{cor:conics through P} has degree $(1,1)$.
The variety $\IQ$ has an involution interchanging the two factors, and the map \eqref{eq:singular conics} factors through this involution.
Keeping this in mind and using projection formula, one can compute the degree of the image of \eqref{eq:singular conics} by intersecting it with a line in $\PP(V^*)$.

Finally, one can check that the quotient by the involution of the map \eqref{eq:singular conics} is a closed embedding, so that its singular locus is the fixed locus of the involution, namely the curve of double lines.
As there is only one orbit of dimension $1$ for the action of $SL_2$ on $\PP(V^*)$, the curve of double lines is a rational normal curve of degree $4$ in $\PP(V^*)$.
\end{proof}


We will also discuss the intersection pairing of a line and a conic in lemma \ref{lemma:ILC}.

\subsection{Cubics}


The Beilinson spectral sequence approach works in the case of cubics in the same way as in the cases of lines and conics. 
As we have dealt with these other two cases in great detail in proposition \ref{prop:fano of lines} and proposition \ref{prop:universal conic}, we will only sketch the proofs in the case of cubics.

\begin{remark}\label{rmk:vanishing for cubics}
By cubic we mean a subscheme $T$ of $Y$ whose Hilbert polynomial is $3t + 1$. 
By corollary \ref{cor:d=1,2,3 in Y} such subschemes have no embedded points and satisfy $h^1(\CO_T) = 0$. 
\end{remark}

\begin{proposition}
The Hilbert scheme of cubics in $Y$ is isomorphic to $\Gr(2,V)$. The universal family of cubics $\UCubic$ has a resolution
\[
0 \to \CU(-1) \boxtimes \CO_{Y}(-1) \to \CO_{\Gr}(-1) \boxtimes V/\CU(-1) \to \CO_{\Gr} \boxtimes \CO_{Y} \to \CO_{\UCubic} \to 0
\]
inside $\Gr(2,V) \times Y$.
\end{proposition}
\begin{proof}
First we construct a family of cubics.
Via Borel--Bott--Weil one can construct on $\Gr(2,V) \times Y$ a complex
\[
\CU(-1) \boxtimes \CO_{Y}(-1) \to \CO_{\Gr}(-1) \boxtimes V/\CU(-1) \to \CO_{\Gr} \boxtimes \CO_{Y}
\]
where $\CU$ is as usual the rank $2$ tautological bundle on $\Gr(2,V)$.
One can then prove that it has a unique non-vanishing cohomology and that its cohomology is flat over $\Gr(2,V)$.
The Chern character computation proves that it is a family of cubics in $Y$.

Given a family of cubics $T \subset S \times Y$, compute the relative Beilinson spectral sequence \ref{thm:relative Beilinson} with respect to the dual collections
\begin{gather*}
\langle \CO_Y(-1), V/\CU(1), \CU, \CO_Y \rangle \\
\langle \CO_Y, V/\CU, \CU^*, \CO_Y(1)	\rangle
\end{gather*}
which one can obtain by dualizing collections \eqref{eq:coll} and \eqref{eq:coll dual}.

From remark \ref{rmk:vanishing for cubics} we can deduce several other vanishings about the tautological bundles on $Y$ restricted to $T$, which yield an exact sequence
\[
0 \to \scrM_{\CO(-1)} \boxtimes \CO_Y(-1) \to \scrL_{V/\CU(-1)} \boxtimes V/\CU(-1) \to \CO_{S} \boxtimes \CO_{Y} \to \CO_T \to 0
\]
By corollary \ref{cor:qcoh flat} the sheaf $\scrM_{\CO(-1)}$ is a vector bundle of rank $2$, while $\scrL_{V/\CU(-1)}$ is a line bundle.

The proof that the two correspondences between families of cubics and maps to $\Gr(2,V)$ are inverse to each other is analogous to that of lines in proposition \ref{prop:fano of lines}.
\end{proof}

\begin{remark}
A more geometric construction for the universal family of cubics is the following.
We will see later that there is a diagram
\[
\begin{diagram}
		&			&\PP_{Y}(\CU)	&\cong \Bl_{\Imsigma} \PP(V)	&			&	\\
		& \ldTo		&			& 						& \rdTo_{b}	&	\\
Y		&			&			&						&			& \PP(V)
\end{diagram}
\]
By proposition \ref{picard of blow up}, the Fourier--Mukai transform functor with kernel the structure sheaf $\Bl_{\Imsigma} \PP(V)$ of a line in $\PP(V)$ is a cubic in $Y$.
The Koszul complex for the universal line 
\[
\PP(\CTaut_2) \subset \Gr(2,V) \times \PP(V)
\]
yields a resolution
\[
0 \to \CO_{\Gr}(-1) \boxtimes \CU(-1) \to V/\CTaut_2 \boxtimes \CO_Y(-1) \to \CO_{\Gr} \boxtimes \CO_Y \to \CO_{\UCubic} \to 0
\]

One can also obtain the above resolution by using the relative Beilinson spectral sequence with respect to the collection $\langle \CO_Y(-1), \CU(-1), \CO_Y, \CU \rangle $.
\end{remark}

\section{More on the geometry of $Y$}\label{sec:More on the geometry of Y}

In this section we rely on the results of section \ref{sec:The setting} to investigate further the geometry of $Y$.
In \ref{sec:embeddings of P(A)} we construct an embedding of $\PP(A)$ in $\Gr(3,V)$ and some exact sequences related to it.
In \ref{sec:Identification} we show that the projectivization $\PP_Y(\CU)$ is naturally isomorphic to a blow up of $\PP(V)$.
This leads to the construction of a family of linear sections of $Y$ parametrized by $\Gr(3,V)$, which we describe in \ref{sec:A map from Gr(3,V)}.
Finally, in \ref{sec:Linear sections of Y} we study a Fourier--Mukai transform in relation to the family of linear sections which we constructed.

\subsection{Embeddings of $\PP(A)$}\label{sec:embeddings of P(A)}

First, recall that by definition \ref{def:sigma} and by remark \ref{rmk:sigma embedding}, the map $\sigma$ is an embedding of $\PP(A)$ in $\PP(V)$ associating to a form $a \in A$ its unique kernel vector $a \wedge a \in \Lambda^4V^* \cong V$. 
The image of $\PP(A)$ via $\sigma$ is denoted by $\Imsigma$. 

We can also embed $\PP(A)$ in $\Gr(3,V)$ by mapping $a \in A$ to the kernel of the natural composition
\[
V \xrightarrow{\quad \,} A^* \otimes V^* \xrightarrow{\ker a} A^*
\]
which we will denote by $A(\ker(a),-)^\perp$.
We denote this embedding by $\korth$ and the tautological bundle on $\Gr(3,V)$ by $\CK$. 
The proof that $\korth$ is an embedding is part of proposition \ref{prop:embed P(A)} and is carried out with the help of lemma \ref{lemma:deg korth}. The image of $\PP(A)$ via $\kappa$ is denoted by $\Imkappa$.

\begin{lemma}\label{lemma:deg korth}
Let $\CK$ be the tautological bundle on $\Gr(3,V)$. There is an exact sequence
\begin{equation}\label{eq:deg korth}
0 \to \korth^* \CK \to V \otimes \CO_{\PP(A)} \to A \otimes \CO_{\PP(A)}(2) \to \CO_{\PP(A)}(3) \to 0
\end{equation}
The degree of $\korth$ is 3.
\end{lemma}
\begin{proof}
Define the middle map as the dual of the composition
\[
A \otimes \CO_{\PP(A)}(-2) \xrightarrow{A \otimes \sigma} A \otimes V \otimes \CO_{\PP(A)} \xrightarrow{\ev} V^* \otimes \CO_{\PP(A)}
\]
Using the fact \ref{lemma:HPD of Gr} that each form in $A$ has rank 4, the exactness of
\[
0 \to \CO_{\PP(A)}(-3) \to A \otimes \CO_{\PP(A)}(-2) \to  V^* \otimes \CO_{\PP(A)}
\]
can be checked pointwise, showing at the same time that the cokernel of
\[
A \otimes \CO_{\PP(A)}(-2) \to  V^* \otimes \CO_{\PP(A)}
\]
is a vector bundle.
The rank $3$ quotient bundle of $V^* \otimes \CO_{\PP(A)}$ induces the map $\korth$ from $\PP(A)$ to $\Gr(3,V)$.
The degree of $\korth$ is the degree of $-c_1(\korth^* \CK)$: 
by exact sequence \eqref{eq:deg korth} we find $\deg(\korth) = 3$.
\end{proof}

Under the canonical identification of  $\Gr(3,V)$ with $\Gr(2,V^*)$, which is the space of pencils of conics in $Y_5$, this embedding maps $a \in \PP(A)$ into the unique pencil of conics sharing the line $L_a$ as a component. 
\begin{lemma}\label{lemma:conics through L}
The scheme of conics containing a line $L := L_a$ is the line in $\PP(V^*)$ corresponding to $\kappa(a) \in \Gr(2,V^*)$.
\end{lemma}
\begin{proof}
Let $L_a$ be the line corresponding to $a \in A$ and $C_w$ the conic corresponding to $w \in V^*$ under the identifications of proposition \ref{prop:fano of lines} and \ref{prop:universal conic}.
By corollary \ref{cor:lines through P and Q}, $L_a$ embeds in $\Gr(2,V)$ as $\PP(\kappa(a)/\ker a)$, while by proposition \ref{prop:universal conic} for any conic $C_w$ we have
\[
C_w = \PP(A^\perp) \cap \Gr(2,\ker w) \subset \PP(\Lambda^2V)
\]
It follows that $L_a$ is contained in $C_w$ if and only if $\kappa(a) \subset \ker w$, which means that $w$ belongs to the line in $\Gr(2,V^*)$ corresponding to the 3-dimensional space $\kappa(a) \subset V$.

The schematic structure on the scheme of conics containing $L$ is the reduced one as by corollary \ref{cor:conics through P} it is the intersection of projective planes in $\PP(V)$.
\end{proof}

Lemma \ref{lemma:Imkappa is unexpected dimension} also provides us with another description of $\Imkappa \subset \Gr(3,V)$.
Note that a general projective plane $\PP(K)$ in $\PP(V)$ intersects $\Imkappa$ in a finite number of points. 
The locus in $\Gr(3,V)$ where this property fails is $\Imkappa$.
\begin{lemma}\label{lemma:Imkappa is unexpected dimension}
The intersection $\PP(K) \cap \Imsigma$ inside $\PP(V)$ is not of expected dimension if and only if $K \in \Imkappa \subset \Gr(3,V)$.
\end{lemma}
\begin{proof}
Restrict the exact sequence
\[
K^\perp \otimes \CO_{\PP(V)}(-1) \to \CO_{\PP(V)} \to \CO_{\PP(K)} \to 0
\]
 to $\PP(A)$ via $\sigma$ to find
\begin{equation}\label{eq:Imkappa is unexpected dimension}
K^\perp \otimes \CO_{\PP(A)}(-2) \to \CO_{\PP(A)} \to \CO_{\PP(K) \cap \Imsigma} \to 0
\end{equation}

The intersection $\PP(K) \cap \Imsigma$ is not of expected dimension if and only if all conics in the pencil given by $K^\perp$ share a common component $L$.
This on the other hand happens if and only if the map \eqref{eq:Imkappa is unexpected dimension} factors as
\[
K^\perp \otimes \CO_{\PP(A)}(-2) \to \CO_{\PP(A)}(-1) \xrightarrow{a^*} \CO_{\PP(A)}
\]
that is to say that there is a commutative diagram
\begin{equation}
\begin{diagram}
K^\perp		& \rTo		& V^*	 	 	\\
\dTo			&			& \dTo	 	 	\\
A^*			& \rTo^{\otimes a^*}	& S^2A^*	 
\end{diagram}
\end{equation}

Note that $A^* \otimes a^*$ is never contained in $V^*$, as $V^*$ is cut inside $S^2A^*$ by the $SL_2$-equivariant quadric $q \in S^2A$ and $q$ is non-degenerate. 
It follows that there is a regular map from $\PP(A)$ to $\Gr(2,V^*)$ which sends $a^*$ to the fiber product of
\[
A^* \xrightarrow{\otimes a^*} S^2A^*
\]
and of the canonical injection $V^* \to S^2A^*$.
Moreover, the locus of $K^\perp$ having $1$-dimensional intersection with $\Imsigma$ is the image of $\PP(A)$ under that map, so that it is at most $2$-dimensional and it is irreducible.

We will prove that all planes in $\PP(V)$ corresponding to points of $\Imkappa$ have $1$-dimensional intersection with $\Imsigma$.
Given a conic $C$, the set of lines which intersect it is a conic in $\PP(A)$.
More precisely, by lemma \ref{lemma:ILC}, the $SL_2$-equivariant linear map from $V^*$ to $S^2A^*$ associates with a conic $C$ the set of lines which intersect it, so that in order to show that $\PP(K) \cap \PP(A)$ is not finite it is enough to show that there are infinitely many lines intersecting all conics of the pencil $K^\perp \subset V^*$.
But for $K \in \Imkappa$ this last fact is true, as by lemma \ref{lemma:conics through L} all conics in the pencil contain a fixed line, which by proposition \ref{prop:intersecting lines} intersects infinitely many other lines.

Finally, as $\Imkappa$ has dimension $2$, it is the whole locus of planes which intersect $\Imsigma$ in a $1$-dimensional scheme. 
\end{proof}
\begin{remark}\label{rmk:Imkappa is unexpected dimension}
What we have actually proved is that if $K \in \Imkappa$, then the intersection 
\[
\PP(K) \cap \Imsigma \subset \PP(A)
\] 
is the extension of the structure sheaf of a line by that of a point. 
Furthermore, one can check that the point and the line are polar with respect to the intersection quadric $\CQ$.
\end{remark}

Denote by $\ILC$ the subvariety of $\PP(A) \times \PP(V^*)$ which is the closure of the locus of pairs $L$, $C$ where $L$ is a line and $C$ is a conic intersecting $L$ but not containing it.
\begin{lemma}\label{lemma:ILC}
The variety 
\[
\ILC \subset \PP(A) \times \PP(V^*)
\] 
of intersecting pairs of a line and a conic is cut by the $SL_2$-equivariant
\[
\CO_{\PP(A)}(-2) \boxtimes \CO_{\PP(V^*)}(-1) \to \CO_{\PP(A)} \boxtimes \CO_{\PP(V^*)}
\]
Moreover, $\ILC$ is irreducible.
\end{lemma}
\begin{proof}
Given $a \in A$ and $w \in V^*$ let $L_a$ be the corresponding line and $C_w$ the corresponding conic under the identifications of propositions \ref{prop:fano of lines} and \ref{prop:universal conic}.
Arguing as in lemma \ref{lemma:conics through L} one gets to
\[
L_a \cap C_w = \PP(\kappa(a) / \ker a) \cap \Gr(2,\ker w) \subset \PP(\Lambda^2 V)
\]
As the intersection of $\kappa(a)$ and $\ker w$ inside $V$ is always at least 2-dimensional, we find that $L_a$ intersects $C_w$ if and only if $\ker a \in \ker w$.

This last condition cuts a divisor inside $\PP(A) \times \PP(V^*)$ which is isomorphic to the following zero locus.
Restrict the tautological
\[
\CO_{\PP(V)}(-1) \boxtimes \CO_{\PP(V^*)}(-1) \to \CO_{\PP(V)} \boxtimes \CO_{\PP(V^*)}
\]
to $\PP(A) \times \PP(V)$ via $(\sigma \times \Id)$ in order to find an $\SSL_2$ invariant divisor cut by
\begin{equation}\label{eq:ILC bis}
\CO_{\PP(A)}(-2) \boxtimes \CO_{\PP(V^*)}(-1) \to \CO_{\PP(V)} \boxtimes \CO_{\PP(V^*)}
\end{equation}
The above map vanishes at $(a, w)$ if and only if $w(\ker a) = 0$.

Note that over $\PP(A)$ the zero locus of \eqref{eq:ILC bis} is the projectivization of the rank $4$ vector bundle given by the kernel of the composition
\[
V^* \otimes \CO_{\PP(A)} \to S^2A^* \otimes \CO_{\PP(A)} \to \CO_{\PP(A)}(2)
\]
so that it is clearly irreducible.
It follows that the zero locus of \eqref{eq:ILC bis} is the variety $\ILC$ of intersecting pairs of lines and conics.
\end{proof}

We can also map $\PP(A)$ into $\Gr(3,V)$ by sending $a \in A$ to the embedded tangent space $T_{\sigma(a)} \Imsigma$. We denote this map by $\tau$ and the image of $\PP(A)$ via $\tau$ by $\Imtau$. 

\begin{proposition}\label{prop:embed P(A)}
The map $\korth$ is an embedding. 
There are exact sequences
\begin{equation}\label{eq:embed P(A) zero}
0 \to A \otimes \CO_{\PP(V)}(-3) \to \Omega_{\PP(V)}(-1) \to \CO_{\PP(V)} \to  \CO_{\Imsigma} \to 0
\end{equation}
\begin{equation}\label{eq:embed P(A)}
0 \to \CO_{\PP(V)}(-1) \oplus \CO_{\PP(V)}(-2) \to V \otimes \CO_{\PP(V)} \to A^* \otimes \CO_{\PP(V)}(1) \to \sigma_{*}\CO_{\PP(A)}(3) \to 0
\end{equation}
and
\begin{equation} \label{eq:embed P(A) bis}
A \otimes \CK \xrightarrow{d_1} V^* \otimes \CO_{\Gr} \xrightarrow{} \korth_{*} \CO_{\PP(A)}(2) \to 0
\end{equation}
\end{proposition}

\begin{proof}
First, note that the exact sequence \eqref{eq:embed P(A)} can also be written as
\[
0 \to \CO_{\PP(V)}(-2) \to \CT_{\PP(V)}(-1) \to A^* \otimes \CO_{\PP(V)}(1) \to \sigma_{*}\CO_{\PP(A)}(3) \to 0
\] 
The above sequence can be obtained dualizing and twisting \eqref{eq:embed P(A) zero} as the dualizing sheaf of $\sigma$ is the determinant of the normal bundle $\CN_{\Imsigma / \PP(V)}$, which has odd degree.
It follows that in order to prove \eqref{eq:embed P(A)} it is enough to prove \eqref{eq:embed P(A) zero}.

We will prove the exactness of \eqref{eq:embed P(A)} by decomposing the ideal $\CI_{\Imsigma}$ of $\Imsigma$ inside $\PP(V)$ with respect to the full exceptional collection
\[
\langle	\CO_{\PP(V)}(-3), \Omega_{\PP(V)}(-1), \CO_{\PP(V)}(-2), \CO_{\PP(V)}, \CO_{\PP(V)}(1) \rangle
\]
As both $\CO_{\PP(V)}$ and $\CO_{\Imsigma}$ lie in the left orthogonal of $\CO_{\PP(V)}(1)$, so does $\CI_{\Imsigma}$.
It is also clear that the ideal of $\Imsigma$ is in the left orthogonal of $\CO_{\PP(V)}$.

In order to check that $\CI_{\Imsigma}$ is in the left orthogonal of $\CO_{\PP(V)}$ we can argue as follows. 
Assume there is a quadric in $\PP(V)$ containing $\Imsigma$, then its pullback to $\PP(S^2A)$ is a quadric containing the image of $\Ver_2$. 
As the ideal of $\Ver_2$ does not contain quadrics, so does that of $\Imsigma$.
Moreover, the defining exact sequence for $\Imsigma$ shows that $h^{>1}(\Imsigma(2))$ vanishes and that $h^1 \neq 0$ if and only if $h^0 \neq 0$, so that we have proved that $H^\bullet(\Imsigma(2)) = 0$.

Summing up we have proved that $\CI_{\Imsigma}$ belongs to 
\[
\langle 	\CO_{\PP(V)}(-3), \Omega_{\PP(V)}(-1)	\rangle
\]
By Beilinson spectral sequence it follows that $\CO_{\PP(V)}(-3)$ and $\Omega_{\PP(V)}(-1)$ can contribute non-trivially only in one degree.
As their Chern characters are independent this leaves the unique possibility of a resolution
\[
0 \to A \otimes \CO_{\PP(V)}(-3) \to \Omega_{\PP(V)}(-1) \to \CI_{\Imsigma} \to 0
\]
which can be extended to \eqref{eq:embed P(A) zero}.

The next part of the statement is that $\kappa$ is an embedding.
First, note that as $\sigma$ is an embedding also the product map $(\kappa, \sigma)$ from $\PP(A)$ to $\Gr(3,V) \times \PP(V)$ is an embedding.
We are now going to prove that the image of $(\kappa, \sigma)$ is the zero locus of the composition
\begin{equation}\label{eq:embed P(A) 2}
A \otimes \CK \boxtimes \CO_{\PP(V)}(-1) \to A \otimes V \otimes \CO_{\Gr(3,V)} \boxtimes V \otimes \CO_{\PP(V)} \to \CO_{\Gr(3,V)} \boxtimes \CO_{\PP(V)}
\end{equation}
of the tautological subbundles of $\Gr(3,V)$ and $\PP(V)$ with the natural pairing between $A$ and $V \otimes V$.
To prove it, we are going to show that the natural projection from $\Gr(3,V) \times \PP(V)$ to $\PP(V)$ induces an isomorphism between the zero locus $Z$ of \eqref{eq:embed P(A) 2} and the degeneracy locus $D$ of the $\SSL_2$-equivariant
\begin{equation}\label{eq:embed P(A) 12}
V \otimes \CO_{\PP(V)} \to V \otimes V \otimes \CO_{\PP(V)}(1) \to A^* \otimes \CO_{\PP(V)}(1)
\end{equation}

The map from $Z$ to $D$ is constructed as follows.
By definition of $Z$, the commutative diagram
\begin{equation}
\begin{diagram}
\CK	 \boxtimes \CO_{\PP(V)}				&		&		\\
\dInto								&\rdTo	&		\\
V \otimes \CO_{\Gr} \boxtimes \CO_{\PP(V)}	&\rTo		& A^* \otimes \CO_{\Gr} \boxtimes \CO_{\PP(V)}(1)
\end{diagram}
\end{equation}
restricts to
\begin{equation}
\begin{diagram}
\restr{\CK}{Z}		&		&		\\
\dInto			&\rdTo^{0}	&		\\
V \otimes \CO_Z	&\rTo		& A^* \otimes \restr{\CO_{\PP(V)}(1)}{Z}
\end{diagram}
\end{equation}
so that we have a rank $3$ subbundle of the kernel of
\[
V \otimes \CO_Z \to A^* \otimes \restr{\CO_{\PP(V)}(1)}{Z}
\]
showing that the map from $Z$ to $\PP(V)$ factors via $D$.

The map from $D$ to $Z$ is constructed as follows.
First note that over $D$ there is an exact sequence
\[
0 \to \Ker_D \to V \otimes \CO_D \to A^* \otimes \CO_D(1) \to \Coker_D \to 0
\]
where $\CO_D(1)$ is the restriction of $\CO_{\PP(V)}(1)$. 
Note moreover that the rank of \eqref{eq:embed P(A) 12} never drops by $2$ as the map $\sigma$ is injective. 
It follows that $\Coker_D$ is a line bundle on $D$, so that $\Ker_D$ is a rank 3 subbundle of $V \otimes \CO_D$.
We now prove that the induced map from $D$ to $\Gr(3,V) \times \PP(V)$ factors via $Z$.
Over $\Gr(3,V) \times \PP(V)$ there is a diagram 

\begin{equation}\label{eq:embed P(A) 4}
\begin{diagram}
\CK \boxtimes \CO_{\PP(V)}	&	\rTo	&V \otimes \CO_{\Gr}\boxtimes \CO_{\PP(V)}	\\
						&\rdTo	&\dTo								\\
						&		&A^* \otimes \CO_{\Gr} \boxtimes \CO_{\PP(V)}(1)	
\end{diagram}
\end{equation}
encoding both the tautological injection of $\Gr(3,V)$ and that of $\PP(V)$. 
By definition of the map from $D$ to $\Gr(3,V) \times \PP(V)$, the diagram \eqref{eq:embed P(A) 4} restricts as
\begin{equation} \label{eq:embed P(A) 5}
\begin{diagram}
\Ker_D	&	\rTo	&V \otimes \CO_D				\\
		&\rdTo^{0}	&\dTo						\\
		&		&A^* \otimes \CO_D(1)	
\end{diagram}
\end{equation}
with the diagonal map from $\Ker_D$ to $A^* \otimes \CO_D(1)$ being the zero map.
By definition of $Z$, we have a map from $D$ to $Z$.

The two maps between $D$ and $Z$ that we have just constructed are mutually inverse, as we are about to prove.
There are commutative diagrams \eqref{eq:embed P(A) 4} (restricted to $Z$) and \eqref{eq:embed P(A) 5}
which, under the two maps which we have just defined, pull back one to the other. 
As a consequence the map $Z \to D \to Z$ pulls the tautological injection of $\CK \subset V \otimes \CO_Z$ back to itself, so that the is $\Id_Z$. 
Analogously, the map $D \to Z \to D$ pulls the universal $\Ker_D$ back to itself, so that it is $\Id_D$.

Now that we know that the schematic zero-locus of \eqref{eq:embed P(A) 2} is $\PP(A)$, we can pushforward $\Kosz(A \otimes \CK \boxtimes \CO_{\PP(V)}(-1))$ to $\Gr(3,V)$ using proposition \ref{prop:unexpected Koszul}.
Note that, by lemma \ref{lemma:closed embedding and pushforward}, in order to prove that $\PP(A)$ embeds into $\Gr(3,V)$ via $\kappa$ it is enough to check that $\kappa$ has relative dimension $0$ and that $\CO_{\Gr(3,V)} \to \kappa_*\CO_{\PP(A)}$ is surjective.

That $\kappa$ is a finite map is clear for the following reason. 
As $\PP(A)$ is proper, it is enough to show that $\kappa$ has finite fibers.
If this is not the case, $\kappa$ contracts a divisor, which implies in turn that $\kappa$ is a constant map.
As $\ker a \in \kappa(a)$, if $\kappa(a)$ is constant then $\Imsigma \subset \kappa(a)$, which is impossible.


The first spectral sequence for the derived pushforward of a complex has second page
\[
R^i\pi_{\Gr*}\left(\Lambda^{-j}\left(A \otimes \CK \boxtimes \CO_{\PP(V)}(-1)\right) \right)
\]
It might have several non-vanishing entries, but the only ones interfering with the cohomology in degree 0 are those with $-j = i-1, i, i+1$.
As $\dim \PP(V) = 4$, $i$ ranges from 0 to 4. 
As $\CO_{\PP(V)}(-i)$ is acyclic for $i \in [1,4]$, the cohomology in degree 0 is the cokernel of
\[
\Lambda^{5}\left(A \otimes \CK \right) \xrightarrow{d_5} \CO_{\Gr}
\]


Note that the composition \eqref{eq:embed P(A) 2} cuts a $2$-dimensional scheme inside the $10$-dimensional scheme $\Gr(3,V) \times \PP(V)$.
It follows by proposition \ref{prop:unexpected Koszul} and by the fact that a smooth subvariety of a smooth variety is a local complete intersection that $\Kosz(A \otimes \CK \boxtimes \CO_{\PP(V)}(-1))$ has non vanishing cohomology only in degrees $-1$ and $0$.
If we denote by $\scrL_{\PP(A)}$ the excess bundle from proposition \ref{prop:unexpected Koszul}, then the second page of the other spectral sequence for the pushforward of $\Kosz(A \otimes \CK \boxtimes \CO_{\PP(V)}(-1))$ is
\[
R^i\pi_{\Gr*}\left( \scrL_{\PP(A)} \right) \qquad R^i\pi_{\Gr*}\left(\CO_{\PP(A)} \right)
\]
As we have already pointed out in this proof, $\korth$ is finite, so that the only non-vanishing entries are
\[
\korth_{*} \left( \scrL_{\PP(A)} \right) \qquad \korth_{*} \left(\CO_{\PP(A)} \right)
\]
with the spectral sequence degenerating at this page.
As a consequence the sequence
\[
\Lambda^{5}\left(A \otimes \CK \right) \xrightarrow{d_5} \CO_{\Gr} \xrightarrow{} \korth_{*}\CO_{\PP(A)} \to 0
\]
is exact and $\korth$ is an embedding.

The last part of the statement is the sequence \eqref{eq:embed P(A) bis}.
Take the Koszul complex of \eqref{eq:embed P(A) 2}, twist it by $\CO_{\PP(V)}(1)$ and push it forward to $\PP(V)$.
As in the previous step of the proof, we will find two spectral sequences: by comparing them and by K\"unneth formula we find an exact sequence
\[
A \otimes \CK \xrightarrow{d_1} V^* \otimes \CO_{\Gr} \xrightarrow{} \korth_{*} \CO_{\PP(A)}(2) \to 0
\]
which finally proves \eqref{eq:embed P(A) bis}.

\end{proof}



\subsection{Identification with blow up}\label{sec:Identification}

The map $\calU \to V \otimes \calO_Y$ induces a diagram
\begin{diagram}
	&			&\PP_Y(\calU)	&			&\\
	&   \ldTo^{p_Y}	&			&\rdTo^{p_V}	&\\
Y	&			&			&			& \PP(V)
\end{diagram}
In the next lemma we will prove that the above diagram is isomorphic to another natural correspondence between $Y$ and $\PP(V)$. 
\begin{proposition} \label{identification with blow up}
There is an isomorphism
\[
f: \PP_Y(\calU) \to \Bl_{\Imsigma}\PP(V)
\]
commuting with the natural projections to $\PP(V)$.
\end{proposition}
\begin{remark}
As a consequence of lemma \ref{identification with blow up}, from now on we will denote also the projections in
\begin{equation}\label{blowup}
\begin{diagram}
	&			&\Bl_{\Imsigma}\PP(V)	&			&\\
	&   \ldTo^{p_Y}	&					&\rdTo^{p_V}	&\\
Y	&			&					&			& \PP(V)
\end{diagram}
\end{equation}
by $p_Y$ and $p_V$.
\end{remark}
\begin{proof}
Over the ambient Grassmannian $\Gr(2,V)$ there are natural identifications
\[
\PP_{\Gr(2,V)}(\CU) \cong \Flag(1,2,V) \cong \PP_{\PP(V)}(\CT_{\PP(V)}(-2))
\]
relative to $\Gr(2,V)$.
The last space is by definition
\[
\Proj_{\PP(V)} \left( \Sym^\bullet (\Omega_{\PP(V)}(2)) \right)
\]
When we restrict to $Y$, we are intersecting the above space with
\[
\Proj_{\PP(V)} \left( \Sym^\bullet (A \otimes \CO_{\PP(V)}) \right)
\]
inside $\PP(V) \times \PP(\Lambda^2V)$.

The result is the projectivization of the symmetric algebra of
\begin{equation}\label{eq:alternate blowup}
\coker \left( A \otimes \CO_{\PP(V)} \to \Omega_{\PP(V)}(2) \right)
\end{equation}
By exact sequence \eqref{eq:embed P(A) zero} and $\SSL_2$-equivariance of the whole construction, the cokernel \eqref{eq:alternate blowup} is isomorphic to a twist of the ideal of $\Imsigma$ in $\PP(V)$, so that its $\Proj$ is by definition the blowup of $\PP(V)$ in $\Imsigma$.
\end{proof}

\begin{definition}
Denote by $H$ the pullback of the ample generator from $Y$, by $h$ the pullback of the ample generator from $\PP(V)$, and by $E$ the exceptional divisor of the blow-up.
\end{definition}
Since the Picard number of $\PP_Y(\calU)$ is 2, we look for a relation between $H$, $h$ and $E$.
We will use the relative tautological sequence
\begin{equation}\label{eq:relative tautological p_Y}
0 \to \calO(-h) \to p_Y^*\calU \to \CO(H - h) \to 0
\end{equation}
on $\PP_Y(\calU)$.
\begin{proposition} \label{picard of blow up}
The relative canonical class of $\PP_Y(\calU)$ over $Y$ is $\calO_{\PP_{Y}(\CU)}(h - E)$. Moreover, 
\begin{equation}\label{eq:picard of blow up}
H = 3h - E
\end{equation} 
\end{proposition}
\begin{proof}
Compute the canonical class of $\PP_Y(\calU)$ in two different ways. On one hand, using the blow-up description, it is $-5h + E$.
On the other hand, on $\PP_Y(\calU)$ there's a tautological injection
\begin{equation}
0 \to \calO_{\PP_{Y}(\CU)}(-h) \to p_Y^*\calU
\end{equation}
which gives as relative cotangent bundle $\omega_{p_Y} = \calO_{\PP_{Y}(\CU)}(H - 2h)$. Together with $\omega_Y = -2H$ this gives $\omega_{\PP(\calU)} = - H -2h$, so that in the end $H = 3h - E$.
\end{proof}

We can use propositions \ref{identification with blow up} and \ref{picard of blow up} to describe how lines transform via the correspondence $\PP_Y(\calU)$.
\begin{proposition} \label{secants to P(A)}
Let $L$ be a line in $\PP(V)$, $\widetilde{L}$ be its strict transform in $\PP(\calU)$ and $L^+ = p_Y(\widetilde{L})$. Then
\[
3 - \ell (\Imsigma \cap L) = \deg(L^+)
\]
where $\ell$ denotes the length of the structure sheaf of a scheme.
\end{proposition}
\begin{proof}
First, note that it makes sense to talk about length, as there is no line inside $\PP(A)$. Then, by projection formula $\deg(L^+) = H \cdot \widetilde{L}$. Moreover $H = 3h - E$ by equation \eqref{eq:picard of blow up}, and again by projection formula $l(\Imsigma \cap L) = l(E \cap \widetilde{L})$.
\end{proof}


Given a point $y \in Y$, its transform $p_Vp_Y^{-1}(y)$ in $\PP(V)$ is a line.
More precisely, the projective bundle $\PP_Y(\CU)$ gives a family of lines in $\PP(V)$ parametrized by $Y$.
By proposition \ref{secants to P(A)} all lines $p_Vp_Y^{-1}(y)$ are trisecants to $\Imsigma$, so that we choose the following notation for the induced map from $Y$ to $\Gr(2,V)$:
\[
\begin{diagram}
\tr: 	&Y 	&\rTo 		&\Tr \subset \Gr(2,V) \\
	&y 	&\rMapsto 	&p_Vp_Y^{-1}(y)
\end{diagram}
\]
where $\Tr$ is the space of trisecant lines to $\Imsigma \subset \PP(V)$. 
The following corollary was already known to Castelnuovo \cite{castelnuovo1891ricerche}.

\begin{corollary}\label{cor:points and trisecants}
The map $\tr$ is an isomorphism from $Y$ to the space $\Tr$ of lines in $\PP(V)$ which are trisecant to $\Imsigma$.
\end{corollary}
\begin{proof}
The map $\tr$ is defined by the family $\PP_Y(\CU)$, so that it is regular.
For a line $L$ in $\PP(V)$, let $\widetilde{L}$ be its strict transform in $\Bl_{\Imsigma}(\PP(V))$.
Define the inverse of $\tr$ as
\[
\begin{diagram}
\tr^{-1}: 	&\Tr	& \rTo		& Y	\\
		&L 	& \rMapsto	& p_Y(\widetilde{L})
\end{diagram}
\]
Note that $\tr^{-1}$ is well defined on $\Tr$ by proposition \ref{secants to P(A)}.
Also $\tr^{-1}$ can be defined in families, and it is easy to check that the two morphisms are inverse to each other.
\end{proof}


\subsection{A map from $\Gr(3,V)$ to $\PP(\Lambda^2V^*/A)$}\label{sec:A map from Gr(3,V)}

Given $\PP(K) \subset \PP(V)$, $\dim(K) = 3$, there are two natural ways to obtain a point in $\PP(\Lambda^2V^*/A)$. 
We define a map $\alpha$ in \eqref{first map from Gr to P} and we prove in proposition \ref{transform of P(K)} that there is another description for it.


\begin{definition}
The map $\alpha$ is the composition of
\begin{equation} \label{first map from Gr to P}
\begin{diagram} 
\Gr(3,V)	& \rTo^\cong	& \Gr(2,V^*)	& \rInto	&\PP(\Lambda^2V^*)	& \rDashto	&\PP(\Lambda^2V^*/A)
\end{diagram}
\end{equation}
where the second map is the Pl{\"u}cker embedding and the third is the linear projection from $\PP(A)$. 
\end{definition} 
The map $\alpha$ is regular on $\Gr(2,V^*)$, as $\PP(A) \cap \Gr(2,V^*) = \emptyset$ by lemma \ref{lemma:HPD of Gr}. 
It provides us with a family of linear sections of $Y$ parametrized by $\Gr(3,V)$.
As we will often use the total preimage $p_V^{-1}(\PP(K))$, we introduce the following notation.
%
%
\begin{definition}\label{def:universal S_K}
Recall that $\CK$ denotes the tautological bundle over $\Gr(3,V)$.
The scheme $\CS$ is the fiber product
\begin{equation}\label{diag:transform of P(K)}
\begin{diagram}
\CS				& \rInto	& \Gr(3,V) \times \PP_Y(\calU) 		\\
\dTo				&\square	& \dTo_{\Id_{\Gr} \times p_V}		\\
\PP_{\Gr(3,V)}(\CK)	& \rInto	& \Gr(3,V) \times \PP(V)
\end{diagram}
\end{equation}
We denote the fiber of $\CS$ over a point $K \subset V$ of $\Gr(3,V)$ by $S_K$, and its image $p_Y(S_K)$ by $Y_K$.
\end{definition}
\begin{remark}\label{rmk:universal S_K notation}
We will see in proposition \ref{transform of P(K)} that $Y_K$ is a linear section of $Y$ and that, under the canonical identification of the space of linear sections of $Y$ with $\PP(\Lambda^2V^*/A)$, $Y_K$ corresponds to $\alpha(K)$.
\end{remark}
Denote by $\Phi_{\Bl}$ the Fourier--Mukai functor with kernel $\CO_{\Bl_{\Imsigma} \PP(V)}$, that is to say
\begin{equation}\label{definition Phi}
\Phi_{\Bl}: = Rp_{V*}Lp_Y^* :D^b(\PP(V)) \to D^b(Y_5)
\end{equation}
where $p_Y$ and $p_V$ are the projections from $\Bl_{\Imsigma} \PP(V)$ to $Y$ and $V$, as defined in diagram \eqref{blowup}.
For $K \subset V$ a $3$-dimensional subspace, proposition \ref{transform of P(K)} shows that $\Phi_{\Bl}(\calO_{\PP(K)})$ is the structure sheaf of a linear section of $Y$.

\begin{proposition} \label{transform of P(K)}
Let $K$ be a $3$-dimensional subspace of $V$. In the notation of definition \ref{def:universal S_K} and remark \ref{rmk:universal S_K notation}
\[
\Phi_{\Bl}(\calO_{\PP(K)}) = \CO_{Y_K}
\]
Moreover, $Y_K$ is the linear section of $Y$ corresponding to $\alpha(K)$ under the natural identification of the space of linear sections of $Y$ with $\PP(\Lambda^2 V^*/A)$.
\end{proposition}
\begin{proof}
First, note that the fiber product
\eqref{diag:transform of P(K)}
has expected dimension because for each $K$ we have $\PP(K) \not\subset \Imsigma$ and because the relative dimension of $p_V$ over $\PP(V)$ is $1$. 
It follows by lemma \ref{lemma:smooth expected dimension} that the diagram is $\Tor$-independent and that $Lp_V^*\calO_{\PP(\CK)} \cong  \calO_{\CS}$.

Now, let $\CK \subset V \otimes \CO_{\Gr(3,V)}$ be the tautological injection.
Over $\PP_{\Gr}(V \otimes \CO_{\Gr})$ there is a universal Koszul resolution
\[
\CO_{\PP(\CK)} \cong \left\{	\Lambda^2\CK^\perp(-2h) \to \CK^\perp(-h) \to \CO_{\PP(V) \times \Gr(3,V)}	\right\}
\]
for the structure sheaf of $\PP_{\Gr(3,V)}(\CK)$.

Pulling back via $\Id_{\Gr} \times p_V$ we get
\begin{equation}\label{Koszul for S}
\calO_{\CS} \cong \left\{ \Lambda^2\CK^\perp(-2h) \to \CK^\perp(-h) \to \CO_{\PP_Y(\CU) \times \Gr(3,V)} \right\}
\end{equation}

As $\calO(h)$ is the relative $\calO(1)$ of $p_Y$, we can use this resolution to compute $\Phi(\calO_{\PP(K)})$ as the limit of a spectral sequence. The only page with non vanishing differentials is
\begin{diagram}
\det(\CK^{\perp}) \boxtimes \calO_Y(-H)		& 				&	0	& 				&	0	\\
									&\rdTo(4,2)_{d_2}	&		&\qquad \quad		&		\\
0									&  				&	0	& 				& \calO_{\Gr(3,V) \times Y}
\end{diagram}
where in order to compute the entries we have used relative Serre duality for the projection from $\Gr(3,V)  \times \PP_Y(\CU)$ to $\Gr(3,V) \times Y$ and the fact that $\det(\CU) = \CO(-H)$.

Twist $d_2$ by $\CO_Y(H)$ and push it forward to $\Gr(3,V)$ to obtain
\begin{equation}\label{eq:transform of P(K)}
\det(\CK) \to \left( \Lambda^2V^*/A \right) \otimes \CO_{\Gr(3,V)}
\end{equation}
where we have substituted $\det(\CK) = \det(\CK^\perp)$ and $H^0(\CO_Y(1)) = (\Lambda^2 V^* / A)$.

We will now check that the quotient of \eqref{eq:transform of P(K)} is locally free by base-changing to points of $\Gr(3,V)$.
Let $K$ be a $3$-dimensional subspace of $V$, then diagram \eqref{diag:transform of P(K)} restricts to
\begin{equation}
\begin{diagram}
S		& \rInto	& \PP_Y(\calU) 		\\
\dTo		&\square	& \dTo_{p_V}		\\
\PP(K)	& \rInto	& \PP(V)
\end{diagram}
\end{equation}
Once more, the fiber product has expected dimension, so that $Lp_V^*\calO_{\PP(K)} \cong p_V^* \calO_{\PP(K)}$ and $p_V^* \calO_{\PP(K)} \cong \calO_{S}$.
As a consequence there is a quasi isomorphism between
\begin{equation}\label{eq:transform of P(K) 2}
\CO_Y(-H) \xrightarrow{\restr{d_2}{[K] \times Y}} \CO_Y
\end{equation}
and $Rp_{Y*}\CO_S$. 
If $d_2$ restricted to $[K] \times Y$ were $0$, then there would be a nontrivial $R^{-1}p_{Y*}\CO_S$.
It follows that $\restr{d_2}{[K] \times Y} \neq 0$ and therefore \eqref{eq:transform of P(K)} is an injection of vector bundles.

Finally, \eqref{eq:transform of P(K)} induces a regular map from $\Gr(3,V)$ to $\PP(\Lambda^2 V^* / A)$ which is $SL_2$-equivariant as $\Phi_{\Bl}$ is $SL_2$-equivariant.
We claim that this map is $\alpha$.
By $SL_2$-equivariance, it is enough that the degree of the map which associates $\Phi_{\Bl}(\CO_K)$ with $K$ is $1$.
This last fact is a straightforward consequence of \eqref{eq:transform of P(K)} and of $\det(\CK) = \CO_{\Gr(3,V)}(-1)$.
\end{proof}

We are now going to describe the preimage via $\alpha$ of the branching locus of $\alpha$ in $\PP(\Lambda^2 V^* / A)$.
If we think of $\PP(\Lambda^2 V^* / A)$ as of the space of hyperplane sections of $Y$, by projective duality the branching locus of $\alpha$ corresponds to singular hyperplane sections of $Y$.
It follows that the preimage via $\alpha$ of the branching locus of $\alpha$ parametrizes subspaces $K \subset V$ such that the transform $p_V^*p_{Y*} \CO_{\PP(K)}$ is the structure sheaf of a singular hyperplane sections.


Proposition \ref{prop:Dtri and Dfat} is stated in terms of two subschemes of $\Gr(2,V)$: we will denote them by $\Dtr$ and by $\Dfat$.
\begin{definition}
The closed subscheme $\Dtr \subset \Gr(3,V)$ is the locus of $K \subset V$ which contain a trisecant to $\Imsigma$.
The closed subscheme $\Dfat \subset \Gr(3,V)$ is the locus of $K \subset V$ such that the intersection with $\Imsigma$ has a non-trivial tangent space.
\end{definition}
Note that $\Dfat$ contains all $K \subset V$ such that the intersection is a non-reduced subscheme or is positive dimensional. 
We have already seen in lemma \ref{lemma:Imkappa is unexpected dimension} that $\Imkappa \subset \Dfat$. 
In the proof of \ref{lemma:Imkappa is unexpected dimension} we have seen that if $K \in \Imkappa$, then the intersection is the union of a conic and a point lying on the plane spanned by the conic: it follows that $\Imkappa \subset \Dtr$.

\begin{proposition}\label{prop:Dtri and Dfat}
The ramification locus of $\alpha$ is linearly equivalent to $\CO_{\Gr(3,V)}(2)$. 
$\Dtr$ and $\Dfat$ are irreducible divisors in $\Gr(3,V)$. Their union is the preimage via $\alpha$ of the branching locus.
\end{proposition}
\begin{proof}
The ramification divisor is linearly equivalent to the relative canonical bundle of $\alpha$.
As $\alpha$ is a linear projection we have
\[
\alpha^*\CO_{\PP(\Lambda^2 V^* / A)}(1) = \CO_{\Gr(3,V)}(1)
\]
It follows that the relative canonical bundle of $\alpha$ is $\CO_{\Gr(3,V)}(2)$.


Let $K \subset V$ be a vector subspace of dimension $3$.
The next step we are going to prove is the following: whenever $K \notin \Dtr \cup \Dfat$, the linear section of $Y$ corresponding to $\alpha(K)$ is regular.

By proposition \ref{transform of P(K)}, the structure sheaf of the linear section corresponding to $\alpha(K)$ is $\CO_{Y_K} = \Phi_{\Bl}(\CO_{\PP(K)})$.
By proposition \ref{picard of blow up}, the projection $p_Y$ from $\Bl_{\Imsigma}\PP(V)$ to $Y$ is given by the complete linear system $3h - E$.
As $K \notin \Dfat$, the total preimage $S_K$ of $\PP(K)$ under $p_V$ is the blow-up of $\PP(K)$ in $4$ distinct points. As $K \notin \Dtr$ any triple of blown up points does not lie on a line.
It is easy to check that in this case $3h - E$ is very ample on $S_K$. 
For example we can first check that $3h - E$ does not contract curves and after that, using $R p_{Y*} \CO_{S_K} = \Phi_{\Bl}(\CO_{\PP(K)}) = \CO_{Y_K}$, we can conclude that
\[
p_Y : S_K \to Y_K
\] 
is an isomorphism on its image by lemma \ref{lemma:closed embedding and pushforward}. 
It follows that the preimage via $\alpha$ of the branch locus is $\Dtr \cup \Dfat$.

Now we prove that $\Dtr$ and $\Dfat$ are distinct irreducible divisors in $\Gr(3,V)$.
First, we check that they are distinct. 

Given a point $y$ in the open orbit of $Y$, by corollaries \ref{cor:3 to 1 cover} and \ref{cor:points and trisecants} we have that $\tr(y)$ is a trisecant to $\Imsigma$ intersecting it in $3$ distinct points, corresponding to the $3$ distinct lines $L_1,L_2,L_3$ through $y$.
Call $K_{a,y}$ the plane spanned by $a$ and $\tr(y)$, where we choose $a \in \Imsigma$ such that $L_a$ does not intersect the $3$ lines $L_1, L_2,L_3$ through $y$ (i.e. the lines in $\PP(V)$ through $a$ and any point in $\tr(y) \cap \Imsigma$ are not trisecants).

If $\PP(K_{a,y})$ contains tangent vectors to $\Imsigma$, then its intersection with $\Imsigma$ has length at least $5$ and is therefore $1$-dimensional.
In this case remark \ref{rmk:Imkappa is unexpected dimension} says that 
$\PP(K_{a,y}) \cap \Imsigma$
is the union of a point and of a conic (in $\PP(K_{a,y})$).
But then any line (in $\PP(K_{a,y})$) through the point is a trisecant, against our choice of $a$ and $\tr(y) \cap \Imsigma$.

To show that $\Dtr$ is an irreducible divisor, construct the universal space $\WDtr$ of pairs given by a plane $\PP(K)$ in $\PP(V)$ and a trisecant to $\Imsigma$ contained in the plane.
As $Y$ is the space of trisecants to $\Imsigma$ (see corollary \ref{cor:points and trisecants}), there is a diagram
\[
\begin{diagram}
	&			&\WDtr	&			&\\
	&   \ldTo		&		&\rdTo^{p_{tri}}	&\\
Y	&			&		&			&\Dtr
\end{diagram}
\]
Moreover, for each trisecant line to $\Imsigma$ there is a $\PP^2$ of planes in $\PP(V)$ containing it.
More precisely, $\WDtr$ is $\PP_Y(V/\CU)$, which is clearly irreducible and $5$-dimensional.

Finally, the generic fiber of $p_{tri}$ is one point, as we are about to show.
Take a plane $\PP(K)$ in $\PP(V)$ intersecting $\Imsigma$ with expected dimension and containing $2$ trisecants.
Each one cuts a scheme of length $3$ inside $\PP(K) \cap \Imsigma$, which has length $4$.
It follows that the $2$ trisecants intersect in a scheme of length at least $2$ and that, being lines, they coincide.
As a consequence $\Dtr$ is $5$-dimensional and irreducible, as it is the image under the generically finite map $p_{tri}$ of a $5$-dimensional irreducible scheme.

As for $\Dfat$, construct the universal space $\WDfat$ of pairs given by a plane $\PP(K)$ in $\PP(V)$ and a tangent direction contained in the intersection $\PP(K) \cap \Imsigma$.
The space of tangent directions to $\Imsigma$ is the projectivization of the tangent bundle to $\PP(A)$, which we denote by  $\PP_{\PP(A)}(\CT_{\PP(A)})$, so that there is a diagram
\[
\begin{diagram}
	&			&\WDfat	&			&\\
	&\ldTo^{p_{fat}}	&		&\rdTo		&\\
\Dfat	&			&		&			& \PP_{\PP(A)}(\CT_{\PP(A)})
\end{diagram}
\]
The space $\WDfat$ is a $\PP^2$-fibration over $\PP_{\PP(A)}(\CT_{\PP(A)})$, so that it is $5$-dimensional and irreducible.

Moreover the fiber of $p_{fat}$ over the generic point of $\Dfat$ is finite for the following reason. 
Assume a plane $\PP(K)$ has 0-dimensional intersection with $\Imsigma$, then if $\PP(K)$ contains a $1$-dimensional family of tangent directions to $\Imsigma$ there is a point $a$ in $\PP(K) \cap \Imsigma$ such that $\PP(K)$ contains a $1$-dimensional family of tangent directions to $\Imsigma$ at $a$.
As $\Imsigma$ has dimension $2$ and is smooth, this is enough to deduce that $\PP(K)$ is the tangent space to $\Imsigma$ at $a$.
The locus of tangent planes to $\Imsigma$ and tangent directions is a $3$-dimensional subvariety of $\WDtr$, so that the generic point of $\WDfat$ is not contained in it.
It follows that $\Dfat$ is $5$-dimensional and irreducible as it is the image under the generically finite map $p_{fat}$ of a $5$-dimensional irreducible scheme.
\end{proof}
\begin{remark}
One can also check that the branch locus has degree $10$, and that the components of the fiber over its the generic point have multiplicities $(2,1,1,1)$, with the non-reduced fiber, which corresponds to a ramification point, belonging to $\Dtr$.
\end{remark}


We will use the description of the ramification locus of proposition \ref{prop:Dtri and Dfat} to prove the following result about smooth linear sections containing a fixed line.
\begin{lemma}
\label{lemma:sections through L}
For any line $L$ and for the generic plane $\PP(K) \subset \PP(V)$ containing $\sigma(L)$, the linear section $Y_K$ is smooth. Moreover, for any line $L$ in $Y$, the generic hyperplane section $Y_K \subset Y$ containing $L$ is smooth.
\end{lemma}
\begin{proof}
The second assertion clearly follows from the first as if $\sigma(L) \in \PP(K)$ then $L \subset Y_K$.

By proposition \ref{prop:Dtri and Dfat} the set of planes $\PP(K) \subset \PP(V)$ inducing singular sections of $Y$ is the union of $\Dtr$ and $\Dfat$.
For any $L$ there is a $4$-dimensional family of planes in $\PP(V)$ containing $\sigma(L)$, namely $\Gr(2, V / \sigma(L))$ and by smoothness of $\Imsigma$ the generic one is not in $\Dfat$.

We still need to prove that the generic $\PP(K)$ through $\sigma(L)$ does not contain any trisecant.
First, we prove that the generic bisecant through $\sigma(L)$ is not a trisecant.
This is true as trisecants through $\sigma(L)$ correspond via the maps $\tr$ and $\tr^{-1}$ of corollary \ref{cor:points and trisecants} to points of $L$, so that there is a one dimensional family of trisecants through $\sigma(L)$ (corresponding to points in $L$).

Now, choose two other points $L_1, L_2 \in \PP(A)$ such that the three bisecants connecting $\sigma(L)$, $\sigma(L_1)$ and $\sigma(L_2)$ are not trisecants and consider the plane $\PP(K)$ spanned by $\sigma(L)$, $\sigma(L_1)$ and $\sigma(L_2)$.
Then clearly $\PP(K) \notin \Dtr$.
\end{proof}

\begin{lemma}\label{lemma:Imkappa in Dtr}
$\Imkappa$ is contained in $\Dtr$.
\end{lemma}
\begin{proof}
By lemma \ref{lemma:Imkappa is unexpected dimension} if $K \in \Imkappa$ then the intersection of $\PP(K)$ with $\Imsigma$ is the union of a conic and a point (possibly lying on the conic).
It follows that any line in $\PP(K)$ passing through the extra point is a trisecant contained in $\PP(K)$.
\end{proof}

\subsection{Linear sections of $Y$ and the Fourier--Mukai transform $\Phi_{\Bl}$}\label{sec:Linear sections of Y}

In proposition \ref{transform of P(K)} we have given an interpretation of the map $\alpha$ in terms of a Fourier--Mukai transform functor.
In this section we use the Fourier--Mukai transform 
\[
\Phi_{\Bl}:= Rp_{Y*}Lp_V^*: D^b(\PP(V)) \to D^b(Y)
\] 
to relate the incidence properties of planes in $\PP(V)$ and $\Imsigma$ with the properties of certain sheaves supported on linear sections of $Y$.

Recall from definition \ref{def:universal S_K} that for any $K \in \Gr(3,V)$ we denote by $S_K$ the total preimage of $\PP(K)$ via $p_V$ and that we denote by $Y_K$ the image of $S_K$ via $p_Y$, as in the following diagram.
\begin{equation}\label{diag:P(K), S_K, Y_K}
\begin{diagram}
\PP(K)		&\lTo	^{p_K}	& S_K		& \rTo^{p_{Y_K}}	& Y_K 	\\
 \dInto^{i_K}	&\square		& \dInto^{i_S}	&				& \dInto^{i_{Y_K}}	\\
 \PP(V)		& \lTo_{p_V}	& \PP_Y(\CU)	& \rTo_{p_Y}		& Y
\end{diagram}
\end{equation}
We will often need to compute the cohomology of sheaves on $S_K$ which come by pullback from $\PP(K)$.
This can be easily done via projection formula, provided one has already computed $Rp_{K*}Lp_{K}^*\CO_{\PP(K)}$
\begin{lemma} \label{singular projection formula}
The projection $p_K: S_K \to \PP(K)$ is such that the canonical map 
\begin{equation}\label{eq:coev_K}
\CO_{\PP(K)} \to Rp_{K*}\CO_{S_K}
\end{equation}
is an isomorphism. 
\end{lemma}
\begin{proof}
The map $p_K$ is defined via the following cartesian diagram
\begin{equation}\label{eq:singular projection formula 7}
\begin{diagram}
S_K			& \rTo^{i_S} 		& \Bl_{\PP(A)}\PP(V) 	\\
\dTo^{p_K}	&				& \dTo_{p_V}			\\
\PP(K)		&\rTo	_{\quad i_K}	& \PP(V)
\end{diagram}
\end{equation}
The fiber product in the cartesian diagram \eqref{eq:singular projection formula 7} has expected dimension and $\PP(K)$, $\PP(V)$, $\PP_Y(\CU)$ are smooth.
It follows by lemma \ref{lemma:smooth expected dimension} that the diagram \eqref{eq:singular projection formula 7} is $\Tor$-independent, so that the base change map is an isomorphism.
As a consequence
\begin{equation}\label{eq:singular projection formula}
Rp_{K*}\CO_{S_K} \cong Rp_{K*}Li_S^* \CO_{\Bl_{\PP(A)}\PP(V)} \cong  Li_{K}^*Rp_{V*}\CO_{\Bl_{\PP(A)}\PP(V)}
\end{equation}
As the $Rp_{V*}\CO_{\Bl_{\PP(A)}\PP(V)}$ is $\CO_{\PP(V)}$, so that the right hand side of \eqref{eq:singular projection formula} is $\CO_{\PP(K)}$.
\end{proof}

The following lemma holds.
\begin{lemma}\label{FM transform}
For any $K \subset V$ 3-dimensional vector subspace, $\PhiK$ is a pure object concentrated in degree 0. Moreover it is equivalent to
\begin{equation}\label{transform of P(K)(2)}
\left\{ \calO_Y \to \calU^{*\oplus 2} \to S^2\calU^{*} \right\}
\end{equation}
\end{lemma}
\begin{proof}
By arguing as in \ref{transform of P(K)} we get
\[
p_V^*\calO_{\PP(K)}(2h) \cong \left\{ \calO_{\PP_Y(\CU)} \to \calO_{\PP_Y(\CU)}(h)^{\oplus 2} \to \calO_{\PP_Y(\CU)}(2h) \right\}
\]
As $\calO_{\PP_Y(\CU)}(h)$ is the relative $\calO(1)$ of $p_Y$, we can use this resolution to get $\Phi_{\Bl} \left(\calO_{\PP(K)}(2h)\right)$ as the limit of a spectral sequence. The page with horizontal differentials is the last one with non-trivial maps and is
\begin{diagram}
0		& \rTo	&	0			& \rTo &	0		\\
\calO_Y	& \rTo 	&	\calU^{*\oplus 2}	& \rTo & S^2\calU^*
\end{diagram}
as $p_Y$ has relative dimension 1. This shows that $R^1p_{Y*}\calO_{\PP(K)}(2h) = 0$ (i.e. that $\PhiK$ is equivalent to a sheaf) and that the sequence
\begin{equation}\label{resolution of Phi(K)}
0 \to \calO_Y \to \calU^{*\oplus 2} \to S^2\calU^{*} \to \PhiK \to 0
\end{equation}
is exact.
\end{proof}

We can use the resolution for $\PhiK$ to compute several things about it. 
Recall that by proposition \ref{transform of P(K)} with any $K \in \Gr(3,V)$ we associate a linear section $Y_K = \Phi_{\Bl}\CO_{\PP(K)}$ of $Y$.


\begin{lemma} \label{lemma:FM transform Tor}
The following facts about $\PhiK$ hold:
\begin{enumerate}[(a)]
\item $\ch \left(\PhiK\right) = H + 7L/2 - P/6$.
\item \label{item:FM transform Tor b} $\PhiK$ is the pushforward of the sheaf
$Rp_{Y_K*}p_K^*\CO_{\PP(K)}(2h)$ from $Y_K$.
%
\item $\PhiK$ has no associated points of codimension 3.
\item $\PhiK$ has associated points in codimension 2 if and only if $K \in \korth(\PP(A))$.
\item If $K \notin \korth(\PP(A))$ in $\Gr(3,V)$, then $\PhiK$ is torsion free on its support $Y_K$.  
\end{enumerate}
\end{lemma}
\begin{proof}
To check associated primes we check 
\[
\Tor_{p}\left(	\calO_Z, \PhiK	\right)
\]  
for $p = \cod(Z)$ and $Z$ any irreducible closed subset of $Y$. 
\begin{enumerate}[(a)]
\item Use additivity of the Chern character and formulas \eqref{Cherntable}. 
\item 
%
%
For our claim to hold, it is enough to check that the base change map for the leftmost square of diagram \eqref{diag:P(K), S_K, Y_K} is an isomorphism.
This follows from lemma \ref{lemma:smooth expected dimension} as $\PP_Y(\CU)$, $\PP(K)$ and $\PP(V)$ are smooth and $S_K$ has expected dimension. 
\item Any $\Tor_3(\PhiK, -)$ vanishes as resolution \eqref{transform of P(K)(2)} has length $2$.
\item Each map from $\CO_Y$ to $\CU^*$ vanishes on a conic, so the first map in \eqref{transform of P(K)(2)} vanishes on the intersection of a pencil of conics, which has codimension 2 if and only if all conics in the pencil are singular and share a component. By lemma \ref{lemma:conics through L} this implies that there is $\Tor_2$ only with $\calO_L$ and if and only if $K \in \korth (\PP(A))$.
\item The computation of $\ch(\PhiK)$ shows that $\Supp(\PhiK) = Y_S$. Moreover, as $K \notin \korth (\PP(A))$, all associated primes have codimension 1, so that the sheaf $\PhiK$ is torsion free on its support by definition.
\end{enumerate}
\end{proof}

The following lemma describes further the correspondence between properties of $\PP(K)$ in $\PP(V)$ and properties of its transform $\PhiK$.

\begin{lemma}\label{lemma:Dtr is common point}
For a plane $\PP(K) \subset \PP(V)$ and a point $P \in Y$ the following conditions are equivalent
\begin{enumerate}[(a)]
\item \label{item:Dtr is common point a} $\Tor_2 \left( \CO_P, \Phi_{\Bl}(\CO_{\PP(K)}(2)) \right) \neq 0$.
\item \label{item:Dtr is common point b0} The line $p_Y^{-1}(P)$ is contained in $S_K$.
\item \label{item:Dtr is common point b} The line $p_Vp_Y^{-1}(P)$ is contained in $\PP(K)$.
\item \label{item:Dtr is common point c} All the conics in the pencil corresponding to $K \in \Gr(3,V) \cong \Gr(2,V^*)$ contain $P$.
\end{enumerate}
\end{lemma}
\begin{proof}
Clearly, condition \eqref{item:Dtr is common point b0} is equivalent to condition \eqref{item:Dtr is common point b}. We will show that condition \eqref{item:Dtr is common point a} is equivalent to conditions \eqref{item:Dtr is common point b} and \eqref{item:Dtr is common point c}.


By adjunction formula
\[
\CO_P \otimes Rp_{Y*}p_V^*\CO_{\PP(K)}(2h) \cong Rp_{Y*} \left( \CO_{p_Y^{-1}(P)} \otimes p_V^* \CO_{\PP(K)}(2h) \right)
\]
The cohomology in degree $-2$ of the left hand side can be computed via the Grothendieck spectral sequence, so that
\[
\Tor_2\left( \CO_P, \PhiK \right) = 0
\]
if and only if
\begin{equation}\label{eq:Dtr is common point}
\Tor_2 \left(\CO_{p_Y^{-1}(P)}, p_V^*\CO_{\PP(K)}(2h) \right)
\end{equation}
has sections.
The sheaf \eqref{eq:Dtr is common point} can be computed via the Koszul complex for $\PP(K)$. The result is the kernel of
\[
\CO_{p_Y^{-1}(P)} \to (V/K)^* \otimes \CO_{p_Y^{-1}(P)}(1)
\]
which can be either $0$ or $\CO_{p_Y^{-1}(P)}$. 
The second case occurs if and only if 
\[
p_Vp_Y^{-1}(P) \not\subset \PP(K)
\]
so that we have proved the equivalence of conditions \eqref{item:Dtr is common point a} and \eqref{item:Dtr is common point b}.

As for the equivalence of conditions \eqref{item:Dtr is common point a} and \eqref{item:Dtr is common point c}, note that in order to compute
\[
\Tor_2 \left( \CO_P, \Phi_{\Bl}(\CO_{\PP(K)}(2)) \right)
\]
it is enough to restrict the locally free resolution \eqref{resolution of Phi(K)} to $P$, so that condition \eqref{item:Dtr is common point a} is equivalent to the fact that
\[
0 \to \calO_Y \to (V/K)^* \otimes \calU^*
\]
is $0$ at $P$, i.e. that $P$ belongs to all conics in the pencil given by $(V/K)^* \subset V^*$.

The last part of the statement follows by letting $P$ vary in $Y$ and by corollary \ref{cor:points and trisecants}.
\end{proof}

In the description of the moduli space of instantons with $c_2 = 3$ we will need the right adjoint $\Phi_{\Bl}^!$ to $\Phi_{\Bl}$.
We discuss it in the following lemma.
\begin{lemma}\label{lemma:right adjoint to Bl}
Let $\Phi_{\Bl}^!$ be the right adjoint to $\Phi_{\Bl}$. Then the following facts hold.
\begin{enumerate}
\item
$\Phi_{\Bl}^!(-) = Rp_{V*}(Lp_Y^*(-)\otimes\calO_{\PP_{Y}(\CU)}(h-E)[1])$
\item
$\Phi_{\Bl}^!(\CO_Y) \cong I_{\Imsigma}(1)[1]$
\item There is an exact sequence
\[
0 \to I_{\Imsigma}(1) \to \Phi_{\Bl}^!(\CU)[-1] \to \CO_{\PP(V)}(-h) \to 0
\]
\end{enumerate}
\end{lemma}
\begin{proof}
\begin{enumerate}
\item
Since $\Phi_{\Bl} = Rp_{Y*} \circ Lp_V^*$, its right adjoint $\Phi_{\Bl}^!$ is $Rp_{V*}\left( Lp_Y^* \left( - \right) \otimes \omega_{p_Y} \right)$[1]. We have already computed $\omega_{p_Y}$ in \ref{picard of blow up}, so by substitution we find 
\[
\Phi_{\Bl}^!(-) = Rp_{V*}(Lp_Y^*(-)\otimes\calO_{\PP_{Y}(\CU)}(h-E)[1])
\]
\item
The sequence
\[
0 \to \calO_{\PP_{Y}(\CU)}(h - E) \to \calO_{\PP_{Y}(\CU)}(h) \to \calO_E(h) \to 0
\]
is exact. Pushing it forward to $\PP(V)$ we find
\[
\Phi_{\Bl}^!(\calO_Y) = Rp_{V*}\calO_{\PP_{Y}(\CU)}(h-E)[1] = I_{\Imsigma}(1)[1]
\]
\item
Recall that by proposition \ref{picard of blow up} on $\PP_Y(\calU)$ there is a tautological extension
\[
0 \to \calO_{\PP_{Y}(\CU)}(-h) \to p_Y^*\calU \to \calO_{\PP_{Y}(\CU)}(-2h + E) \to 0
\]
Twisting it by $\calO_{\PP_{Y}(\CU)}(h-E)$ and pushing it forward we find that $\Phi_{\Bl}^!(\calU)$ is concentrated in degree $1$ and that it sits in an exact sequence
\begin{equation}\label{beta3 eq1}
0 \to I_{\Imsigma} \to \Phi_{\Bl}^!(\calU)[-1] \to \calO_{\PP(V)}(-h) \to 0
\end{equation}
\end{enumerate}
\end{proof}

Recall from definition \ref{def:universal S_K} that we denote  by $S_K$ the total preimage of $\PP(K)$ via $p_V$, and that $S_K$ it fits in diagram \eqref{diag:P(K), S_K, Y_K}.

\begin{lemma}\label{adjunction S_K Y_K}
For any $K \in \Gr(3,V)$
\[
Rp_{Y_K*}\CO_{S_K} \cong \CO_{Y_K}
\]
\end{lemma}
\begin{proof}
As $Lp_{Y_K}^*\CO_{Y_K} \cong \CO_{S_K}$, we only need to compute $Rp_{Y_K*}\CO_{S_K}$. 
Note that in \ref{transform of P(K)} we have checked that $i_{Y_K*}Rp_{Y_K*}\CO_{S_K}$ and $\CO_{Y_K}$ are isomorphic, so that by exactness of $i_{Y_K*}$ we find that $Rp_{Y_K*}\CO_{S_K}$ is equivalent to a complex with cohomology concentrated in degree $0$ and isomorphic to $\CO_{Y_K}$.
\end{proof}

\begin{lemma}\label{lemma:p_{Y_K} isomorphism}
For $K \notin \Dtr$, the map $p_{Y_K}: S_K \to Y_K$ is an isomorphism.
\end{lemma}
\begin{proof}
First, we prove that $p_{Y_K}$ has relative dimension $0$.
Clearly, the only irreducible curves which $p_Y$ contracts are the fibers of points of $Y$.
By lemma \ref{cor:points and trisecants} such curves are the strict transforms of trisecants to $\Imsigma$ inside $\PP(V)$.
As $K \notin \Dtr$, $\PP(K)$ does not contain trisecants, so that $S_K$ does not contain the strict transform of any trisecant.

Finally note that $p_{Y_K}$ is proper as $p_Y$ is proper and that the canonical map 
\[
\CO_{Y_K} \to p_{Y_K*}\CO_{S_K}
\] 
is an isomorphism by \ref{adjunction S_K Y_K}.
By lemma \ref{lemma:closed embedding and pushforward}, $p_{Y_K}$ is an isomorphism.
\end{proof}


\section{Instantons} \label{sec:Instantons}

In section \ref{subsec:instantons on fano 3-folds} we introduce instantons on Fano 3-folds and state the main results that are already known about them. Section \ref{subsec:Instantons on Y} is devoted to instantons on $Y$ and does not contain new results, but it contains a detailed discussion about the description of $\CMI_n$ as a geometric quotient of an affine variety. Section \ref{subsec:Jumping lines and conics} contains the main new results about the splitting type of instantons on lines and conics. 


\subsection{Instantons on Fano 3-folds}\label{subsec:instantons on fano 3-folds}


Let $X$ be a smooth projective variety of dimension $\dim X$. One says that $X$ is Fano when its anti-canonical class $-\omega_X$ is ample. 
Assume that $\Pic (X) = \ZZ$ and denote the ample generator by $\CO_X(1)$.
The index $\iota_X$ of $X$ is the maximal integer dividing the canonical class $\omega_X$. 

By \cite{kobayashi1973characterizations}, we have $\iota_X \leq \dim X + 1$. Moreover, the only Fano variety of index $\dim X + 1$ is the projective space and 
the only Fano variety of index $\dim X$ is the non-singular quadric hypersurface inside the projective space of dimension $\dim X + 1$.
In dimension $3$ the classification is complete \cite{iskovskikh1980anticanonical} \cite{mukai1983minimal}. 
A list of the deformation classes of Fano 3-folds can be found for example in \cite{iskovskikh1999fano}. 

In analogy with $\PP^3$, where the twist by $\CO(-2)$ is the twist by a square root of the canonical class, Kuznetsov gives the following definition \cite{kuznetsov2012instanton} of instanton on a Fano 3-fold of index 2.
\begin{definition}\label{def:instanton kuznetsov}
An instanton $E$ on a Fano 3-fold $X$ of index $2$ is a rank $2$ $\mu$-stable vector bundle such that $c_1(E) = 0$ and $H^1(E(-1)) = 0$.  The second Chern class $c_2(E) = n$ is called the charge of $E$.
\end{definition}
\begin{remark}\label{rmk:ch(E)}
One can easily check that
\begin{equation}\label{eq:ch(E)}
\ch(E) = 2 - c_2(E) = 2 - n \cdot L
\end{equation}
Moreover, it follows from the stability (and from the anti-selfduality) of $E$ that
\[
h^1(E) = c_2(E) - 2 = n - 2
\]
and that $h^2(E) =  h^3(E) = 0$.
It follows that there are no instantons with $c_2(E) < 2$.
\end{remark}

In the more general setting of a Fano 3-fold of arbitrary index, Faenzi gives a definition of instanton bundle in \cite{faenzi2011even}.

\begin{definition}\label{def:instanton faenzi}
An instanton $E$ on a Fano 3-fold $X$ is the normalization of a rank $2$ Gieseker-stable vector bundle $F$ such that $F \cong F^* \otimes \omega_X$ and $H^1(F) = 0$.
\end{definition}

In the next sections we will work on the Fano 3-fold $Y$ defined in section \ref{sec:The setting}, but in this section we will recall a few results that are already known about instantons on Fano 3-folds, starting from the following remark.
\begin{remark}
For Fano $3$-folds of index $2$ the definitions \ref{def:instanton kuznetsov} and \ref{def:instanton faenzi} are equivalent.
This follows directly from lemma \ref{lemma:mumford vs Gieseker}.
\end{remark}

We will denote by $\CMI_{X,k}$ the moduli space of instantons of charge $k$ on $X$. 
When working with $X = Y_5 = Y$ we omit the $X$ in the index, so that $\CMI_k$ stands for $\CMI_{Y,k}$. 
There are some simple constraints on the values of $k$ and of some discrete parameters of $X$ which guarantee non-emptiness of $MI_{X,k}$ and generic smoothness of at least an irreducible component of it. 

\begin{theorem}[\cite{faenzi2011even}]\label{thm:faenzi smooth}
The emptiness of $\CMI_{X,k}$ depends only on $k$ and on discrete invariants of $X$. If it is non-empty, $\CMI_{X,k}$ has a generically smooth irreducible component of dimension depending only on $k$ and on discrete invariants of $X$.
\end{theorem}

A classical result about instantons on $\PP^3$ is the existence of a monadic description for them.
More precisely, for any instanton $E$ of charge $k$ on $\PP^3$ there is an anti-selfdual complex (called monad)
\[
\CO_{\PP^3}(-1)^{\oplus n} \to \CO_{\PP^3}^{\oplus 2n+2} \to \CO_{\PP^3}(1)^{\oplus n}
\]
such that the first map is injective, the last is surjective and the middle cohomology is $E$.
From this point of view, the istantonic condition $H^1(E(-2))$ and the stability of $E$ provide the orthogonality of $E$ to $\CO_{\PP^3}(2)$, and the monad is the decomposition of $E$ with respect to the selfdual semi-orthogonal decomposition for $\CO_{\PP^3}(2)^\perp$ inside $\D^b(\PP^3)$.

This point of view generalizes to all $X$ Fano threefolds of Picard number $1$ having a full exceptional collection. Such $X$ are completely classified and they are:
\begin{itemize}
\item $\PP^3$;
\item $Q$, a smooth quadric inside $\PP^4$;
\item $Y \subset \Gr(2,5)$, the intersection of $3$ linear sections of $\Gr(2,5)$ 
\item $X_{22} \subset \Gr(3,7)$, the zero locus of three sections of $\Lambda^2\CU_3^*$, where $\CU_3$ is the rank 3 tautological bundle. 
\end{itemize}
For any $X$ in the above list and for any integer $k$, one can imitate the monadic description on $\PP^3$. 
The identification of the moduli spaces of instantons on the above Fano threefolds with geometric quotients of spaces of linear forms is carried out in \cite[thm. B]{faenzi2011even}.

In all other cases, only a partial decomposition is available for $\D^b(X)$.
When the index of $X$ is $2$, the derived category of $X$ decomposes as
\[
\D^b(X) = \langle \calB_X, \CO_X, \CO_X(1) \rangle
\]
where $\calB_X$ is called the non-trivial part of $\D^b(X)$.

The instantonic condition $H^1(E(-1))$ and the stability of $E$ imply the orthogonality of $E$ with respect to $\CO_X(1)$.
The projection $\widetilde{E}$ of an instanton $E$ to $\calB_X$ sits in an exact sequence
\begin{equation}\label{eq:acyclic extension}
0 \to E \to \widetilde{E} \to H^1(E) \otimes \CO_X \to 0
\end{equation}
as the universal extension of $\CO_X$ by $E$.
Conversely, given $\widetilde{E}$ with the appropriate Chern character and $H^0$, if $\WE$ is self dual with respect to a certain anti-autoequivalence of $\langle \CO_X, \CO_X(1) \rangle^\perp$, it is proved in \cite[thm 3.10]{kuznetsov2012instanton} that one can reconstruct uniquely an instanton $E$ such that $\widetilde{E}$ is its universal extension.


\subsection{Instantons on Y}\label{subsec:Instantons on Y}


In the case of instantons on $Y$, one can write an anti-selfdual monad  in terms of the exceptional collection $\CU, \CO_Y, \CU^*, \CO_Y(1)$.
It is often convenient to work with the decomposition of an instanton with respect to collection \eqref{eq:full exceptional}, that is to say $\CU, \CU^\perp, \CO_Y, \CO_Y(1)$.

Fix a positive integer $n$ and an $n$-dimensional vector space $H$.
Recall that $A$ is both the space of linear sections which cut $Y$ inside $\Gr(2,V)$ and the space of maps from $\CU$ to $\CU^\perp$ by lemma \ref{lemma:cohomologies on Y}.
As a consequence, given $\gamma \in A \otimes S^2H^*$ one can associate to it a map $\gamma^{\prime}$ from $H \otimes \CU$ to $H^* \otimes \CU^\perp$ and a map $\hat{\gamma}$ as in the following definition.
\begin{definition}
Given $\gamma \in A \otimes S^2H^*$, we denote by $\gamma^{\prime}$ the composition
\[
\gamma': H \otimes \CU \xrightarrow{\gamma} A \otimes H^* \otimes \CU \xrightarrow{H^* \otimes \ev} H^* \otimes \CU^\perp
\]
and by $\hat{\gamma}$ the map
\[
\hat{\gamma}:H \otimes V \to H^* \otimes V^*
\]
obtained by polarizing $\gamma$ with respect to $V$.
\end{definition}
\begin{remark}
Note that $\hat{\gamma} = \Hom( \gamma^{\prime}, \CO_Y)$ under the canonical isomorphisms $\Hom(\CU^\perp, \CO_Y) \cong V$ and $\Hom(\CU,\CO_Y) \cong V^*$.
Moreover, we will often write
\[
\gamma: A \to S^2H^*
\]
by means of the canonical $SL_2$-invariant identification between $A^*$ and $A$.
\end{remark}

It turns out (\cite{kuznetsov2012instanton}, \cite{faenzi2011even}) that every instanton $E$ on $Y$ is isomorphic to the middle cohomology of a $3$-term complex
\begin{equation}\label{MONAD}
0 \to H \otimes \CU \to H^* \otimes \CU^\perp \to H^{\prime} \otimes \CO_Y \to 0
\end{equation}
with $\dim(H^{\prime}) = \dim(H) - 2 = n-2$ and such that the first map is fiberwise injective and the last is surjective.
Complex \eqref{MONAD} is usually called a \emph{monad} for $E$.
The following theorem is an improvement of this description.

\begin{theorem}[\cite{faenzi2011even}, \cite{kuznetsov2012instanton}] \label{maininstantontheorem} 
Let $H^*$ be a vector space of dimension $n$ and let $G := \GL(H)/\{\pm 1\}$. Denote by $\Mon_n(Y)$ the set of all $\gamma \in A\otimes S^2H^*$ which satisfy the following conditions
\begin{enumerate}[(i)]
\item \label{itm:conditionfiber} the map $\gamma ': H \otimes \calU \to H^* \otimes \calU^\perp$ is a fiberwise monomorphism of vector bundles,
\item \label{itm:conditionrank} the rank of the map $\hat{\gamma} : H \otimes V \to H^* \otimes V^*$ equals $4n+2$.
\end{enumerate}
Then the coarse moduli space $\CMI_n(Y)$ of instantons of charge $n$ on $Y$ is the GIT-quotient $\Mon_n(Y)/G$. 
\end{theorem}
\begin{remark}[Scheme structure]
More precisely, $\Mon_n := \Mon_n(Y)$ should be given the following scheme structure.
Consider the polarized tautological map
\begin{equation}\label{eq:polarized tautological}
\CO_{A \otimes S^2H^*} \otimes H \otimes V  \to \CO_{A \otimes S^2H^*} \otimes H^* \otimes V^*
\end{equation}
on $A \otimes S^2H^*$.
Take the degeneration locus of rank $4n + 2$, i.e. the locus where the map \eqref{eq:polarized tautological} has rank exactly $4n +2$.
This is a locally closed subscheme in $A \otimes S^2H^*$.

Next, consider over $A \otimes S^2H^* \times Y$ the universal composition
\begin{equation}\label{eq:scheme structure for Mon_n}
\CO_{A \otimes S^2 H^*} \otimes H \boxtimes \CU \to \CO_{A \otimes S^2 H^*} \otimes H^* \boxtimes A \otimes \CU \to
\CO_{A \otimes S^2 H^*} \otimes H^* \boxtimes \CU^\perp
\end{equation}
It is a map of vector bundles of rank $2n$ and $3n$ and it is fiberwise a monomorphism out of a closed subset.
As $Y$ is proper, the image of this subset in $A \otimes S^2H^*$ is again closed, so that the condition that the map \eqref{eq:scheme structure for Mon_n} is injective at all points of $Y$ defines an open subscheme of $A \otimes S^2H^*$.

The intersection of the two subschemes which we have just defined, provides $\Mon_n$ with a natural scheme structure, which is the one used in theorem \ref{maininstantontheorem}.
\end{remark}

\begin{remark}[Geometric quotient]\label{rmk:geometric quotient}
Even though theorem \ref{thm:faenzi smooth} says that for each $n \geq 2$ there is a generically smooth irreducible component of $\CMI_n$, nothing is known about possible other components when $n \geq 4$. 
It is not even known whether in general $\Mon_n$ is reduced or not. 

On the other end, in the case $n = 2$ or $n=3$, condition \ref{maininstantontheorem} \eqref{itm:conditionrank} is open.
It follows that $\Mon_2$ and $\Mon_3$ are smooth.

For arbitrary $n$, it is straightforward to check that the orbits of $G$ in $\Mon_n$ are closed and that all stabilizers are trivial (i.e. that all points of $\Mon_n$ are GIT-stable with respect to the action of $G$). 
It suffices to note that if an orbit is not closed, then there are two non-isomorphic monads mapping to the same point in the coarse moduli space of instantons, which is impossible by the naturality of the Beilinson spectral sequence.
It follows that $\CMI_n$ is a geometric quotient of $\Mon_n$ by $G$, so that $\CMI_2$ and $\CMI_3$ are smooth.
\end{remark}

\begin{remark}[Coarse moduli space]
Denote by $\FMI_n$ the functor associating with a scheme $S$ the set of families of instantons on $Y$ up to twists by line bundles on $S$.
In order to check that $\Mon_n /G$ is the coarse moduli space for $\FMI_n$ one has to check a few facts, namely that there is a natural transformation from $\FMI_n$ to $\Mon_n/G$ which is bijective on closed points and that it corepresents $\FMI_n$.

The key point in the proof of the above facts is that over $\Mon_n \times Y$ there is a universal monad, as we are about to see.
Note that over $A \otimes S^2 H^*$ there is a map
\begin{equation}\label{eq:universal short n-monad}
0 \to  H \boxtimes \CU \to  H^* \boxtimes \CU^\perp \to 0
\end{equation}
called the universal short monad.


By property \ref{maininstantontheorem} \eqref{itm:conditionfiber}, the cohomology of the universal short monad \eqref{eq:universal short n-monad} on $\Mon_n \subset A \otimes S^2H^*$ is locally free of rank $n$.

By property \ref{maininstantontheorem} \eqref{itm:conditionrank}
over $\Mon_n$ there is a universal complex
\begin{equation}\label{eq:universal n-monad}
H \boxtimes \CU \to H^* \boxtimes \CU^\perp \to \calD^* \boxtimes \CO_Y
\end{equation}
where $\calD$ 
is a vector bundle of rank $n-2$ sitting in an exact anti-selfdual exact sequence
\begin{equation}\label{eq:universal n-kernel}
0 \to \calD \to H \otimes V \otimes \CO_{A \otimes S^2H^*} \to H^* \otimes V^* \otimes \CO_{A \otimes S^2 H^*} \to \calD^* \to 0
\end{equation}
We will call the object in \eqref{eq:universal n-monad} the universal monad, as its restriction to each point of $A \otimes S^2H^*$ is isomorphic to
\begin{equation}\label{eq:n-monad}
0 \to H \otimes \CU \to H^* \otimes \CU^\perp \to H^{\prime} \otimes \CO_Y  \to 0
\end{equation}

The universal monad \eqref{eq:universal n-monad} is $GL(H)$-equivariant but not $G$-equivariant, as $-1$ acts non-trivially on it.
When $n$ is odd, one can twist the $GL(H)$ action by the character $\det$ and make it $G$-equivariant.
When $n$ is even all characters of $GL(H)$ are trivial on $-1$, so that one cannot use the same method.
We will discuss the behaviour for $n=2$ in proposition \ref{prop:minimal is not fine}.

In order to prove that $\Mon_n/G$ is a coarse moduli space, one should first provide a natural transformation from the functor $\FMI_n$ to $\Mon_n/G$.
This can be done in the same way as in theorem \cite[4.1.12]{okonek1980vector}. 
Given a family of instantons parametrized by $S$, use the Beilinson spectral sequence to get a monad for it, restrict to open subsets $S_i$ where the bundles in the monad trivialize to get maps from $S_i$ to $\Mon_n$ and finally show that the induced maps to $\Mon_n/G$ are compatible and therefore glue to a map from $S$ to $\Mon_n/G$.

The proof that this natural transformation is bijective on closed points can be carried out exactly in the same way of \cite[4.1.12]{okonek1980vector}.

As we do not know whether $\Mon_n$ is reduced or not, we have to be a bit more careful while proving that $\Mon_n/G$ corepresents $\FMI_n$.
The proof in \cite[4.1.12]{okonek1980vector} uses the fact that the space of monads is reduced, and we want to avoid using it.
The main problem is in the construction of a commutative diagram
\[
\begin{diagram}
\FMI_n			& \rTo^{\Psi}	& \Hom( - , N)	\\
\uTo^{\mathsf{E}}	&			& \uDashto		\\
\Hom( - , \Mon_n)	& \rTo^{\quad \pi \quad}	& \Hom( - , \Mon_n/G)
\end{diagram}
\]
once a natural transformation $\Psi$ is given. 
Here $\mathsf{E}$ is the map from $\Mon_n$ to $\FMI_n$ corresponding to the cohomology of the universal monad \eqref{eq:universal n-monad}.
Note that it is not clear that the map $\Psi \circ \mathsf{E}$ descends to a map from $\Mon_n/G$, as the universal $\mathsf{E}$ is not necessarily $G$-equivariant, but only $\GL(H)$-equivariant.

On the other hand, if we want to check that the diagram 
\begin{equation}\label{diag:G-descent}
\begin{diagram}
G \times \Mon_n	& \rTo	& \Mon_n	\\
\dTo				&		& \dTo	\\
\Mon_n			& \rTo	& N
\end{diagram}
\end{equation}
is commutative, we can replace $G \times \Mon_n$ by a flat cover.
As a flat cover we choose $GL(H) \times \Mon_n$, so that by $\GL(H)$-equivariance of $\mathsf{E}$ we can deduce the commutativity of
\[
\begin{diagram}
\GL(H) \times \Mon_n	& \rTo	& \Mon_n	\\
\dTo					&		& \dTo	\\
\Mon_n				& \rTo	& N
\end{diagram}
\]
and therefore of \eqref{diag:G-descent}.

In the above way one can avoid referring to closed points and complete the proof of the fact that $\Mon_n/G$ is a coarse moduli space following the lines of \cite[4.1.12]{okonek1980vector}.
\end{remark}

\subsection{Jumping lines and conics} \label{subsec:Jumping lines and conics}

We are now going to discuss the restriction of $E$ to rational curves in $Y$.
By rational curve $C$ we mean a scheme of pure dimension $1$ with $h^1(\CO_C) = 0$ and $h^0(\CO_C) = 1$.
Let us first recall the following definition.

\begin{definition}\label{def:jumping}
A rational curve $C$ is jumping for an instanton $E$ if $\restr{E}{C}$ is not trivial. If $C$ is smooth we will say that $C$ is $n$-jumping for $E$ if $\restr{E}{C} \cong \CO_C(-n) \oplus \CO_C(n)$. In this case we will call $n$ the \emph{order of jump} of $E$ at $C$.
\end{definition}

Recall that in proposition \ref{prop:fano of lines} and \ref{prop:universal conic} we have constructed the universal families of lines
\begin{equation}\label{diag:universal line}
\begin{diagram}
	&			&\CL		&			&\\
	&   \ldTo^{r_Y}	&		&\rdTo^{r_A}	&\\
Y	&			&		&			& \PP(A)
\end{diagram}
\end{equation}
and of conics
\begin{equation}\label{diag:universal conic}
\begin{diagram}
	&			&\calC		&			&\\
	&   \ldTo^{q_Y}	&			&\rdTo^{q_V}	&\\
Y	&			&			&			& \PP(V^*)
\end{diagram}
\end{equation}

\begin{definition}
The Fourier--Mukai transforms from $D^b(Y)$ to $D^b(\PP(V^*))$ with kernel $\CO_{\CL}$ and $\CO_{\calC}$ are denoted by
\begin{gather}
\Phi_{\CL} := Rr_{A*}Lr_Y^* : D^b(Y) \to D^b(\PP(A)) \\
\Phi_{\calC} := Rq_{V*}Lq_Y^*: D^b(Y) \to D^b(\PP(V^*))
\end{gather}
\end{definition}

Theorem \ref{thm:jumping sheaf} and corollary \ref{cor:jumping order} clarify the relation between jumping lines and the Fourier--Mukai transform $\Phi_{\CL}(E(-1))$. 
\begin{theorem}[\cite{kuznetsov2012instanton}]\label{thm:jumping sheaf}
Let $E$ be an instanton on $Y$ and 
\[
\gamma_E \in A \otimes S^2H^* \cong A^* \otimes S^2H^*
\]
via the canonical identification of $A$ and $A^*$. Then there is a distinguished triangle 
\[
\Phi_{\CL}(E(-1)) \to H \otimes \CO_{\PP(A)}(-3) \xrightarrow{\gamma_E} H^* \otimes \CO_{\PP(A)}(-2)
\]
\end{theorem}
We include a sketch of a proof of theorem \ref{thm:jumping sheaf} different from the one which is given in \cite{kuznetsov2012instanton}.
\begin{proof}
First, as $\CO_L(-1)$ is acyclic,
\[
\Phi_{\CL}(E(-1)) = \Phi_{\CL}(\widetilde{E}(-1))
\]
where $\widetilde{E}$ is the acyclic extension of $E$ defined by \eqref{eq:acyclic extension}. By theorem \ref{maininstantontheorem} we have a resolution
\[
0 \to H \otimes \CU \to  H^* \otimes \CU^\perp \to \widetilde{E} \to 0
\]
By lemma \ref{lemma:splitting on lines}, $\Phi_{\CL}(\CU(-1))$ and $\Phi_{\CL}(\CU^\perp(-1))$ are line bundles over $\PP(A)$ shifted by $-1$.
By means of the resolution for $\CO_{\CL}$ given in proposition \ref{prop:fano of lines} one can check that they are $\CO_{\PP(A)}(-3)$ and $\CO_{\PP(A)}(-2)$.
Finally, by computing $\Phi_{\CL}(\CO_L(-1))$ one can also check that
\[
\Phi_{\CL}: \Hom(\CU(-1), \CU^\perp(-1)) \to \Hom(\CO_{\PP(A)}(-3), \CO_{\PP(A)}(-2))
\]
is an isomorphism.
\end{proof}
We introduce a shorthand notation for $\Phi_{\CL}(E(-1))[1]$.
\begin{definition}\label{def:JL}
We will call $\Phi_{\CL}(E(-1))[1]$ the \emph{object of jumping lines of $E$} and we will denote it by $\JL_E$. 
If the generic line is not jumping for $E$, then by theorem \ref{thm:jumping sheaf} it is equivalent to a sheaf and we will therefore call it the \emph{sheaf of jumping lines} of $E$.
\end{definition}
\begin{corollary}\label{cor:jumping order}
Let $L \subset Y$ be a line, then $\restr{E}{L} = (-n,n)$ if and only if the corank of $\gamma_E$ at $a_L$ is $n$.
\end{corollary}
\begin{proof}
The integer $n$ such that $\restr{E}{L} = (-n,n)$ is the dimension of $H^1(E(-1))$.
By flat base change at the point representing $L$ for the diagram
\[
\begin{diagram}
L	&\rTo		&\CL \\
\dTo	&		&\dTo \\
a_L	&\rTo		&\PP(A)
\end{diagram}
\]
we have that the derived restriction of $\Phi_{\CL}(E(-1))$ at $a_L$ is the cohomology $H^\bullet(E(-1))$.
By theorem \ref{thm:jumping sheaf} the dimension of $H^1(E(-1))$ is the corank of $\gamma_E$ at $a_L$.
\end{proof}

In the following theorem we will prove that for an instanton $E$ of charge $n$ the order of jump at a line is strictly less than $n$. 
This theorem is a key step in the description \ref{minimalmoduli} of instantons of charge $2$. 
It is also useful in higher charge as it shows that the map
\[
\gamma : A^* \to S^2 H^*
\]
associated with a monad for an instanton $E$ induces a regular map from $\PP(A^*)$ to $\PP(S^2H^*)$.

\begin{theorem}\label{thm:gamma injective}
Let $E$ be an instanton and
\begin{equation}\label{eq:gamma injective 1}
\gamma^{\prime}: H \otimes \CU \to H^* \otimes \CU^\perp 
\end{equation}
the initial part of a monad for E.
Let $\gamma : A^* \to S^2 H^*$ the map associated to the monad. Then $\gamma$ is injective. 
\end{theorem}
\begin{proof}
Assume there is $a$ such that $\gamma(a) = 0$. 
By base change $\gamma(a)$ is $H^1$ of 
\[
\restr{\gamma'}{L_a}: H \otimes \restr{\CU}{L} \to H^* \otimes \restr{\CU^\perp}{L}
\]
which by lemma \ref{lemma:splitting on lines} is isomorphic to
\[
\restr{\gamma'}{L_a}(-1): H \otimes (\CO_L(-1) \oplus \CO_L(-2)) \xrightarrow{0} H^* \otimes (\CO_L(-1) \oplus \CO_L(-1) \oplus \CO_L(-2))
\]
As $\gamma'$ represents an instanton, by theorem \ref{maininstantontheorem} it is fiberwise injective, so that it induces a fiberwise injection
\begin{equation}\label{eq:gamma injective}
H \otimes \CO_L(-2) \xrightarrow{0} H^* \otimes ( \CO_L(-1) \oplus \CO_L(-2))
\end{equation}
which becomes $0$ when we apply $H^1$.
It follows that the map \eqref{eq:gamma injective} is a fiberwise injection factoring via $H^* \otimes \CO_L(-1)$.
As there is no fiberwise injection from $H \otimes \CO_L(-2)$ to $H^* \otimes \CO_L(-1)$, we conclude that it is impossible to have $\gamma(a) = 0$.
\end{proof}

\begin{corollary}\label{cor:no n-jumps}
The order of jump of $E$ at a line $L$ is at most $c_2(E) - 1$.
\end{corollary}
\begin{proof}
By corollary \ref{cor:jumping order} the order of jump at a line $L$ is the corank of $\gamma(a_L)$. By theorem \ref{thm:gamma injective} the rank of $\gamma(a_L)$ is at least $1$, so that the corank is at most $c_2(E) - 1$.
\end{proof}

It is not known whether $\JL_E$ is supported on the whole $\PP(A)$ or not (see e.g. \cite[conj. 3.16]{kuznetsov2012instanton}.
The reason is that the usual Grauert--M\"ulich theorems, such as theorem \ref{thm:GM}, can only be used provided the family of lines through a general point is irreducible.
By corollary \ref{cor:3 to 1 cover} this is clearly not the case for $Y$.

If we turn our attention to conics, the situation changes completely: by corollary \ref{cor:conics through P} the family of conics passing through a point of $Y$ is parametrized by $\PP^2$, so that the assumptions of theorem \ref{thm:GM} are satisfied.

Recall that by definition \ref{def:jumping} a conic is jumping for $E$ if and only if $\restr{E}{C}$ is not trivial.
We are now going to construct an analogue of $\JL_E$ in the case of conics.

\begin{remark}\label{rmk:support of jumping lines}
If $L$ is a line, then the restriction $\restr{E(-1)}{L}$ is acyclic if and only if $L$ is not jumping for $E$. Analogously, by corollary \ref{splittingonconic}, if $C$ is a smooth conic, then the restriction $\restr{E\otimes \CU}{C}$ is acyclic if and only if $C$ is not jumping for $E$.
\end{remark}

The above remark suggests that a possible analogue of $\JL_E$ in the case of conics could be a shift of $\Phi_{\calC}(E \otimes \CU)$.
We will now check that the restriction of $E \otimes \CU$ to singular conics detects whether $C$ is jumping or not for $E$, that is to say whether $\restr{E}{C}$ is trivial or not.

Recall that by conic we mean a subscheme $C$ of $Y$ with Hilbert polynomial $2t + 1$ and that this implies that $C$ has pure dimension $1$ and $h^1(\CO_C) = 0$ (see remark \ref{rmk:h1 of a conic}).
Let $C$ be a singular conic and denote by $L_1, L_2$ its components (in the case $C$ is a double line, take $L_1 = L_2$).
Then, we lemma \ref{trivialityoncross} and proposition \ref{prop:equivalence of jumping} provide us with a cohomological way to check whether a conic is jumping or not.

As some conics in $Y$ are non-reduced, we will need the following remark.
\begin{remark}\label{rmk:exact for conics}
For each singular conic $C$ (possibly non-reduced) the ideal $\CI_1$ of $L_1$ in $C$ is isomorphic to $\CO_{L_2}(-1)$.
The reason is that by additivity the Hilbert polynomial of $\CI_1$ is $t$, which shows that it is schematically supported on a component of $C$.
Moreover it has no associated points of dimension $0$ because it is a subsheaf of $\CO_C$.
This is enough to conclude that it is isomorphic to $\CO_{L_2}(-1)$, so that  there is an exact sequence
\begin{equation}\label{eq:exact for conics}
0 \to \CO_{L_2}(-1) \to \CO_{C} \to \CO_{L_1} \to 0
\end{equation}
both in the case $L_1 \neq L_2$ and in the case $L_1 = L_2$
\end{remark}

\begin{lemma} \label{trivialityoncross}
Let $F$ be a vector bundle of rank $r$ on a singular conic $C$ with (possibly coinciding) components $L_1, L_2$. If $F$ is trivial on $L_1$ and $L_2$ then $F$ is trivial on $C$.
\end{lemma}
\begin{proof}
As $F$ is trivial on $L_1$, we can choose an isomorphism
\[
s_1: \CO_{L_1}^{\oplus r} \to \restr{F}{L_1}
\]
Tensor the exact sequence \eqref{eq:exact for conics} by $F$. 
As $F$ is trivial on $L_2$, the restriction $\restr{F(-1)}{L_2}$ is acyclic, so that $s_1$ lifts to
\[
s: \CO_C^{\oplus r} \to F
\]

As $\CO_C^{\oplus r}$ and $F$ are vector bundles of the same rank, in order to check that $s$ is an isomorphism it is enough to check that $s$ is surjective. This is equivalent to surjectivity of $s_1$ and $s_2 = \restr{s}{L_2}$.
As $F$ is trivial on each $L_i$ and $L_i$ is irreducible, it is enough to check that each $s_i$ is surjective at some closed reduced point $P_i \in L_i$.
Choose $P_1 = P_2 = P \in L_1 \cap L_2$, then 
\[
\restr{s_1}{P} \cong \restr{s_2}{P}
\]
and the first one is an isomorphism by definition.
\end{proof}

\begin{proposition}\label{prop:equivalence of jumping}
A conic $C$ is jumping for an instanton $E$ if and only if 
\[
H^\bullet(\restr{E \otimes \CU}{C}) \neq 0
\]
\end{proposition}
\begin{proof}
Assume $\restr{E}{C}$ is trivial on $C$. 
Then there is an isomorphism between $\restr{E \otimes \CU}{C}$ and $\restr{\CU}{C}^{\oplus 2}$. 
By lemma \ref{cohomologyonconic2} we have $H^\bullet(\restr{E \otimes \CU}{C}) = 0$.

In the other direction, for $C$ smooth conic, $H^\bullet(\restr{E \otimes \CU}{C}) = 0$ implies $\restr{E}{C}$ trivial by \ref{splittingonconic}.
For $C$ singular conic, by lemma \ref{trivialityoncross} it is enough to prove that $H^\bullet(\restr{E \otimes \CU}{C}) = 0$ implies $\restr{E}{L}$ trivial for $L$ component of $C$.
Tensoring sequence \eqref{eq:exact for conics} by  $E \otimes \CU$ we find is an injection
\[
H^0(E \otimes \CU \otimes \CO_L(-1)) \to H^0(\restr{E \otimes \CU}{C}) = 0
\]
so that also $H^0(E \otimes \CU \otimes \CO_L(-1))$ vanishes.
By lemma \ref{lemma:splitting on lines} and $c_1(E) = 0$ we deduce that $\restr{E}{L}$ is trivial.
\end{proof}

The following remark is analogous to remark \ref{rmk:support of jumping lines}, but concerns also singular conics.
\begin{remark}\label{rmk:support of jumping conics}
By proposition \ref{prop:equivalence of jumping} a conic $C$ is jumping for $E$ if and only if, in the notation of definition \ref{def:notation for conics}, $v_C \in \PP(V^*)$ lies in the support of $\Phi_{\calC}(E \otimes \CU)$.
\end{remark}

While the family of lines through a point in $Y$ is disconnected, we have proved in \ref{cor:conics through P} that the family of conics through a point $P$ in $Y$ is an irreducible variety.
This allows us to use \ref{thm:GM} in order to state and prove the following theorem.
\begin{theorem}\label{cor:generic conic is not jumping}
For an instanton $E$ on $Y$, the generic conic is not jumping. Moreover, $\Phi_{\calC}(E \otimes \CU)[1]$ is the cokernel of an injection
\begin{equation}\label{eq:resolution of JC}
H \otimes \big( \CO(-2) \oplus \left(A \otimes \CO(-1)\right) \big) \to H^* \otimes \big( \left(A^*\otimes \CO(-1) \right) \oplus \CO \big)
\end{equation}
and is equivalent to a sheaf supported on a divisor.
\end{theorem}
\begin{proof}
We will use theorem \ref{thm:GM} on the family of smooth conics in $Y$.
The flatness condition \eqref{itm:flat projection} holds as the projection from $\calC$ to $Y$ is the projection from $\PP_Y(\CU^\perp)$ to $Y$ by proposition \ref{prop:universal conic}.
The condition \eqref{itm:irreducible fiber} about the irreducibility of the generic fiber follows from corollary \ref{cor:conics through P}. 
The proportionality condition \eqref{itm:proportionality} is always true on $Y$ as each even cohomology group is isomorphic to $\ZZ$. 

We are now interested in the splitting type of an instanton $E$ on the generic smooth conic. By theorem \ref{thm:GM} we have that if $\ST_C(E) = (-j,j)$ with $j \geq 0$, then
\begin{equation}\label{eq:generic conic is not jumping -1}
2j \leq - \mu_{min}\left(\CT_{\calC/Y}/(\calC/\PP(V^*)\right)
\end{equation}
where the right hand side is the minimal slope in the Harder-Narasimhan filtration of the relative tangent bundle restricted to the general fiber of $\calC \to \PP(V^*)$.

This minimal slope can be computed as follows. 
First, by proposition \ref{prop:universal conic}, identify the universal conic (together with its two projections to $Y$ and $\PP(V^*)$) with $\PP_Y(\CU ^ \perp)$.

It follows that there is a relative Euler exact sequence
\[
0 \to \CO_{\calC} \to \CU^\perp \otimes p_V^*\CO_{\PP(V^*)}(1) \to \CT_{\calC/Y} \to 0
\]
which by corollary \ref{splittingonconic} restricts to any smooth conic $C$ as
\[
0 \to \CO_C \to \CO_C \oplus \CO_C(-1) \oplus \CO_C(-1) \to \restr{\CT_{\calC/Y}}{C} \to 0
\]
This is enough to conclude that $\ST_C(\CT_{\calC/Y}) = (-1,-1)$, so that inequality \eqref{eq:generic conic is not jumping -1} becomes $2j \leq 1$, that is to say $j=0$.

In order to show that $\Phi_{\calC}(E \otimes \CU)$ is equivalent to a shifted sheaf, note that as $\Phi_{\calC}(\CU) = 0 $ we have
\[
\Phi_{\calC}(E \otimes \CU) \cong \Phi_{\calC}(\WE \otimes \CU)
\]
where $\WE$ is the acyclic extension defined in \eqref{eq:acyclic extension}.
Moreover, for $\Phi_{\calC}(\WE \otimes \CU)$ we have a distinguished triangle
\[
H \otimes \Phi_{\calC}(\CU \otimes \CU) \to H^* \otimes \Phi_{\calC}(\CU^\perp \otimes \CU) \to \Phi_{\calC}(E \otimes \CU)
\]
Via Borel--Bott--Weil and the resolution for $\CO_{\calC}$ of proposition \ref{prop:universal conic}, the above distinguished triangle becomes
\[
H \otimes \big( \CO(-2) \oplus \left(A \otimes \CO(-1)\right) \big) \to H^* \otimes \big( \left(A^*\otimes \CO(-1) \right) \oplus \CO \big) \to \Phi_{\calC}(E \otimes \CU)[1]
\]
where we write $\CO$ instead of $\CO_{\PP(V^*)}$ for brevity.
By the first part of this proof, $\Phi_{\calC}(E \otimes \CU)$ vanishes at the generic point of $\PP(V^*)$, so that the map
\begin{equation}\label{eq:resolution of JC bis}
H \otimes \big( \CO(-2) \oplus \left(A \otimes \CO(-1)\right) \big) \to H^* \otimes \big( \left(A^*\otimes \CO(-1) \right) \oplus \CO \big)
\end{equation}
is injective at the generic point of $\PP(V^*)$.
By looking at the determinant of \eqref{eq:resolution of JC bis}, it follows that $\Phi_{\calC}(E \otimes \CU)[1]$ is equivalent to a sheaf supported on a divisor.
\end{proof}

In analogy with definition \ref{def:JL}, we give the following definition.
\begin{definition}\label{def:JC}
We will call $\Phi_{\calC}(E \otimes \CU)[1]$ the \emph{sheaf of jumping conics} of $E$ and we will denote it by $\JC_E$.
\end{definition}
\begin{remark}
By corollary \ref{cor:generic conic is not jumping} the object $\JC_E$ is equivalent to a sheaf supported on a divisor.
By remark \ref{rmk:support of jumping conics} a point $v_C \in \PP(V^*)$ is in the support of $\JC_E$ if and only if $C$ is jumping for $E$.
\end{remark}

The next corollary is a consequence of \ref{prop:equivalence of jumping} and \ref{cor:generic conic is not jumping} and allows to translate statements about jumping lines into statements about jumping conics.
Hopefully, this can be a step towards a proof of the following conjecture.
\begin{conjecture}\label{the conjecture}
For an instanton $E$ on $Y$ the generic line is not jumping.
\end{conjecture}

\begin{corollary}\label{cor:jumping lines vs conics}
For an instanton $E$ the following conditions are equivalent.
\begin{enumerate}
\item \label{itm:conics to lines 1}  The generic line is jumping for $E$.
\item \label{itm:conics to lines 2}  $\JC_E$ contains the divisor of singular conics.
\end{enumerate}
\end{corollary}
\begin{proof}
To prove that condition \ref{itm:conics to lines 2}  implies condition \ref{itm:conics to lines 1} we can argue as follows.
Assume that the generic line is not jumping for $E$, so that the support of $\JL_E$ is a curve. 
It follows that $\JL_E \times \JL_E$ is a surface, while by \ref{prop:intersecting lines} the variety of intersecting lines $\IQ$ is $3$-dimensional.
As a consequence, there is at least a pair of intersecting lines $L_1, L_2$ such that $E$ restricts trivially to both of them.
By lemma \ref{trivialityoncross}, $E$ is trivial on the conic which is the union of $L_1$ and $L_2$, so that the support of $\JC_E$ not contain the locus of singular conics.

Condition \ref{itm:conics to lines 1} implies condition \ref{itm:conics to lines 2} by lemma \ref{trivialityoncross}.
\end{proof}

\begin{proposition}\label{prop:degree of jumping conics}
The divisor of jumping conics for an instanton $E$ of charge $n$ has degree $n$ (counted with multiplicity).
\end{proposition}
\begin{remark}
By counted with multiplicity we mean the following.
The support of $\JC_E$ can be reducible and have many components $J_i$: $n$ will be the sum of $\deg(J_i) \cdot O_i$ where $O_i$ is the order of jump of $E$ at the generic curve of $J_i$.

If a component $J_i$ is the divisor of reducible conics, then by corollary \ref{cor:jumping lines vs conics} the generic line is jumping for $E$.
In this case $O_i$ is $2k - 1$, where $k$ the order of jump of $E$ at the generic line.
\end{remark}
\begin{proof}
By resolution \eqref{eq:resolution of JC}, the leading term of the Chern character of $\JC_E$ is $n$ times the class of a quadric in 
$\PP(V^*)$.
This shows that the sum
\[
\sum_i \deg(J_i) \cdot \rank(\JC_E, J_i) = 2n
\]
where $\rank(\JC_E, J_i)$ is the rank of $\JC_E$ at the generic point of $J_i$.
Our claim will be that for any component $J$ of $\JC_E$ we have
\[
\rank(\JC_E, J_i) = 2 \cdot O_i
\]

Note that by base change we always have the equality
\[
\rank(\JC_E, J_i) = h^1(\restr{E \otimes \CU}{C})
\]
where $C$ is a general conic in $J_i$.

Assume now that the generic conic $C$ in $J_i$ is smooth.
Then by corollary \ref{splittingonconic} we have
\[
\rank(\JC_E, J_i) = h^1(\restr{E \otimes \CU}{C}) = 2 \cdot h^1(\restr{E}{C}) = 2 \cdot O_i
\]

The only case which is left is that in which $J_i$ is the divisor of reducible conics.
Let $C$ be a general reducible reduced conic with components $L_1, L_2$.
By lemma \ref{lemma:order of jump on crosses} 
\[
h^0(\restr{E \otimes \CU}{C}) = 4k - 2
\]
where $k$ is the order of jump of $E$ at $L_i$.
Finally as $\chi (\restr{E \otimes \CU}{C}) = 0$ we deduce
\[
\rank(\JC_E, J_i) = h^1(\restr{E \otimes \CU}{C}) = 2 \cdot O_i
\]
\end{proof}

\begin{lemma}\label{lemma:order of jump on crosses}
If the generic line is $k$-jumping for $E$, $k > 0$, then $h^0(\restr{E \otimes \CU}{C}) = 4k - 2$.
\end{lemma}
\begin{proof}
Denote the intersection point of $L_1$ and $L_2$ by $P$. 
Tensor the exact sequence
\[
0 \to  \CO_C \to \CO_{L_1} \oplus \CO_{L_2} \to \CO_P \to 0
\]
by $E \otimes \CU$ and take the long exact sequence for sheaf cohomology.
\[
h^0(\restr{E \otimes \CU}{C}) = 4k + 2 - d
\]
where $d \in [0,4]$ is the dimension of the image of
\begin{equation}\label{eq:order of jump on crosses}
H^0(\restr{E \otimes \CU}{L_1}) \oplus H^0(\restr{E \otimes \CU}{L_2}) \to H^0(\restr{E \otimes \CU}{P})
\end{equation}

The following argument shows that $d = 4$ for the generic pair of intersecting $L_1$ and $L_2$.
First, each point $P$ not in the closed orbit lies is the intersection of at least two distinct lines (lemma \ref{cor:3 to 1 cover}).
For each such point $P$ and each pair $L_1$, $L_2$ intersecting at $P$, we get two vectors $e_1,e_2$ in the fiber of $E$ at $P$:
the images of the canonical
\[
\CO_{L_i}(k) \to \restr{E}{L_i}
\]

Note that
\[
\restr{\CU}{L_i}(k) \to \restr{E \otimes \CU}{L_i}
\]
is surjective on global sections as $k > 0$.
It follows that the map \eqref{eq:order of jump on crosses} is actually the restriction
\[
e_1 \otimes H^0(\restr{\CU}{L_1}(k)) \oplus  e_2 \otimes H^0(\restr{\CU}{L_2}(k)) \to E_P \otimes U_P
\]
As $\restr{\CU}{L_i}(k)$ is generated by global sections for $k > 0$, \eqref{eq:order of jump on crosses} becomes
\[
\left( e_1 \oplus e_2 \to E \right) \otimes U_P
\]
so that its image is not $4$-dimensional if and only if $e_1 = e_2$.

Finally, we are going to prove that if for the generic pair of lines the induced $e_1, e_2$ are equal, then there is a section of $E$ defined on the complement of a curve.

The locus of lines whose order of jump is at least $k+1$ is at most a curve, so that its preimage inside the universal line $\CL$ is at most a surface $S_E$.
On the complement of $S_E$ we can define a section of $r_Y^*E$ which (by flatness of $r_Y$) descends to a section of $E$ on $r_Y(\CL \setminus S_E)$, which is an open subvariety of $Y$ as $r_Y$ is flat.

Note that if $Y \setminus r_Y(\CL \setminus S_E)$ contains a surface, then its preimage in $\CL$ is saturated with respect to both $r_Y$ and $r_A$, contradicting the fact that for each two points we can find a chain of lines connecting them (as for any two lines there is a third one intersecting both of them by proposition \ref{prop:intersecting lines}).

To conclude, note that note that a section of $E$ defined on a complement of a curve extends to a global section of $E$, contradicting the stability of $E$.
\end{proof}


\section{Minimal instantons}\label{sec:Minimal instantons}

In this section we provide a complete description of the moduli space on instantons on $Y$ in the case of $c_2=2$.

As explained in remark \ref{rmk:ch(E)}, for any instanton $E$ we have
\[
h^1(E) = c_2(E)-2 = n - 2
\]
As a consequence, there are no instantons with $c_2(E) < 2$. 
\begin{definition}
Instantons such that $c_2(E)$ takes the minimal possible value, i.e. $c_2(E) = 2$ are called minimal instantons.
\end{definition}

The first result that we prove about minimal instantons is the following consequence of theorem \ref{thm:gamma injective}.
\begin{proposition}\label{cor:generic line minimal}
For a minimal instanton $E$ the generic line is not jumping. Moreover, the support of $\JL_E$ is a smooth conic.
\end{proposition}
\begin{proof}
We prove directly the second part of the statement, which includes the first.
From \cite[Prop. 4.10]{kuznetsov2012instanton} we know that the support of $\JL_E$ is $\gamma^{-1}(\Delta_H)$, where $\Delta_H$ is the discriminant inside $\PP(S^2H^*)$ and
\[
\gamma: A \cong A^* \to S^2H^*
\]
is any preimage of $E$ via the projection from $\Mon_n$ to $\CMI_2$.

By \ref{thm:gamma injective} $\gamma$ is injective. 
In the case of minimal instantons
\[
\dim(A) = \dim(S^2H^*) = 3
\] 
so that $\gamma$ is also surjective. 
As $\gamma$ is an isomorphism and $\Delta_H$ is a smooth conic in $\PP(S^2 H^*)$, the jumping divisor of a minimal instanton is a smooth conic.
\end{proof}

The main result in this section is theorem \ref{minimalmoduli}. It provides an $SL_2$-equivariant open embedding of $\CMI_2$ inside $\PP(S^2(A^*))$ and a complete $SL_2$-equivariant description of the complement $\PP(S^2A^*) \setminus \CMI_2$. 

One of the key steps in the proof of theorem \ref{minimalmoduli} is that in the case of minimal instantons, the conditions of theorem \ref{maininstantontheorem} specialize to a simpler single condition, as explained in lemma \ref{minimalequivalentconditions}.
\begin{lemma} \label{minimalequivalentconditions}
In the case $n=2$, condition \ref{maininstantontheorem} \eqref{itm:conditionfiber} is equivalent to condition \ref{maininstantontheorem} \eqref{itm:conditionrank}.
\end{lemma}
\begin{proof}
Assume that $\gamma$ is a fiberwise monomorphism. Then the quotient
\[
0 \to H \otimes \calU \xrightarrow{\gamma '} H^* \otimes \calU^\perp \to F \to 0
\]
is a vector bundle of rank $2$ with $c_1(F) = 0$. Moreover $H^\bullet(F) = 0$ 
 as the other two bundles are acyclic. 

We dualize and we find
\[
0 \to F \to H \otimes V/\calU \xrightarrow{ - {\gamma'}^*} H^* \otimes \calU^* \to 0
\]
Taking long exact sequence in cohomology and substituting $H^\bullet(F) = 0$ 
 we find that $\gamma^{\prime}$ induces an isomorphism
\[
H \otimes V \xrightarrow{ - \hat{\gamma}} H^* \otimes V^*
\]
which is condition \eqref{itm:conditionrank}.

The other way round, assume $\rank(\hat{\gamma}) = 10$, i.e. assume $\hat{\gamma}$ is an isomorphism, and restrict diagram
\[
\begin{tikzcd}
0  \arrow{r}	& H \otimes \CU 	\arrow{r}\arrow{d}{\gamma '}	& H \otimes V \otimes \CO_{Y} \arrow{d}{\hat{\gamma}} \\
0  \arrow{r} 	& H^* \otimes \CU^\perp \arrow{r}						& H^* \otimes V^* \otimes \CO_{Y}
\end{tikzcd}
\]
to any closed point $U$ to find
\begin{equation}
\begin{tikzcd}
0  \arrow{r}	& H \otimes U 	\arrow{r}\arrow{d}{\restr{\gamma '}{U}}		& H \otimes V \arrow{d}{\hat{\gamma}} \\
0  \arrow{r} 	& H^* \otimes U^\perp \arrow{r}			& H^* \otimes V^* 
\end{tikzcd}
\end{equation}
As the map $\hat{\gamma}$ is injective, the map $\restr{\gamma '}{U}$ is injective as well, so that condition \eqref{itm:conditionfiber} holds.
\end{proof}

The correspondence between instantons and curves with theta-characteristics takes a very simple form in the minimal case, as on smooth conics there's a unique choice of a theta-characteristic. 
We want to use this correspondence to give a complete description of $\mathcal{MI}_2$ inside $\PP(S^2A^*)$. 

Let us first recall that, by theorem \ref{thm:jumping sheaf},
$\mathcal{MI}_2$ embeds in $\PP(S^2A^*)$ by sending an instanton $E$ to its curve of jumping lines, which is a conic in $\PP(A)$ and can be therefore thought of as an element of $\PP(S^2 A^*)$.
\begin{proof}(Sketch)
Given a minimal instanton $E$ one can obtain a conic in $\PP(A)$ via theorem \ref{thm:jumping sheaf}.
Conversely, given a conic $C$ 
\[
i_C: C \to \PP(A)
\]
which comes from an instanton via theorem \ref{thm:jumping sheaf}, take the pushforward $i_{C*}\CO_C(-1)$ and decompose it with respect to the exceptional collection
\[
\langle	\CO_{\PP(A)}(-2), \CO_{\PP(A)}(-1), \CO_{\PP(A)} 	\rangle
\]
This will provide a resolution
\[
0 \to H \otimes \CO_{\PP(A)}(-2) \to H^* \otimes \CO_{\PP(A)}(-1) \to i_{C*}\CO_C(-1) \to 0
\]
Finally, one checks that such a resolution is symmetric in $H$ and recovers in this way an element in $A^* \otimes S^2H^*$ corresponding to an instanton.
\end{proof}

In section \ref{sec:SL2action} we have discussed an $SL(W)$-equivariant construction of $Y$: its first step is the choice of an $\SSL(W)$-invariant $3$-dimensional $A \subset \Lambda^2V^*$.
As $\SSL(W)$ acts on $A$, it also acts on $A \otimes S^2H^*$.
In the case of minimal instantons there is another copy of $SL_2$ acting on $A \otimes S^2H^*$, namely $\SSL(H)$.

As the two actions of $SL(W)$ and $SL(H)$ on $A \otimes S^2H^*$ commute, they induce an action of $\SSL(W) \times \SSL(H)$.
The description of the complement of the locus of instantonic $\gamma$ in $A \otimes S^2 H^*$ involves the two $SL(W) \times SL(H)$-invariant divisors which we are about to introduce. 
\begin{definition}\label{def:det and Q}
We denote the degree $3$ divisor in $\PP(A \otimes S^2H^*)$ of degenerate maps from $A^*$ to $S^2H^*$ by $\divQ_3$.

There is a unique divisor of degree $2$ in $\PP(A \otimes S^2H^*)$ which is invariant under the action of $SL(W) \times SL(H)$.
We denote it by $\divQ_2$.
\end{definition}
\begin{remark}\label{rmk:det and Q}
Let us show that $\divQ_2$ is well defined.
It is easy to determine the irreducible components of each direct summand of
\[
S^2(A^* \otimes S^2H) = \left( S^2 A^* \otimes S^2 S^2 H  \right) \oplus \left( \Lambda^2 A^* \otimes \Lambda^2 S^2 H \right)
\] 
with respect to the action of $\SSL(W) \times \SSL(H)$.
They are all the products of irreducible components of the factors.
It follows that there is only one $1$-dimensional $\SSL(W) \times \SSL(H)$ irreducible representation, which shows that $\divQ_2$ is well defined.
Note moreover that, as there are clearly no $\SSL(W) \times \SSL(H)$ invariant elements in $A^* \otimes S^2 H$, the divisor $\divQ_2$ is irreducible.
\end{remark}

\begin{proposition}\label{minimalmoduliabove}
The complement $\partial \Mon_2 = A \otimes S^2 H^* \setminus \Mon_2$ is the union of $\divQ_2$ and $\divQ_3$.
\end{proposition}
\begin{proof}
By lemma \ref{minimalequivalentconditions} we know that $\gamma \in \Mon_2$ if and only if it satisfies condition \eqref{itm:conditionrank}. 
As $\hat{\gamma}$ is a 10 by 10 antisymmetric matrix, it has rank 10 if and only if its Pfaffian does not vanish. 
As the entries of $\hat{\gamma}$ are linear in the coordinates of $A \otimes S^2H^*$, $\partial \Mon_2$ is cut by a degree 5 equation.
Moreover, by \ref{thm:gamma injective} we know that for $\gamma$ to be in $\Mon_2$ it is necessary that it induces an isomorphism from $A^*$ to $S^2 H^*$, so that $\divQ_3$ is contained in $\partial \Mon_2$. 

As $\divQ_3$ has degree 3, we are left with the task of finding a missing quadric, which we claim is $\divQ_2$.
To prove it, we use the joint action of $SL(W)$ and $SL(H)$. 
The instantonic conditions \eqref{itm:conditionfiber}, \eqref{itm:conditionrank} are clearly invariant under the action of both copies of $SL_2$, so that also the missing quadric has the same property.
By remark \ref{rmk:det and Q}, having degree $2$ and being $\SSL(W) \times \SSL(H)$-invariant uniquely determine $\divQ_2$, so that finally $\partial \Mon_2$ is the union of $\divQ_2$ and $\divQ_3$.
\end{proof}

Theorem \ref{minimalmoduli} completes the description of $\mathcal{MI}_2$ inside $\PP(S^2 A^*)$.
In order to state it, we need to choose a notation for two divisors in $\PP(S^2 A^*)$. We do it in the following definition.
\begin{definition}\label{def:det and L}
We denote the degree $3$ divisor in $\PP(S^2 A^*)$ of degenerate conics in $\PP(A)$ by $\Delta_A$. 

There is a unique hyperplane in $\PP(S^2 A^*)$ which is invariant under the action of $SL(W)$.
We denote it by $H_q$ as under the canonical $S^2A^* \cong S^2A$ it is cut by $q$.
\end{definition}

Denote by $F$ the rational surjection
\[
F: \PP(A \otimes S^2H^*) \DashedArrow[->,densely dashed    ] \PP(S^2A^*)
\]
sending $\gamma$ to $\gamma^{-1}(\Delta_H)$.
Note that $F$ has degree $2$ and that by the decomposition of $S^2(A \otimes S^2H^*)$ and $S^2A^*$ into $\SSL(W) \times \SSL(H)$-irreducibles it is the unique $SL(W) \times \SSL(H)$-equivariant map of degree $2$ from $\PP(A \otimes S^2H^*)$ to $\PP(S^2A^*)$.
Another way to describe $F$ is to say that
\begin{equation}\label{eq:F in coordinates}
F(\gamma) = \gamma^{T} \cdot q_H \cdot \gamma
\end{equation}
where $\gamma^{T}$ is the transpose of $\gamma$ and $q_H \in S^2H^*$ is the unique $\SSL(H)$-invariant.

We are now ready to state and prove theorem \ref{minimalmoduli}.
\begin{theorem}\label{minimalmoduli}
The map which associates with an instanton $E$ of charge $2$ its conic of jumping lines is an $\SSL_2$-equivariant isomorphism
\[
\CMI_2 \to \PP(S^2A^*) \setminus \left( \Delta_A \cup H_q \right)
\]
where $\Delta_A$ and $H_q$ are defined in \ref{def:det and L}.
\end{theorem}
\begin{proof}
As $F$ is surjective, it is enough to show that
\begin{equation*}
F^{-1}(H_q) = Q_2	\qquad \qquad
F^{-1}(\Delta_A) = Q_3
\end{equation*}
The first one holds as $F$ is $\SSL(H)$-equivariant of degree $2$ and $H_q$ is $\SSL(W)$-equivariant, so that $F^{-1}(H_q)$ is invariant under the action of both $\SSL(W)$ and $\SSL(H)$. By remark \ref{rmk:det and Q}, the only such quadric is $Q_2$.

As for the second one, by equation \eqref{eq:F in coordinates} we have
\[
\det( F(\gamma)) = \det(\gamma)^2 \det(q_H)
\]
so that $F(\gamma)$ is degenerate if and only if $\gamma$ is degenerate.
\end{proof}

A straightforward consequence of theorem \ref{minimalmoduli} is the existence of a unique minimal instanton $E_0$ with an $SL_2$-equivariant structure.
\begin{corollary}\label{cor:SL_2 equivariant instanton}
There is a unique minimal instanton $E_0$ with an $SL_2$-equivariant structure.
\end{corollary}
\begin{proof}
We have already discussed the existence and uniqueness of an $\SSL_2$-invariant point in $q \in \PP(S^2A^*)$.
By theorem \ref{minimalmoduli}, to show that $q$ comes from an instanton it is enough to check that 
\[
q \notin H_q \cup \Delta_A
\]
By definition of irreducible component, $q \notin H_q$. 
To conclude the proof, recall that $\Delta_A$ is the locus of degenerate conics in $A$ while $q$ is non-degenerate, as its kernel is an $\SSL_2$-submodule of $A$.
\end{proof}
\begin{remark}
It is interesting to play the $2$-rays game on the projectivization $\PP_Y(E_0)$ (see for example \cite{corti2000singularities}).
\end{remark}

\begin{remark}
It is possible to combine proposition \ref{prop:degree of jumping conics} and proposition \ref{prop:singular conics} into another proof of the fact that for each minimal instanton the generic line is not jumping.

The argument is the following: if there is a minimal instanton $E$ such that the generic line is jumping for it, then the support of $\JC_E$ contains the divisor of singular conics. 
This contradicts the fact that the degree of the support of $\JC_E$ is $2$, while the degree of the locus of singular conics is $3$.
\end{remark}

The last fact we prove about $\CMI_2$ is that it is not a fine moduli space.
\begin{proposition}\label{prop:minimal is not fine}
There is no universal family of minimal instantons.
\end{proposition}
\begin{proof}
Assume on the contrary there is a universal family $\CE$ of minimal instantons.
Then, by theorem \ref{thm:jumping sheaf} there is a family $\Phi_{\CL}(\CE(-1))$ of smooth conics with a theta-characteristic.
This means that if we consider over $\PP(A) \times \PP(S^2 A^*)$ the incidence divisor $I$, there is an open subset $I^\circ$ with a line bundle restricting to $\CO(-1)$ on each fiber of the projection to $\PP(S^2 A^*)$.

As $I$ is a smooth divisor of degree $(2,1)$, by Lefschetz hyperplane section theorem $\Pic(I)$ is generated by the restrictions of $\CO(1,0)$ and $\CO(0,1)$ from the ambient space.
As $I^\circ$ is open in $I$, the same holds for $I^\circ$.
It follows that the restriction of any line bundle on $I^\circ$ to a fiber of the projection to $\PP(S^2A^*)$ has even degree.
This provides a contradiction with the existence of $\CE$.
\end{proof}

\section{Instantons of charge $3$}\label{sec:Instantons of charge 3}

In this section we are going to describe the moduli space $\CMI_3$ of instantons of charge $3$.
More precisely, in \ref{sec:maptoGr} we construct a natural dominant map $\beta$ from $\CMI_3$ to $\Gr(3,V)$ and we introduce a class of instantons which we call special.
After this, in \ref{sec:An embedding of CMI_3} we lift the map $\beta$ to an embedding into a relative Grassmannian.
The embedding provides $\CMI_3$ with a natural compactification.
Finally, in \ref{sec:Special instantons} and \ref{sec:Non-special instantons} we treat the cases of special and of non-special instantons separately, focussing on the properties of their jumping lines.

\subsection{A map from $\CMI_3$ to $\Gr(3,V)$} \label{sec:maptoGr}

It turns out that there is a surjective map from $\CMI_3$ to $\Gr(3,V)$, which we will denote by $\beta$. 
Its fibers are either projective spaces or Grassmannians and it is convenient to stratify $\Gr(3,V)$ with respect to their type.
More precisely, there will be two strata and $\beta$ will be smooth on each of them.
This will correspond to a distinction between two kinds of instantons: the special ones and the non-special ones.

\begin{notation}\label{not:stratification of B}
We will denote $\Gr(3,V)$ by $\Bg$.
We will also denote $\Imkappa$, that is to say the image of $\kappa$ in $\Gr(3,V)$, by $\Bs$ and its complement $\Bg \setminus \Bs$ by $\Bn$.
\end{notation}
\begin{notation}\label{not:3-monad}
When $c_2(E) = 3$, the space $H^{\prime}$ in monad \eqref{MONAD} is $1$-dimensional. For this reason we will write
\begin{equation}\label{eq:3-monad}
H \otimes \CU \to H^*\otimes \CU^\perp \to \CO_Y
\end{equation}
for any monad associated with an instanton $E$ of charge $3$.
The first map in the complex 
\begin{equation}\label{special instanton spectral sequence}
H \xrightarrow{\gamma} H^* \otimes A \to V^*
\end{equation}
obtained by applying $\Hom(\CU, -)$ to \eqref{eq:3-monad} is denoted by $\gamma$.
\end{notation}

\begin{lemma}
For a charge $3$ instanton the following properties are equivalent
\begin{itemize}
\item $\ext^1(\calU,E) \neq 0$
\item $\hom(\calU,E) \neq 1$
\end{itemize}
\end{lemma}
\begin{proof}
We can compute both $\Hom(\CU,E)$ and $\Ext^1(\CU, E)$ using the monad \eqref{eq:n-monad}: they are the cohomology of the complex \eqref{special instanton spectral sequence}.
Note that $\gamma$ is injective as its kernel would contribute to $\Ext^{-1}(\CU, E)$.

It follows from \eqref{special instanton spectral sequence} that
\[
\hom(\calU,E) - \ext^1(\calU,E) = 1
\]
showing the equivalence of the two conditions.
\end{proof}

The above lemma motivates the following definition.
\begin{definition} \label{special instanton}
An instanton $E$ of charge 3 is \emph{special} if $\Ext^1(\calU,E) \neq 0$. Equivalently, an instanton is \emph{non-special} if $\Hom(\calU,E) = \CC$.
\end{definition}
We will see that special instantons are special in many ways. 
Essentially by definition they do not have a canonical map into $\CU^*$, while most of instantons have one. 
They are the only instantons having $2$-jumping lines (see proposition \ref{prop:2-jump again}). 
They are the only instantons whose associated theta-characteristic is not locally free (see proposition \ref{prop:reducible theta}).

We are now going to introduce the main character in the description of $\CMI_3$.

\begin{definition}\label{beta2}
The map 
\[
\beta: \CMI_3 \to \Bg
\] 
sends $E$ to $\Ext^1( - , E)$ of the tautological map $V^* \otimes \calO_Y \to \calU^*$.
\end{definition}

So far it is not clear that
\[
\Ext^1(\CU^* , E) \to \Ext^1(V^* \otimes \CO_Y, E) \cong V \otimes H^1(E) \cong V
\]
is a $3$-dimensional subspace of $V$ and that $\beta$ is regular: we prove it in proposition \ref{prop:beta regular} by means of lemma \ref{no maps from Uperp}.

\begin{lemma} \label{no maps from Uperp}
If $E$ is an instanton with $c_2(E) \geq 3$, then $\Hom(\calU^\perp, E) = 0$.
\end{lemma}
\begin{proof}
Both $E$ and $\CU^\perp$ are $\mu$-stable (see lemma \ref{lemma:stability on Y}) of slope respectively $\mu(E) = 0$ and $\mu(\CU^\perp) = -1/3$. 
It follows that there is no map of rank $1$ from $\CU^\perp$ to $E$, as otherwise
\[
\mu(E) > \mu(\Image) > \mu(\CU^\perp)
\]
which contradicts the fact that the slope of the image $\mu(\Image)$ is an integer.

As $E$ is torsion-free, there are no maps of rank $0$, so that we only have to check that there are no maps of rank $2$. 
To prove it, assume there is such a map and complete it to an exact sequence.
\[
0 \to \mathcal{L} \to \calU^\perp \to E \to Q \to 0
\]
Then $\CL$ is reflexive of rank 1 and therefore it is a line bundle. 
Moreover, $c_1(\mathcal{L}) \leq -1$ by stability of $\calU^\perp$ and $c_1(\mathcal{L}) \geq -1$ by stability of $E$, so that $\mathcal{L} \cong \calO(-1)$. 
To conclude the proof, we use equations \eqref{Cherntable} and \eqref{eq:ch(E)} to compute $\ch(Q)$ and derive a contradiction:
\[
\ch(Q) = \ch(E) - \ch(\calU^\perp) + \ch(\calO(-1))
\]
yields 
\[
\ch(Q) = \left(3L - c_2(E)\right) - P
\] 
so that the leading term is negative when $c_2(E) \geq 3$.
\end{proof}
The bound in lemma \ref{no maps from Uperp} is sharp as instantons of charge $2$, by monadic description, always have a $2$-dimensional space of maps from $\calU^\perp$.

\begin{proposition}\label{prop:beta regular}
$\beta$ is a regular map from $\CMI_3$ to $\Bg$.
\end{proposition}
\begin{proof}
By self duality of $E$, $\ext^1( \calU^* , E)  = \ext^1( E, \calU)$, and by monadic description \eqref{MONAD} we have $\ext^1( E, \calU) = 3$. 
Moreover the kernel of $\Ext^1(\calU^*,E) \to V \otimes H^1(E)$ sits inside $\Hom(\calU^\perp, E)$, which vanishes by lemma \ref{no maps from Uperp}, so that
\[
\Ext^1(\CU^*,E) \subset V \otimes H^1(E)
\]
gives a point in $\Bg$.

We prove that the map $\beta$ is regular by constructing it in families. Take any flat family $\CE$ of instantons of charge 3 over a base scheme $S$. Take the tautological map $V^* \otimes \CO_Y \to \CU^*$ and out of it construct
\[
\CHom(\CU^*, \CE) \to V \otimes \CHom(\CO_Y,{\CE})
\]

Push it forward to $S$. By cohomology and base change, we obtain a (shifted by 1) rank 3 subbundle of $V \otimes \CH^1(E)$, where $\CH^1(\CE)$ is a line bundle. 
By twisting the subbundle by $\CH^1(\CE)^*$ we finally get a rank 3 subbundle of $V \otimes \CO_S$, which by defining property of $\Bg$ gives a map from $S$ to $\Bg$. 
By the universal property for the coarse moduli space of instantons, we induce a unique map from $\CMI_3$ to $\Bg$.
\end{proof}

Lemma \ref{beta2second} and proposition \ref{prop:beta1 is beta2} deal with different descriptions of $\beta$.
\begin{lemma}\label{beta2second}
$\beta(E)$ is the $3$-dimensional subspace of $V$ given by $\Ext^1(E, -)$ applied to the tautological $\CU \to V \otimes \CO_Y$. 
\end{lemma}
\begin{proof}
By self-duality of $E$, there is a functorial isomorphism between $\Ext^1( - , E)$ and $\Ext^1(E, -^*)$ inducing a commutative square
\begin{equation}\label{diag:beta2second}
\begin{diagram}
\Ext^1(E, V \otimes \calO) 	& \lTo 	& \Ext^1(E,\calU) 	\\
	\dTo_{\cong}			&		&	\dTo_{\cong}	\\
\Ext^1(V^* \otimes \calO, E)	& \lTo	& \Ext^1(\calU^*,E)
\end{diagram}
\end{equation}
\end{proof}
\begin{proposition}\label{prop:beta1 is beta2}
Let \eqref{eq:3-monad} be a monad for $E$. $\beta(E)$ is the image of the injective map
\[
H^* \cong \Hom(\CU^\perp, H^* \otimes \CU^\perp) \to \Hom(\CU^\perp, \CO_Y) \cong V
\]
obtained from \eqref{eq:3-monad} by applying $\Hom(\CU^\perp, -)$.
\end{proposition}
\begin{proof}
Under the canonical identification of $B$ with the Grassmannian of $2$-dimensional quotients of $V$, $\beta(E)$ is the bottom row the the following commutative diagram.
\begin{equation}\label{diag:beta1 is beta2}
\begin{diagram}
\Hom(\CU^\perp, H^*\otimes \CU^\perp)	&\rTo^\cong&\Hom(\CU^\perp,\widetilde{E})	&\lTo_0&0=\Hom(\CO_Y,\widetilde{E})\\
								& \rdTo	&		\dTo				&				&	\dTo_0	\\
								&  		& \Hom(\CU^\perp, \CO_Y)	&\lTo_{\quad \cong} &V\otimes H^0(\CO_Y)\\
								&		& 		\dTo				&				&\dTo^\cong		\\
								&		& \Ext^1(\CU^\perp, E)		&\lTo_{\quad \beta(E)}&V \otimes H^1(E)
\end{diagram}
\end{equation}
The upper left triangle is induced by the monad \eqref{eq:3-monad}, the middle and right column are induced by the acyclic extension sequence \eqref{eq:acyclic extension}. 
The horizontal maps are induced by $\CU^ \perp \to V^* \otimes \CO_Y$.
The commutativity of \eqref{diag:beta1 is beta2} shows that 
\[
\Hom(\CU^\perp, H^*\otimes \CU^\perp) \to \Hom(\CU^\perp, \CO_Y)
\] 
is the kernel of $\beta(E)$. 
\end{proof}
\begin{notation}
By proposition \ref{prop:beta1 is beta2}, for a family of instantons $\CE$ over a base $S$ we will write
\begin{equation}\label{eq:3-monad family}
\CH \boxtimes \CU \to \beta^*\CK \boxtimes \CU^\perp \to \CO_{S \times Y}
\end{equation}
for its decomposition with respect to the usual collection \eqref{eq:coll}.
\end{notation}

Next, we characterize the special instantons in terms of their image under $\beta$.
Recall that we have introduced a stratification $\{ \Bs, \Bn \}$ for $\Bg = \Gr(3,V)$, and that the close stratum $\Bs$ is the image under $\kappa$ of $\PP(A)$.
\begin{proposition}\label{beta special}
An instanton $E$ is special if and only if $\beta(E) \in \Bs$. 
\end{proposition}
\begin{proof}
$E$ is special if and only if in the sequence \eqref{special instanton spectral sequence} the map
\[
H^* \otimes A \xrightarrow{\epsilon} V^*
\]
is not surjective. Instead of computing $\Ext^1( - , E)$ only on $\calU$, compute it on the universal map $\calU \to A^* \otimes \calU^\perp$. The result is a commutative diagram
\begin{diagram}
A \otimes H^*		& \rTo_{\epsilon}			&	V^*		\\
\uTo^{Id}_{\cong}	& 						& \uTo_{a(v,-)}	\\
 A \otimes H^*		& \rTo_{A \otimes \beta(E)}	& A \otimes V
\end{diagram}
showing that $\epsilon$ has a cokernel if and only if there is $v \in V$ such that $A(\beta(E),v) = 0$. As $\dim \beta(E) = 3$, this can happen if and only if $\dim A(v, -) \leq 2$, which is the case if and only if there is $a_v \in A$ such that $v = \ker a_v$. Finally, if $v = \ker a_v$, there's a unique choice for $\beta(E)$ and it is $\korth(a_v)$.
\end{proof}

The following is a more accurate description of the moduli space $\CMI_3^{\sfs}$ of special instantons.
Note that so far we have defined what a single special instanton is, but we do not have a notion of family of special instantons.

Let $S$ be a scheme and denote the projections from $Y \times S$ to $Y$ and $S$ respectively by $\pi_Y$ and $\pi_S$.

\begin{definition}\label{def:special instanton definition}
We will denote by $\CMI_3^{\sfs}$ the functor which associates with $S$ the set of families of instantons $\CE$  such that the schematic support $R^1\pi_{\CMI_3*}\CHom(\pi_Y^*\CU,\CE)$ is the whole $S$.
\end{definition}

\begin{remark}
The formation of $R^1\pi_{\CMI_3*}\CHom(\pi_Y^*\CU,\CE)$ commutes with arbitrary base changes $T \to S$. 
The reason is that by monad \eqref{eq:3-monad} all higher pushforwards 
\[
R^{>1}\pi_{\CMI_3*}\CHom(\pi_Y^*\CU,\CE)
\] 
vanish.
\end{remark}
\begin{remark}
Also the formation of the support of $R^1\pi_{\CMI_3*}\CHom(\pi_Y^*\CU,\CE)$ commutes with arbitrary base change.
This is true as the rank of $R^1\pi_{\CMI_3*}\CHom(\pi_Y^*\CU,\CE)$ is never greater than $1$, so that its the schematic support is the first degeneracy locus of the vector bundle map
\begin{equation}\label{eq:special instanton definition}
A \otimes \beta^*\CK \to V^* \otimes \CO_S
\end{equation}
To conclude, note that the formation of the above map commutes with arbitrary base change by functoriality of the decomposition with respect to a full exceptional collection.
\end{remark}

\begin{proposition}\label{prop:special as fiber product}
There is a cartesian diagram
\begin{equation}\label{diag:special as fiber product}
\begin{diagram}
\CMI_3^{\sfs}		& \rInto		& \CMI_3		\\
\dTo^{\beta^{\sfs}}	&\square		& \dTo^{\beta}	\\
\PP(A)		&\rInto^{\kappa}	& \Bg
\end{diagram}
\end{equation}
where the embedding of $\CMI_3^{\sfs}$ in $\CMI_3$ is the natural one.
\end{proposition}
\begin{proof}
Given a family of special instantons $\CE$ on a scheme $S$, consider its associated monad \eqref{eq:3-monad family}.
Recall that $\beta(\CE)$ is by definition $\beta^*{\CK} \to V \otimes \CO_S$, whose cokernel is, by definition of family of special instantons, supported everywhere on $S$.
It follows that the map from $S$ to $\Bg$ factors uniquely via the degeneracy locus of \eqref{eq:special instanton definition}, which by proposition \ref{prop:embed P(A)} is $\Bs$.
As a consequence, we have constructed the maps in \eqref{diag:special as fiber product}.

In the other direction, given a family of instantons $\CE$ over $S$ such that $\beta(\CE)$ factors via $\Bs$, we clearly have that
\[
R^1\pi_{\CMI_3*}\CHom(\pi_Y^*\CU,\CE)
\]
is supported on the whole $S$ as it comes by pullback from $\Bs$, where by \ref{prop:embed P(A)} it is a line bundle.
\end{proof}

By definition \ref{special instanton}, any non-special instanton has a canonical map $E \to \CU^*$.
By stability of $E$ and $\CU^*$, this map is injective, so that we have an exact sequence
\[
0 \to E \to \CU^* \to Q_E \to 0
\]
where $Q_E$ is by definition the cokernel of $E \to \CU^*$. The Chern character computation yields 
\begin{equation}\label{ch Q_E}
ch(Q_E) = H + \frac{7L}{2} - \frac{P}{6}
\end{equation}
so that $Q_E$ is supported on a linear section of $Y$. Later on, in theorem \ref{thm:general instanton}, we will prove that for any non-special $E$ such a $Q_E$ is always of the form $\Phi_{\Bl}\left(\calO_{\PP(K)}\left(2\right)\right)$. 
In view of this, we prove that if $Q_E$ is isomorphic to $\Phi_{\Bl}\left(\calO_{\PP(K)}\left(2\right)\right)$, then $\beta(E) = K$.

\begin{proposition}\label{beta3}
For every short exact sequence
\[
0 \to E \to \calU^* \to \Phi_{\Bl}\left(\calO_{\PP(K)}\left(2\right)\right) \to 0
\]
where $E$ is an instanton and $K \in \Bg \setminus  \Dtr$, it is true that $\beta(E) = K$.
\end{proposition}
\begin{proof}
We will compute $\beta(E)$ using the slight modification of definition \ref{beta2} which is described in lemma \ref{beta2second}.

We need to compute $\Ext^1(E,-)$ on the tautological sequence
\[
0 \to \calU \to V \otimes \calO_Y \to V/\calU \to 0
\]
As $\calU^*$ lies in the left orthogonal to all bundles in the sequence, we get isomorphisms
\begin{equation}\label{eq:beta3 3}
\Ext^1(E,-) = \Ext^2\left(\Phi_{\Bl}\left(\calO_{\PP(K)}\left(2\right)\right),-\right) = 
\Ext^2(\calO_{\PP(K)}(2),\Phi_{\Bl}^!(-))
\end{equation}

Our next goal is to check that the image of
\begin{equation}\label{beta3 eq2}
\Ext^2(\calO_{\PP(K)}(2), \Phi_{\Bl}^!(\calU) ) \to V \otimes \Ext^2(\calO_{\PP(K)}(2), \Phi_{\Bl}^!(\calO_Y)) 
\end{equation}
is $K$.
We will apply $\Phi_{\Bl}^!$ to the tautological injection $\CU \to V \otimes \CO_Y$. In order to do it, recall the identifications of lemma \ref{lemma:right adjoint to Bl}. 
Note also that over $\PP_Y(\CU)$ there is a commutative diagram
\[
\begin{diagram}
				& 		& \calO_{\PP_{Y}(\CU)}(-E)		 		\\
				&\ldTo	& \dTo_{\mathfrak{e}_V(-E)}			\\
p_Y^* \calU(h - E)	&\rTo		&V \otimes \calO_{\PP_{Y}(\CU)}(h-E)
\end{diagram}
\]
where the diagonal map is the twisted relative tautological injection on $\PP_Y(\CU)$ and the horizontal map is the pullback of the tautological injection on $Y$.
If we push the diagram forward to $\PP(V)$ we find another commutative diagram
\begin{equation}\label{beta3 diagram}
\begin{diagram}
				& 		& I_{\Imsigma}		 			\\
				&\ldTo	& \dTo_{\Fe_V(-E)}				\\
\Phi_{\Bl}^!(\CU)[-1]	&\rTo		&V \otimes \Phi_{\Bl}^!(\CO_Y)[-1]
\end{diagram}
\end{equation}
To find the map \eqref{beta3 eq2}, apply the functor $\Ext^2(\calO_{\PP(K)}(2), - )[1]$ to the above diagram. By Grothendieck duality
\[
\Ext^2(\calO_{\PP(K)}(2), - )[1] \cong H^1_{\PP(K)}(Lj_K^*( - ))
\]
Applying $H^1_{\PP(K)}(Lj_K^*( - ))$ to diagram \eqref{beta3 diagram}.
The result is
\begin{equation}\label{beta3 diagram2}
\begin{diagram}
							&				&H^1(Lj_K^*I_{\Imsigma})			\\
							&\ldTo^{\cong}		&\dTo_{H^1(\mathfrak{e}_V)}	\\
H^1(Lj_K^*\Phi_{\Bl}^!(\calU)[-1]) 	&\rTo_{\beta(E)}	& V \otimes H^1(Lj_K^*I_{\Imsigma}(h))	
\end{diagram}
\end{equation}
where the diagonal arrow is an isomorphism as its cokernel is $\calO(-h)$ is acyclic on $\PP (K)$.
It follows that the image of $\beta(E)$ coincides with the image of $H^1(\mathfrak{e}_V)$.

As we assumed $K \notin \Dtr$, we know by lemma \ref{lemma:Imkappa is unexpected dimension} and lemma \ref{lemma:Imkappa in Dtr} that the intersection of $\PP(K)$ with $\Imsigma$ has expected dimension, so that $I_{\Imsigma}$ pulls back to an ideal sheaf on $\PP(K)$.
More precisely, $Lj_K^*I_{\Imsigma}$ is the ideal of a length $4$ subscheme of $\PP(K)$.
We will denote $\PP(K) \cap \Imsigma$ by $K_4$ and its ideal in $\PP(K)$ by $I_4$, so that the vertical arrow in the diagram \eqref{beta3 diagram2} becomes
\begin{equation}\label{beta3 eq4}
H^1_{\PP(K)}(\mathfrak{e}_V): H^1(I_4) \to V \otimes H^1(I_4(h))
\end{equation}

Note hat the restriction to $\PP(K)$ of the Euler map $\mathfrak{e}_V$ factors via the Euler map $\mathfrak{e}_K$ of $K$:
\[
\begin{diagram}
H^1(I_4)					&						&					\\
\dTo^{H^1(\mathfrak{e}_K)}	& \rdTo^{H^1(\mathfrak{e}_V)}	& 					\\
K \otimes H^1(I_4(h)) 		& \rTo_{j_K} 				& V \otimes H^1(I_4(h))	
\end{diagram}
\]
In order to conclude the proof, we only need to check that $H^1(\mathfrak{e}_K)$ is an isomorphism. 
As the cokernel of the Euler map is $\CT_{\PP(K)}$, if
\begin{equation}\label{eq:T otimes I_4}
H^1(\CT_{\PP(K)} \otimes I_4) = H^2(\CT_{\PP(K)} \otimes I_4) = 0
\end{equation}
then $H^1(\mathfrak{e}_K)$ is an isomorphism.

To check this last vanishing, we use a resolution for $I_4$. Using the fact that $\PP(K) \notin \Dtr$, we will prove that $K_4$ is always cut by two conics, so that $I_4$ has a Koszul resolution
\begin{equation*}
0 \to \calO(-4) \to \calO(-2)^{\oplus 2} \to I_4 \to 0
\end{equation*}
For any scheme of length 4 in $\PP(K)$ there is at least a pencil of conics containing it. 
Take a pencil of conics in $\PP(K)$ containing $K_4$: if they do not have common components, then the intersection is $K_4$, otherwise the conics in the pencil share a component.
If they share a component, $K_4$ is contained in a line with an embedded point, and therefore intersects the line in a scheme of length at least $3$, against the fact that $\PP(K) \notin \Dtr$.

Finally,
\[
H^\bullet(\CT_{\PP(K)}(-2))  = H^\bullet(\CT_{\PP(K)}(-4))  = 0
\]
imply that \eqref{eq:T otimes I_4} is true.
\end{proof}

\subsection{An embedding of $\CMI_3$ in a relative Grassmannian}\label{sec:An embedding of CMI_3}

The monadic data for the construction of an instanton of charge 3 consist of two maps: one from $H \otimes \CU$ to $H^* \otimes \CU^\perp$, the other from $H^* \otimes \CU^\perp$ to $\CO_Y$. 
When we constructed the map $\beta$ in section \ref{sec:maptoGr} we used only the second one. By taking also the first one into account we expect to construct an embedding of $\CMI_3$.

Recall that we are using notation \ref{not:stratification of B}, so that $\Bg$ is $\Gr(3,V)$ and $\{\Bn, \Bs\}$ is a stratification of $\Bg$.
We will construct a map $\Gamma$ from $\CMI_3$ to $\RelGr$ which will fit in a commutative diagram
\[
\begin{diagram}
\CMI_3	& \rTo^{\Gamma}	& \RelGr		\\
		&\rdTo_{\beta}		& \dTo_{\pi_{\Bg}}	\\
		&				& \Bg
\end{diagram}
\]
where $\pi_{\Bg}$ is the canonical projection.

Let $\CE$ be a family of instantons parametrized by $S$. 
Denote by $\pi_S$ and $\pi_Y$ the projections from $S \times Y$ to $S$ and $Y$. 
Denote by $\beta_{\CE}$ the composition of the map $S \to \CMI_3$ corresponding to $\CE$ with $\beta$.
The relative Beilinson spectral sequence provides us with a monad for $\CE$:
\begin{equation}\label{RELATIVE MONAD}
\CH \boxtimes \CU \xrightarrow{\gamma_{\CE}} \beta_{\CE}^*\CK  \boxtimes \CU^\perp \xrightarrow{\beta_{\CE}} \CO_{S \times Y}
\end{equation}
where $\CK$ is the tautological bundle on $\Bg$ (see notation \ref{not:3-monad}).

\begin{remark}
In order to construct a map $\Gamma$ from $\CMI_3$ to $\RelGr$ which is relative to $B$, it is enough to associate naturally to any family of instantons $\CE$ a rank $3$ subbundle of $\beta_{\CE}^*\CK \otimes A$.
\end{remark}
\begin{lemma}\label{lemma:definition gamma}
Given a family of instantons $\CE$ parametrized by $S$,
\begin{equation}\label{eq:definition gamma}
\pi_{S*}\CHom(\pi_Y^*\CU, \CH \boxtimes \CU) \xrightarrow{\pi_{S*}\CHom(p_Y^*\CU, \gamma)} \pi_{S*}\CHom(\pi_Y^*\CU, \beta_{\CE}^*\CK \boxtimes \CU^\perp)
\end{equation}
is a rank $3$ subbundle of $\beta_{\CE}^*\CK \otimes A$.
\end{lemma}
\begin{proof}
First, note that by K\"unneth formula the map \eqref{eq:definition gamma} becomes
\[
\CH \otimes \CO_S \to \beta_{\CE}^*\CK \otimes A
\]
Next, let $s \in S$ be a closed reduced point and $E_s$ the corresponding instanton.
By base change and flatness of $p_S$, it is enough to check that for each closed reduced point of $S$ the map
\begin{equation}\label{eq:definition gamma at s}
\Hom(\CU, H \otimes \CU) \to \Hom(\CU, \beta(E_s) \otimes \CU^\perp)
\end{equation}
is injective.
This holds as a non-trivial kernel of \eqref{eq:definition gamma at s} contributes to $\Hom^{-1}(\CU, E_s)$, which vanishes.
\end{proof}

\begin{definition}\label{def:Gamma}
$\Gamma$ takes a family of instantons $\CE$ to the rank $3$ subbundle \eqref{eq:definition gamma}.
\end{definition}

Our next goal is to describe the closure of the image of $\Gamma$. Denote by $Z_3$ the zero locus on $\RelGr$ of the composition
\begin{equation}\label{eq:imageofGamma}
\CTaut^{rel}_3 \to A \otimes \CK \to A \otimes V \otimes \CO_{\RelGr} \to V^* \otimes \CO_{\RelGr}
\end{equation}
where the first map is the relative tautological injection, the second one is the pullback of the tautological injection from $\Bg$ tensored by $A$ and the third one is the evaluation of forms in $A$ on $V$. 

First, we describe $Z_3$. By proposition \ref{prop:embed P(A)} there is an exact sequence
\begin{equation}\label{eq:describe Z}
0 \to \CKer \to A \otimes \CK \to V^* \otimes \CO_{\Bg} \to \kappa_*\left(\CO_{\PP(A)}(2)\right) \to 0
\end{equation}
which defines a reflexive sheaf $\CKer$.

\begin{definition}\label{def:G and F}
The restriction of $\CKer$ to $\Bn$, that is to say the kernel of
\[
A \otimes \restr{\CK}{\Bn} \to V^* \otimes \CO_{\Bn}
\] 
is a rank $4$ vector bundle over $\Bn$ denoted by $\CKer^{\sfn}$.
The kernel of
\[
A \otimes \restr{\CK}{\Bs} \to V^* \otimes \CO_{\Bs}
\]
is a rank $5$ vector bundle over $\Bs$ denoted by $\CKer^{\sfs}$.
\end{definition}

\begin{lemma}\label{lemma:Z_3 over Imkappa}
The restriction of $Z_3$ to $\Bn$ is $\PP_{\Bn}(\CKer^{\sfn *})$.
The restriction of $Z_3$ to $\Bs$ is  $\Gr_{\Bs}(3,\CKer^{\sfs})$.
\end{lemma}
\begin{proof}
The formation of zero loci of maps of vector bundles commutes with arbitrary base change.
\end{proof}

\begin{lemma}\label{lemma:Z_3 irreducible}
$Z_3$ is the zero locus of a regular section of $\CTaut^{rel*}_3 \otimes V^*$. Moreover, $Z_3$ is irreducible and $\Gamma$ factors via $Z_3$.
\end{lemma}
\begin{proof}
$Z_3$ is schematically cut by a section of the rank 15 vector bundle
\[
\CTaut^{rel}_3 \otimes V
\]
by its definition in \eqref{eq:imageofGamma}. 
We only need to check that the section is regular, i.e. that $\cod(Z_3) = 15$. As $\dim(\RelGr) = 24$, it is enough to check that $\dim(Z_3) \leq 9$. 

In the notation of definition \ref{def:G and F} and by lemma \ref{lemma:Z_3 over Imkappa}
\[
Z_3 = \PP_{\Bn}(\CKer^{\sfn *}) \sqcup \Gr_{\Bs}(3,\CKer^{\sfs})
\]
This gives a stratification of $Z_3$ into an open $9$-dimensional subscheme and a closed $8$-dimensional subscheme, so that $\dim(Z_3) \leq 9$.

The $9$-dimensional open subscheme is irreducible and its complement is $8$-dimensional, so that, if $Z_3$ were reducible, one of its components would be at most $8$-dimensional. 
This is impossible as $Z_3$ is cut by 15 equations inside a Cohen-Macaulay scheme of dimension 24.
It follows that $Z_3$ is irreducible.

Finally, by definition of $\Gamma$, the pullback of $\CHom(\CU, - )$ of
\[
\CTaut^{rel}_3 \to A \otimes \CK \to V^* \otimes \CO_{\RelGr}
\]
is the relative monad \eqref{RELATIVE MONAD}.
As the relative monad is a complex, $\Gamma$ factors via $Z_3$.
\end{proof}

We are now going to prove that the map $\Gamma$ is an embedding of $\CMI_3$ into $\RelGr$ and to show that the closure of its image is $Z_3$.

\begin{theorem}\label{thm:compactify MI_3}
$\Gamma$ is an embedding and the closure of its image is $Z_3$.
\end{theorem}
\begin{proof}
Over $\RelGr \times Y$ there are canonical maps
\begin{equation} \label{eq:reconstruction MI3}
\CTaut^{rel}_3 \boxtimes \CU \to \CK \boxtimes \CU^\perp \to \CO_{\RelGr} \boxtimes \CO_Y
\end{equation}
The first map is the composition
\[
\CTaut^{rel}_3 \boxtimes \CU \to A \otimes \CK \boxtimes \CU \to \CK \boxtimes \CU^\perp
\]
where the first map is the relative tautological injection and the second is the unique $SL_2$-equivariant one.
The second map is the unique $SL_2$-invariant element of 
\[
\Hom_{\RelGr}(\CK \boxtimes \CU^\perp, \CO_{\RelGr} \boxtimes \CO_Y) = V^* \otimes V
\]

By definition of $Z_3$ we have that the restriction of $\eqref{eq:reconstruction MI3}$ to $Z_3 \times Y$ is a complex.
By theorem \ref{thm:faenzi smooth}, the moduli space $\CMI_3$ is non-empty. 
Given an instanton $E$, by definition of $\Gamma$ the restriction of $\eqref{eq:reconstruction MI3}$ to $\Gamma(E) \times Y$ is a monad for $E$.
It follows that on a non-trivial open subset $Z_3^o \subset Z_3$, the only non-trivial cohomology of the complex \eqref{eq:reconstruction MI3} is the middle one, and that it is a family of instantons.

Summarizing, for any family of instantons $\CE$ over a base $S$, there is an induced map to $Z_3^o$ given by $\Gamma(\CE)$.
Conversely, for any map from $S$ to $Z_3^o$ the pullback of \eqref{eq:reconstruction MI3} is a family of instantons.
\end{proof}

\subsection{Special instantons}\label{sec:Special instantons}

By proposition \ref{prop:special as fiber product}, the moduli space $\CMI_3^{\sfs}$ of special instantons embeds in $\Gr_{\Bs}(3,\Ker^{\sfs})$.
In this section we are going to provide an instantonic interpretation for the projection to $\Bs$ and we are going to show that such a projection is surjective.

\begin{proposition}\label{prop:2-jump again}
Given an instanton $E$ and a line $L$, the order of jump of $E$ at $L$ is $2$ if and only if $\beta(E) = \kappa(a_L)$. In particular, $E$ has a 2-jumping line if and only if it is special.
\end{proposition}
\begin{proof}
Note that by corollary \ref{cor:no n-jumps} there are no jumping lines of order greater than $2$, so that $L$ is a $2$-jumping line for $E$ if and only if $H^1(\restr{E}{L}) \neq 0$.

Recall from corollary \ref{cor:lines through P and Q} that
\[
L = \PP \left( \kappa(a_L) / \ker (a_L) \right)
\]
so that 
\[
\restr{\CU^\perp}{L} \cong (A(\ker a_L, -) \otimes \CO_L) \oplus \CO_L(-1)
\]

It follows that if we restrict a monad \eqref{MONAD} for $E$ to $L$ we get
\[
H \otimes (\CO_L \oplus \CO_L(-1)) \to \beta(E) \otimes \big( A(\ker a_L, -) \otimes \CO_L \oplus \CO_L(-1) \big) \to \CO_L 
\]
The only contribution to $H^1(\restr{E}{L})$ comes from the rightmost term, so that $H^1(\restr{E}{L})$ is not zero if and only if the composition
\[
\beta(E) \otimes A(\ker a_L, -) \to V \otimes V^* \to \CC
\]
vanishes.

As $\beta(E)$ is $3$-dimensional and $A(\ker a_L, -)$ is $2$-dimensional, this happens if and only if
$\kappa(a_L) = \beta(E)$.
\end{proof}

Next, we will prove that for each line $L$ there is an instanton $E$ such that $\beta(E) = \kappa(L)$.
By proposition \ref{prop:2-jump again} it is enough to show that $L$ is a $2$-jumping line for $E$. 

Let $L$ be a line in $Y$ and let $K \subset V$ be a $3$-dimensional space such that $\kappa(L) \in \PP(K)$.
By lemma \ref{lemma:sections through L}, we can choose $K$ such that $Y_K$, the linear section of $Y$ induced by $K$, is smooth.
Note that in this case, by lemma \ref{lemma:p_{Y_K} isomorphism}, $Y_K$ is the blow up of $\PP(K)$ in $4$ distinct points such that no $3$ of them lie on a line.
Denote the $4$ components of the exceptional divisor by $L, L_1, L_2, L_3$.
In this notation, the following proposition holds.
\begin{proposition}\label{prop:special instantons on demand}
Given a surjective map
\begin{equation}\label{eq:special instantons on demand -1}
\CU \to \CO_{Y_K}(h-2L)
\end{equation}
denote its kernel by $E(-1)$. Then $E$ is an instanton and $L$ is a $2$-jumping line for $E$.
\end{proposition}
\begin{proof}
We need to prove that $H^\bullet(E(-1)) = 0$, that $\ch(E) = 2 - 3L$, that $E$ is locally free and that it is $\mu$-stable.

The vanishing of $H^\bullet(E(-1))$ follows from \ref{lemma:splitting on lines} and from the exact sequence
\[
0 \to \CO_{Y_K}(-L) \to \CO_{Y_K}(h-2L) \to \CO_{h-L}(-1) \to 0 
\]
which shows that $\CO_{Y_K}(h-2L)$ is acyclic.

The Chern character computation follows from \eqref{Cherntable} and from the fact that it is easy to compute the Euler characteristic of several twists of $\CO_{Y_K}(h- 2L)$ by means of the identification of $Y_K$ with the blow up of $\PP(K)$ in $4$ points such that no $3$ of them are collinear.

The fact that $E$ is locally free follows from the isomorphism
\[
\Tor_i(E, - ) \cong \Tor_{i+1}(\CO_{Y_K}(h-2L), - )
\]
and from the fact that locally $\CO_{Y_K}(h-2L)$ is the structure sheaf of a Cartier divisor, so that $\Tor_{> 1}(\CO_{Y_K}(h-2L), - ) = 0$.

The stability of $E$ can be checked via Hoppe's criterion \ref{thm:hoppe}.
In order to do it, it is enough to check that $H^0(E)$ vanishes.
This is equivalent to the fact that
\begin{equation}\label{eq:special instantons on demand}
H^0(\CU^*) \to H^0(\CO_{Y_K}(h-2L))
\end{equation}
is surjective.
Note that the restriction map $\CU^* \to \CU_{Y_K}$ is an isomorphism in cohomology, as its kernel $\CU$ is acyclic, so that the surjectivity of \eqref{eq:special instantons on demand} is equivalent to that of
\begin{equation}\label{eq:special instantons on demand 1}
H^0(\restr{\CU^*}{Y_K}) \to H^0(\CO_{Y_K}(h-2L))
\end{equation}
Finally, as $\CU^*_{Y_K} \to \CO_{Y_K}(h - 2L)$ is surjective, its kernel is a line bundle. More precisely, it is determined by its $c_1$, which is 
\[
H - h - 2L = 2h + L - L_1 - L_2 - L_3
\]
It follows that \eqref{eq:special instantons on demand 1} is surjective if and only if $H^1(\CO_{Y_K}(2h +L - L_1 - L_2 - L_3))$ vanishes.
To conclude note that
\[
H^1(\CO_{Y_K}(2h +L - L_1 - L_2 - L_3)) = H^1(2h - L_1 - L_2 - L_3) = 0
\]

Our last claim is that $L$ is a $2$-jumping line for $E$. Let us prove it.
Note that the Euler characteristic of the derived tensor product of $\CO_{Y_K}(h-2L)$ with a line does not depend on the choice of the line, so that it is always $1$.
In the case of our line $L$ we have
\[
\CO_L \otimes \CO_{Y_K}(h - 2L) = \CO_L(2)
\]
which implies
\[
\Tor_1(\CO_L, \CO_{Y_K}(h-2L)) = \CO_L(1)
\]
It follows that the $\Tor(\CO_L, - )$ exact sequence for the defining sequence \eqref{eq:special instantons on demand -1} of $E(-1)$ begins with an injection
\[
\CO_L(1) \to \restr{E(-1)}{L}
\]
which shows that $L$ is $2$-jumping for $E$.
\end{proof}

\begin{lemma}
For any $L$ there is a smooth linear section $Y_K$ of $Y$ such that $\sigma(a_L) \in \PP(K)$ and that the generic map from $\CU$ to $\CO_{Y_K}(h - 2L)$ is surjective.
\end{lemma}
\begin{proof}
First, note that by lemma \ref{lemma:sections through L} the generic plane $\PP(K)$ in $\PP(V)$ through $\sigma(a_L)$ is transverse to $\sigma(A)$ and contains no trisecants, so that the induced section $Y_K$ is smooth.

The space of  maps from $\CU$ to $\CO_{Y_K}(h - 2L)$ is a $\CC^4$.
This can be shown by noting that by adjunction
\[
\Hom(\CU, \CO_{Y_K}(h - 2L)) \cong \Hom(p_Y^*\CU, \CO_{S_K}(h - 2L))
\] 
and that the relative tautological sequence \eqref{eq:relative tautological p_Y} induces an exact sequence
\[
0 \to H^0(\CO_{S_K}(3h - E - 2L)) \to \Hom(\CU, \CO_{Y_K}(h - 2L)) \to H^0(\CO_{S_K}(2h - 2L)) \to \ldots 
\]
By lemma \ref{singular projection formula}, $Rp_{V*} \CO_{S_K} \cong \CO_{\PP(K)}$, so that the cohomology groups in the above sequence are easily computed via projection formula.

Note that as $Y_K$ is smooth $K \notin \Dtr$, so that by lemma \ref{lemma:p_{Y_K} isomorphism} a map in $\Hom(\CU, \CO_{Y_K}(h - 2L))$ is surjective if and only if the correpsonding one in $ \Hom(p_Y^*\CU, \CO_{S_K}(h - 2L))$ is surjective.
Recall that $S_K$ was defined in definition \ref{def:universal S_K} and is $p_V^{-1} \PP(K)$.
We will show that the space of non-surjective maps is $2$-dimensional inside the $\PP^3$ of maps from $p_Y^*\CU$ to $\CO_{S_K}(h- 2L)$.
Denote by $\phi$ a map from $p_Y^*\CU$ to $\CO_{S_K}(h- 2L)$. 
Consider the composition
\begin{equation}\label{eq:existence of special instantons}
\CO_{\PP_Y(\CU)}(-h) \to p_Y^*\CU \xrightarrow{\phi} \CO_{S_K}(h - 2L)
\end{equation}
If it is zero, then $\phi$ factors via $\CO_{\PP_Y(\CU)}(-2h + E)$.
As $h^0(\CO_{\PP_Y(\CU)}(3h - E - 2L)) = 1$, we can assume that \eqref{eq:existence of special instantons} is non-zero, so that, as $S_K$ is irreducible, it is also injective.
Consider the diagram
\begin{equation*}
\begin{diagram}
\CO_{\PP_Y(\CU)}(-h)		&\rTo			& p_Y^*\CU			&\rTo			& \CO_{\PP_Y(\CU)}(-2h + E) 	\\
\dTo						&			& \dTo^{\phi}			&			& \dTo					\\
\CO_{S_K}(h - 2L)			&\rTo^{1}		&\CO_{S_K}(h-2L)		&\rTo			& 0
\end{diagram}
\end{equation*}
By snake lemma, the cokernel of $\phi$ is isomorphic to the cokernel of
\[
\CO_{\PP_Y(\CU)}(-2h + E) \to \CO_{Z}(h-2L)
\]
where $Z$ is the zero locus of a section of $\CO_{S_K}(2h - 2L)$.
It follows that $Z$ contains a $\PP^1 \subset S_K$ linearly equivalent to $h-L$.
As a consequence, $\phi$ factors via $\CO_{S_K}(h-L)$.
One can finally check that there is a $\PP^1$ of maps from $\CU$ to $\CO_{S_K}(- L)$ and that there is a $\PP^1$ of maps from $\CO_{S_K}(-L)$ to $\CO_{S_K}(h - 2L)$, so that the space of non-surjective maps is the image of $\PP^1 \times \PP^1$, showing eventually that the generic $\phi$ is surjective.
\end{proof}


\subsection{Non-special instantons}\label{sec:Non-special instantons}

The aim of this section is to construct a family of instantons corresponding to an open dense subset of $\CMI_3$. We will also prove that this family is the universal family for non-special instantons. Consistently with notation \ref{not:stratification of B} and definition \ref{def:special instanton definition}, we will denote the moduli space of non-special instantons by $\CMI_3^{\sfn}$.

Note that, by lemma \ref{lemma:Z_3 over Imkappa} and theorem \ref{thm:compactify MI_3}, we already have a description of the moduli space of non-special instantons as an open subset of $\PP_{\Bn}(\CKer^{\sfn*})$.
We want to improve this description in two ways: by finding the image $\beta(\CMI_3^{\sfn})$ inside $\Bn$ and by describing the complement of $\CMI_3^{\sfn}$ inside $\PP_{\Bn}(\CKer^{\sfn*})$.


In the next proposition we will construct non-special instantons as kernels of surjective maps from $\calU^*$ to $\PhiK$. 
\begin{proposition}\label{general instanton}
If $K \notin \Dtr$, then the kernel $E$ of a surjective map from $\calU^*$ to $\PhiK$ is a non-special instanton of charge 3.
\end{proposition}
\begin{proof}
Assume the kernel $E$ of a surjective map from $\calU^*$ to $\PhiK$ is given.
To prove that $E$ is an instanton of charge 3, we need to prove that it is locally free, that it is stable, that $H^1(E(-1)) = 0$ and that $ch(E) = 2 - 3L$. We will denote the map $\calU^* \to \PhiK$ by $\phi$. After that, we will check that it is non-special.

$E$ is locally free if and only if the $\Tor$-dimension of $E$ is  $0$, i.e. if for any $i>0$ and any $y$ closed point $\Tor_i(\calO_{y},E) = 0$ . The defining sequence
\[
0 \to E \to \calU^* \xrightarrow{\phi} \PhiK \to 0
\]
gives an isomorphism between $\Tor_i(\calO_{y},E)$ and $\Tor_{i+1}(\calO_{y},\PhiK)$. 
We compute the latter via resolution \eqref{resolution of Phi(K)}, so that for $i > 1$ we always get $0$, while for $i = 1$ there might be a non-trivial result. 
More precisely, by lemma \ref{lemma:Dtr is common point} the condition $\Tor_2(\calO_{y},\PhiK) \neq 0$ is equivalent to $K \in \Dtr$, so that under the assumption $K \notin \Dtr$ we have that $E$ is locally free.


By Hoppe's criterion \ref{thm:hoppe}, $E$ is stable if and only if $H^0(E) = 0$. To check $H^0(E) = 0$, we check that the map
\begin{equation}\label{general instanton eq1}
H^0(\calU^*) \xrightarrow{H^0(\phi)} H^0\left(\PhiK\right)
\end{equation}
is injective. By the Grothendieck spectral sequence for the composition of two pushforwards, we get that the above map is the same as
\begin{equation}\label{general instanton eq2}
H^0(p_Y^*\calU^*) \xrightarrow{H^0(\phi^{ad})} H^0\left(p_V^*\calO_{\PP(K)}(2)\right)
\end{equation}
induced by the adjoint map $\phi^{ad}$ to $\phi$, so that it is enough to check that $H^0(\phi^{ad})$ is injective.

As $\phi$ is surjective, so is $\phi^{ad}$ by lemma \ref{lemma:surj above-below}. 
Now we claim that the restriction map
\begin{equation} \label{general instanton eq3}
p_Y^*\calU^* \to \restr{p_Y^*\calU^*}{S_K}
\end{equation}
is an isomorphism in cohomology. By Koszul resolution \eqref{Koszul for S} 
this follows if we show that $p_Y^*\calU^*(-2h)$ and $p_Y^* \calU^* (-h)$ are acyclic. By twisting the tautological sequence
\[
0 \to \calO_{\PP_{Y}(\CU)}(2h-E) \to p_Y^*\calU^* \to \calO_{\PP_{Y}(\CU)}(h) \to 0
\]
by $\calO_{\PP_{Y}(\CU)}(-2h)$ we get immediately that $p_Y^*\calU^*(-2h)$ is acyclic. By twisting it by $\calO_{\PP_{Y}(\CU)}(-h)$ we get a non trivial extension of $\calO_{\PP_{Y}(\CU)}$ by $\calO_{\PP_{Y}(\CU)}(h-E)$, so that $p_Y^*\calU^*(-h)$ is acyclic if and only if $h^1\left(\calO_{\PP_{Y}(\CU)}(h-E)\right) = 1$ and its other cohomology groups vanish. The short exact sequence
\[
0 \to \calO_{\PP_{Y}(\CU)}(h-E) \to \calO_{\PP_{Y}(\CU)}(h) \to \calO_E(h) \to 0
\]
is $SL_2$-equivariant and induces an $SL_2$-equivariant cohomology long exact sequence
\begin{equation}\label{general instanton eq4}
0 \to H^0\left(\calO_{\PP_{Y}(\CU)}(h-E)\right) \to S^4W \to S^4W \oplus \CC \to H^1\left(\calO_{\PP_{Y}(\CU)}(h-E)\right) \to 0
\end{equation}
so that $h^1\left(\calO_{\PP_{Y}(\CU)}(h-E)\right) = 1$. 
Summing up, we have proved that \eqref{general instanton eq3} is an isomorphism in cohomology.

In order to prove stability of $E$, we still need to prove that if
\[
\restr{p_Y^*\calU^*}{S_K} \xrightarrow{\restr{\phi^{ad}}{S_K}} p_V^*\calO_{\PP(K)}(2h)
\]
is surjective, then $H^0(\restr{\phi^{ad}}{S_K})$ is injective, that is to say that its kernel has no global sections. 
As we have just proved that $\restr{\phi^{ad}}{S_K}$ surjective, its kernel is locally free of rank one, and is determined by its $c_1$. 
As $c_1(\calU^*) = H = 3h - E$, our claim is now that $\calO_{S_K}(h-E)$ has no global sections. 
This last fact can be reduced  via resolution \eqref{Koszul for S} to some vanishing in the cohomology of 
\[
\CO_{\PP_{Y}(\CU)}(-h-E), \quad \CO_{\PP_{Y}(\CU)}(-E), \quad \CO_{\PP_{Y}(\CU)}(h-E)
\]
It is easy to check that the first two line bundles are acyclic, and that $h^0(\CO_{\PP_{Y}(\CU)}(h-E))$, so that finally also $\CO_{S_K}(h-E)$ has no global sections.

The vanishing of $H^1(E(-1))$ follows from 
\[
H^1(\calU) = 0, \quad H^0(\PhiK(-H)) = 0
\] 
While the vanishing of $H^1(\calU)$ is trivial, we check that 
\[
H^0(\PhiK(-H)) = 0
\]
Recall that $\PhiK$ has a resolution \eqref{transform of P(K)(2)}.
Twist it by $\CO_Y(-H)$ and note that the only nontrivial vanishing is that of $H^\bullet(S^2\CU^*(-H))$.
On the other hand, the vanishing of the cohomology of $S^2\CU^*(-H)$ is equivalent to the fact that $\CU$ is exceptional, which is in proposition \ref{prop:full exceptional}. 
It follows that $E(-1)$ is acyclic.

The Chern character of $E$ is $\ch(\calU^*) - \ch(\PhiK)$. 
To check that it is $2 - 3L$ it is enough to substitute the Chern characters which are computed in table \eqref{Cherntable} and lemma \ref{lemma:FM transform Tor}.

Finally, the fact that $E$ is non-special directly follows from proposition \ref{beta3} and criterion \ref{beta special}, which tell us that $\beta(E) = K$ and that $E$ is special if and only if $\beta(E) \in \Bs$. 
By assumption $\beta(E) \notin \Dtr$, while by lemma \ref{lemma:Imkappa in Dtr} $\Bs \subset \Dtr$, so that we have that $E$ is not special.
\end{proof}


\begin{lemma}\label{lemma:different 3-instantons}
For $i = 1,2$, let $E_i$ be the instanton associated (via proposition \ref{general instanton}) to $\phi_i: \calU^* \to \Phi_{\Bl}(K_i)$, $K_i \notin \Dtr$. If $E_1 \cong E_2$, then $K_1 = K_2$ and $\phi_1 = \phi_2$ up to rescaling.
\end{lemma}
\begin{proof}
As $E_1 \cong E_2$, we have $K_1 = \beta(E_1) = \beta(E_2) = K_2$.
As for $\phi_1$ and $\phi_2$, there is a canonical arrow from $E_1 \cong E_2$ to $\calU^*$, so that (up to rescaling $f$), we can find a diagram
\begin{diagram}
E_1		& \rTo	& \calU^*	& \rTo^{\phi_1}	& \Phi_{\Bl}(K_1) 	\\
\dTo^{f}	&		& \dTo^{1}	&			& \dTo			\\
E_2		& \rTo	& \calU^* 	& \rTo^{\phi_2}	& \Phi_{\Bl}(K_2)
\end{diagram}
showing that $\phi_1$ and $\phi_2$ have the same kernel.

Moreover, by lemma \ref{lemma:p_{Y_K} isomorphism} and as $K_i \notin \Dtr$, the sheaf $\Phi_{\Bl}(K_i)$ is a line bundle on its support.
As its support is irreducible, the identity is its only automorphism up to rescaling.
\end{proof}

We should also check that our construction produces at least one instanton. We will check it in lemma \ref{generic map to Phi(K)}.

The next lemma, together with \ref{generic map to Phi(K)}, will give us the dimension of the family of instantons which we have just constructed.

%

\begin{lemma}\label{maps to Phi(K)}
For any $K$, there is an exact sequence 
\begin{equation}\label{eq:maps to PhiK}
0 \to K^* \to \Hom(\calU^*, \PhiK) \to \CC \to 0
\end{equation}
\end{lemma}
\begin{proof}
Let $\Phi^*$ be the left adjoint functor to $\Phi_{\Bl}$.
Then by definition there is a natural isomorphism
\[
\Hom_{Y}(\calU^*, \PhiK) \cong \Hom_{\PP(V)}(\Phi^*( \calU^*), \calO_{\PP(K)}(2))
\]
As $\Phi_{\Bl} = Rp_{Y*} \circ Lp_V^*$, its left adjoint $\Phi^*$ is $Rp_{V*}(- \otimes \omega_{p_V}) \circ Lp_Y^*$. Substituting $\omega_{p_V} = \calO(E)$ and using the tautological sequence \eqref{eq:relative tautological p_Y} we get a distinguished triangle
\begin{equation}\label{maps to Phi(K) inner 1}
Rp_{V*} \CO_{\PP_{Y}(\CU)}(2h) \to \Phi^*(\CU^*) \to Rp_{V*}\CO(h+E)
\end{equation}
The leftmost term is $\CO_{\PP(V)}(2h)$ by projection formula. The rightmost is computed by pushing forward the exact sequence
\begin{equation}
0 \to \CO(h) \to \CO(h+E) \to \CO_E(h+E) \to 0
\end{equation}
to $\PP(V)$. 
Note that for every point $v$ in $\Imsigma$, the sheaf $\CO_E(h+E)$ is acyclic on the fiber $p_V^{-1}(v)$, as it is isomorphic to $\CO_{p_V^{-1}(v)}(-1)$, so that
\[
Rp_{V*}\CO(h+E) \cong Rp_{V*}\CO(h)
\]
In the end, by taking the long exact sequence for \eqref{maps to Phi(K) inner 1}  we have shown that there is a canonical exact sequence on $\PP(V)$
\begin{equation}
0 \to \CO_{\PP(V)}(2h) \to \Phi^*(\CU^*) \to \CO_{\PP(V)}(h) \to 0
\end{equation}
which after the application of $\Hom( - , \CO_{\PP(K)}(2h))$ becomes sequence \eqref{eq:maps to PhiK}.
\end{proof}

Now we want to identify surjective maps inside $\Hom(\CU^*, \PhiK)$.
The next lemma allows us to do it without discussing the possible singularities of $S_K$ and $Y_K$.
Note that for any $K$ there is an isomorphism
\begin{equation}\label{eq:surj above-below -1}
\Hom_Y\left( \CU^*, \PhiK \right) \xrightarrow{\cong} \Hom_{S_K}\left( p_{Y_K}^*\CU^*, \CO_{S_K}(2h) \right)
\end{equation}
given by adjunction with respect to pullback and pushforward via the composition
\begin{equation}\label{eq:surj above-below -2}
S_K \xrightarrow{p_{Y_K}} Y_K \to Y
\end{equation}

\begin{lemma}\label{lemma:surj above-below}
Let $K \notin \Dtr$.
A map in $\Hom_Y\left( \CU^*, \PhiK \right)$ is surjective the corresponding one in $\Hom_{S_K}\left( p_{Y_K}^*\CU^*, \CO_{S_K}(2h) \right)$ via \eqref{eq:surj above-below -1} is surjective.
\end{lemma}
\begin{proof}
By lemma \ref{lemma:p_{Y_K} isomorphism}, if $K \notin \Dtr$ then the composition \eqref{eq:surj above-below -2} is a closed embedding.
Adjunction with respect to pullback and pushforward via closed embeddings clearly preserves surjectivity.
\end{proof}

\begin{lemma}\label{generic map to Phi(K)}
If $K \notin \Dtr$, then the general map in $\Hom_Y(\calU^*, \PhiK)$ is surjective. The locus of non-surjective maps is a union of at most $5$ hyperplanes. 
\end{lemma}
\begin{proof} 
We will classify non-surjective maps and show that they form a divisor inside $\PP\left(\Hom_Y(\CU,\PhiK)\right)$. By lemma \ref{lemma:surj above-below}, a map in $\PP\left(\Hom_Y(\CU,\PhiK \right)$ is surjective if and only if the corresponding one in $\Hom_{\PP_{Y}(\CU)}(p_Y^*\calU^*, p_V^*\CO_{\PP(K)}(2h))$ is surjective. 

Take any map $f \in \Hom_{\PP_{Y}(\CU)}(p_Y^*\calU^*, p_V^*\CO_{\PP(K)}(2h))$: the composition with the tautological map $\CO_{\PP_{Y}(\CU)}(2h-E) \to \CU^*$ is either 0 or not. If it vanishes, then $f$ factors via $\CO_{\PP_{Y}(\CU)}(h)$. 
In this case $f$ vanishes on a linear section of $S_K$, so that it is never surjective. 

To check how many such maps there are, we compute $\hom(\CO_{S_K}(h), \CO_{S_K}(2h)) = h^0(\CO_{S_K}(h))$. 
By lemma \ref{singular projection formula}, the blow-up projection $p_K: S_K \to \PP(K)$ is such that the coevaluation map \eqref{eq:coev_K} is an isomorphism.
Therefore, by projection formula for $p_K$, $H^\bullet(\CO_{S_K}(h)) \cong H^\bullet(\CO_{\PP(K)}(h)) = K^*$, which means that maps factoring via $\CO(h)$ form a hyperplane.

By lemma \ref{lemma:Imkappa is unexpected dimension}, whenever $K \notin \Bs$ the fiber product $S_K$ is irreducible. 
It follows that if the composition 
\[
\CO_{\PP_{Y}(\CU)}(2h-E) \to \CU^* \xrightarrow{f} p_V^* \CO_{\PP(K)}(2h)
\] 
does not vanish, then it is injective.
As a consequence, on $S_K$ there is a commutative diagram
\[
\begin{diagram}
\CO_{\PP_Y(\CU)}(2h-E)	&\rTo			& p_Y^*\CU^*			&\rTo			& \CO_{\PP_Y(\CU)}(h) 	\\
\dTo					&			& \dTo^{f}				&			& \dTo				\\
p_V^*\CO_{\PP(K)}(2h)	&\rTo^{1}		&p_V^*\CO_{\PP(K)}(2h)	&\rTo			& 0
\end{diagram}
\]
By the snake lemma, there is an exact sequence
\[
\CO(h) \to  \CO_E(2h) \to \Coker f \to 0
\]
As $h$ restricts trivially to $E$, we have
\[
\CO_E \to \CO_E \to \Coker f \to 0
\]
$E$ might be reducible or non-reduced, but in any case, if $\Coker f \neq 0$ then there is a $\PP^1$ denoted by $e$ such that $\Coker f \to \CO_{e}$. 

We want to show that for each such $e$ there is a $\PP^2$ of maps from $p_Y^*\CU^*$ to $p_V^*\CO_{\PP(K)}(2h)$ which are not surjective on $e_i$, and thus factor via the twisted ideal $I_{e}(2h)$ of $e$ in $S_K$.
Note that $e_i$ might not be a Cartier divisor inside $S_K$, so that we will write $I_{e}$ instead of $\CO_{S_K}(-e)$ for the ideal of $e$. 
To check $\hom(p_Y^*\CU^*,I_{e}(2h)) = 3$, use the long exact sequence induced by
\[
0 \to I_e(2h) \to \CO_{S_K}(2h) \to \CO_e \to 0
\]

We can easily check that $\hom(p_Y^*\CU^*,\CO_e) = 1$ via the tautological sequence on $S_K$ and using the fact that $E$ restricts to $e$ as -1.
Then we need to check that $\ext^1(p_Y^*\CU^*, I_e(2h)) = 0$. 
Again by tautological sequence it is enough to show that $H^1(\CO_e(2h-E)) = H^1(\CO_e(h)) = 0$, which is true as $h$ restricts trivially to $e$, while $E$ restricts as $\CO_e(-1)$.

Finally, we have proven that the set of non-surjective maps is a union of at most $5$ hyperplanes, such that each one represents maps that are non surjective on a divisor. 
The divisor where surjectivity fails can be either linearly equivalent to $h$ or to one of the (at most 4) irreducible components $e_i$ of $E$ on $S_K$.
\end{proof}
%
%



On the other hand, there are no surjective maps from $\CU^*$ to $\PhiK$ for $K \in \Dtr$, as proved in the following lemma.

\begin{lemma}\label{lemma:no surj on Dtr}
For $K \in \Dtr$ there are no surjective maps from $\CU^*$ to $\PhiK$
\end{lemma}
\begin{proof}
We will show that if $\PP(K)$ contains a trisecant $L$ to $\Imsigma$, then the fiber of $\PhiK$ at the point $y \in Y$ correpsonding to the trisecant via \ref{cor:points and trisecants} is $3$-dimensional.

By base change and as $p_Y$ has relative dimension $1$, the dimension of the restriction of $\PhiK$ to $y$ is greater than the dimension of
\[
H^0\left(p_Y^{-1}(y),\restr{ p_V^*\CO_{\PP(K)}(2h)}{p_Y^{-1}(y)} \right)
 \]
which is isomorphic to $H^0(\CO_L(2))$.
As the rank of $\CU^*$ is $2$, it follows that there are no surjective maps from $\CU^*$ to $\PhiK$.
\end{proof}


By performing the construction in \ref{general instanton} in families, we are now going to construct a 9-dimensional family of instantons and to prove that it is the universal family for the moduli space of non special istantons. 

First, recall from definition \ref{def:universal S_K} that over $\Bg$, one can construct a universal family $\CS$ for $S_K$.

By lemma \ref{maps to Phi(K)} the pushforward of $\CHom_{\Bg \times Y}(\CU^*, Rp_{Y*}\CO_{\CS}(2h))$ to $\Bg$ is a rank 4 vector bundle. Denote it by $\CG$.
It turns out that we have already discussed this vector bundle (or at least its restriction to the complement of $\Dtr$): it is the bundle 

\begin{lemma}
There is an isomorphism of vector bundles
\[
\CKer^{\Bg \setminus \Dtr} \cong \restr{\CG^*}{\Bg \setminus \Dtr}
\]
up to a twist by a line bundle on ${\Bg \setminus \Dtr}$.
\end{lemma}
\begin{proof}
First, note that by Grothendieck duality for the embedding of $Y_K$ in $Y$ there is an isomorphism
\[
\Hom(\CU^*, \CO_{Y_K}(2h)) \cong \Ext^1(\CO_{Y_K}(h-E), \CU)
\]
which can be constructed in family, up to a twist by a line bundle on the base of the family.
Recall also that by lemma \ref{lemma:p_{Y_K} isomorphism}, the map $p_{Y_K}$ is an isomorphism between $S_K$ and $Y_K$ whenever $K \notin \Dtr$.

Next, compute the decomposition of $\CO_{Y_K}(h-E)$ with respect to the collection $\CU, \CU^\perp, \CO_Y, \CO_Y(H)$.
As $H = 3h - E$, $\CO_{Y_K}(h-E)$ lies in the right orthogonal to $\CO_Y(H)$.
Also the component with respect to $\CO_Y$ is easily computed:
\[
\RHom(\CO_Y, \CO_{Y_K}(h-E)) = \CC[-1]
\]
so that there is a non trivial extension
\begin{equation}\label{eq:late}
0 \to \CO_{Y_K}(h-E) \to \LL_{\CO_Y} \CO_{Y_K}(h-E) \to \CO_Y \to 0 
\end{equation}

To compute the $\CU^\perp$ component of $\LL_{\CO_Y} \CO_{Y_K}(h-E)$, use the defining sequence \eqref{eq:late}:
\[
0 \to V/\CU \otimes \CO_{Y_K}(h-E) \to V/\CU \otimes \LL_{\CO_Y} \CO_{Y_K}(h-E) \to V/\CU \to 0
\]
Note that the connection morphism
\[
V \cong H^0(V/\CU) \to H^1(V/\CU \otimes \CO_{Y_K}(h-E))
\]
is isomorphic to the second map in the exact sequence
\[
0 \to H^1(\CU \otimes \CO_{Y_K}(h-E)) \to V \otimes H^1(\CO_{Y_K}(h-E)) \to H^1(V/ \CU \otimes \CO_{Y_K}(h-E)) \to 0
\]
induced by the tautological sequence on $Y$.
Let us now prove that the above sequence is isomorphic to
\[
0 \to K \to V \to V/K \to 0
\]
For each component $e_i$ of the exceptional locus, get the commutative diagram
\[
\begin{diagram}
H^0(\CO_{e_i})		&		&	\\
\dInto			& \rdInto	&	\\
H^1(\CO_{Y_K}(-E))	& \rTo	& V \otimes H^1(\CO_{Y_K}(h-E))
\end{diagram}
\]
The diagonal map is the natural injection of a component of the intersection $\PP(K) \cap \Imsigma$ into $V$, so that as $i$ varies among all components, the images span $K \subset V$.
As $h^1(\CO_{Y_K}(-E)) = 3$, we have proved that the mutation of $\CO_{Y_K}(h-E)$ across $\CU^\perp, \CO_Y, \CO_Y(H)$ is the cone of
\[
\CO_{Y_K}(h-E) \to \{ K \otimes \CU^\perp \to \CO_Y\}
\]
As this last cone is isomorphic to a direct sum of shifts of $\CU$, the computation of the rank yields the following decomposition for $\CO_{Y_K}(h-E)$
\begin{equation} \label{eq:late 1}
\{	\Ker(K\otimes A \to V^*) \otimes \CU \to K \otimes \CU^\perp \to \CO_Y	\}
\end{equation}

Finally, we can use decomposition \eqref{eq:late 1} to show that
\[
\Ker(K \otimes A \to V^*) \cong \Ext^1(\CO_{Y_K}(h-E), \CU)
\]
This is clear as the $\CU^\perp$ and $\CO$ part of \eqref{eq:late 1} do not contribute to $\Hom( - , \CU)$.

The same isomorphisms, performed in families over $\Bg \setminus \Dtr$, show that 
\[
\CKer^{\Bg \setminus \Dtr} \cong \restr{\CG^*}{\Bg \setminus \Dtr}
\]
up to a twist by a line bundle on $\Bg \setminus \Dtr$.
\end{proof}


Our claim is that the moduli space of non-special instantons of charge $3$ is the open subset of $\PP_{\Bg}(\CG)$ parametrizing surjective maps from $\CU^*$ to $\CO_{S_K}$. We will denote it by $\PP(\CG)^\circ$.
By lemma \ref{lemma:no surj on Dtr} there are no such maps if $K \in \Dtr$, so that the image of $\PP(\CG)^\circ$ in $\Bg$ does not intersect $\Dtr$. 
By lemma \ref{generic map to Phi(K)} and lemma \ref{lemma:no surj on Dtr}, the image of $\PP(\CG)^\circ$ in $\Bg$ is exactly the complement of $\Dtr$.

\begin{theorem}\label{thm:general instanton}
The moduli space $\CMI_3^{\sfn}$ of non-special instantons of charge $3$ is isomorphic to $\PP(\CG)^\circ$.
\end{theorem}
\begin{proof}
First we prove that $\PP(\CG)^\circ$ has a natural identification with an open subscheme of $\CMI_3^{\sfn}$.
There is a natural map from $\PP(\CG)^\circ$ to $\CMI_3^{\sfn}$ induced by the relative version of lemma \ref{general instanton}. 

The diagram
\begin{equation}\label{diag:general instanton}
\begin{diagram}
\PP(\CG)^\circ	& \rTo	& \CMI_3^{\sfn}		\\
			& \rdTo	& \dTo_{\beta}		\\
			&		& \Bg
\end{diagram}
\end{equation}
is commutative by proposition \ref{beta3}.
By lemma \ref{lemma:different 3-instantons} the map is injective on closed points.
Moreover the two varieties have the same dimension and are smooth, as explained in remark \ref{rmk:geometric quotient} about the GIT construction.
It follows from Zariski's main theorem that the map is an open embedding.

Now we prove that the map from $\PP(\CG)^\circ$ to $\CMI_3^{\sfn}$ is surjective by proving that every non-special instanton is the kernel of a surjective map from $\CU^*$ to $\Phi(K)$ for some $K$.

First, note that by proposition \ref{beta special} the diagram \eqref{diag:general instanton} factors as
\begin{equation}\label{diag:general instanton 1}
\begin{diagram}
\PP(\CG)^\circ	& \rTo	& \CMI_3^{\sfn}		\\
			& \rdTo	& \dTo_{\beta}		\\
			&		& \Bn
\end{diagram}
\end{equation}
Over $Y \times (\Bn )$ we can consider the sheaf
\begin{equation}\label{eq:general instanton}
\Phi_{\Bl} \left(\CO_{\PP(\CK)}(2h) \right)
\end{equation}
and think of it as of a family of Gieseker semistable pure sheaves.
The fact that the sheaves $\PhiK$ are pure follows from lemma \ref{lemma:FM transform Tor}. Their stability follows from the fact that they have rank $1$ on their support (which is irreducible, by \ref{lemma:Imkappa is unexpected dimension}), so that they have no non-trivial saturated subsheaves.
It follows that \eqref{eq:general instanton} defines a regular map 
\[
\Bn \to \CM_{\ch(\PhiK)}
\]
into the moduli space of Giesecker semistable pure sheaves with the same Chern character as $\PhiK$.

One can also define a family of such sheaves over $\CMI_3^{\sfn}$ by taking the cokernel of the coevaluation map
\[
0 \to \CE \to \CU^*
\]
Note that this family of cokernels restricts to $\PP(\CG)^\circ$ as 
\[
0 \to \restr{\CE}{\PP(\CG)^\circ} \to \CU^* \to \Phi_{\Bl}(\CO_{\PP(\CK)}(2h)) \to 0
\]
where the surjective map is the universal one on $\PP(\CG)^\circ$ from $\CU^*$ to $\PhiK$.

What we have proved so far is that there is a diagram
\[
\begin{diagram}
\PP(\CG)^\circ				& \rInto		& \CMI_3^{\sfn}		\\
\dTo						& \ldTo^{\beta}	& \dTo_{\mathsf{ck} = \coker(\CE \to \CU^*)}	\\
\Bn						& \rTo_{\delta}	& \CM_{\ch(\PhiK)}
\end{diagram}
\]
and that the square in it is commutative.
As $\PP(\CG)^\circ$ is an open Zariski dense subvariety of $\CMI_3^{\sfn}$ and as
the moduli space $\CM_{\ch(\PhiK)}$ is separated, the two maps $\delta \circ \beta$ and $\mathsf{ck}$ coincide. 
In particular, for any non-special instanton $E$ of charge $3$ we have an exact sequence
\[
0 \to E \to \CU^* \to \Phi_{\Bl}\CO_{\beta(E)}(2h) \to 0
\]
that is to say that $E$ is in the image of $\PP(\CG)^\circ$.
\end{proof}

As a corollary, we get the following relation between jumping lines for $E$ and bisecants to $\Imsigma$ contained in $\PP(\beta(E))$.
\begin{corollary}\label{cor:bisecant jump}
Given a bisecant $L$ to $\Imsigma$ and a non-special instanton $E$ such that $L \subset \beta(E)$, the image $L^{+} = p_Y(\WL)$ of the strict transform $\WL$ of $L$ in $\PP_Y(\CU)$ is a $1$-jumping line for $E$.
\end{corollary}
\begin{proof}
By theorem \ref{thm:general instanton} for any non-special instanton there is an exact sequence
\[
0 \to E \to \CU^* \to \CO_{Y_{\beta(E)}}(2h) \to 0
\]
By twisting it by $\CO_Y(-H)$ and restricting it to $L^{+}$ we get a long exact sequence
\begin{equation}\label{eq:no 2-jump}
0 \to \Tor_1(\CO_{Y_{\beta(E)}}(2h - H), \CO_{L^+}) \to \restr{E}{L^+}(-1) \to \CO_{L^+} \oplus \CO_{L^+}(-1) \to \CO_{L^+ \cap Y_{\beta(E)}}(2h-H) \to 0
\end{equation}
As $L^+$ lies in $Y_{\beta(E)}$ and $h$ restricts to $L$ as $\CO_L(1)$, sequence \eqref{eq:no 2-jump} becomes
\begin{equation}\label{eq:no 2-jump 2}
0 \to \CO_{L^+} \to \restr{E}{L^+}(-1) \to \CO_{L^+} \oplus \CO_{L^+}(-1) \to \CO_{L^+}(1) \to 0
\end{equation}
which shows that $L^+$ is $1$-jumping for $E$
\end{proof}

\begin{remark}
For $K \notin \Dtr \cup \Dfat$ there are $6$ bisecants to $\Imsigma$ contained in $\PP(K)$, corresponding to lines in $\PP(K)$ via the $4$ points $\PP(K) \cap \Imsigma$.
If $K \in \Dfat \setminus \Dtr$ there will be in general less and less bisecants to $\Imsigma$, as long as $K \notin \tau(\PP(A))$.

When $K \in \tau(\PP(A))$, i.e. when $\PP(K)$ is the tangent plane to $\Imsigma$ at some point $a_L \in \PP(A)$, there are infinitely many bisecants to $\Imsigma$ contained in $\PP(K)$, namely all lines in $\PP(K)$ through $a_L$.
It follows that for instantons $E$ such that $\beta(E) \in \tau(\PP(A)) \setminus \Dtr$ there is a $\PP^1 \subset \PP(A)$ of jumping lines which is the $\PP^1$ of lines intersecting $L$.

In other words, corollary \ref{cor:bisecant jump} says that the family of instantons with a fixed $\beta(E)$ induces a linear system of cubic plane curves (the supports of the associated theta characteristics) whose base locus corresponds in the dual $\PP(A)$ to the intersection of $\PP(K)$ and $\Imsigma$.
\end{remark}


\section{An example: degenerate theta-characteristics} \label{sec:An example}
In this section we will discuss the possible singularities of the theta-characteristic associated with a charge $3$ instanton.
We will prove that the instantons whose associated theta-characteristic is not locally free are the special instantons, and we will provide a geometric description of a family of instantons whose associated theta-characteristic is supported on a reducible curve.

Classically, a theta-characteristic on a smooth curve is a line bundle such that
\begin{equation} \label{eq:classical theta}
\CM \otimes \CM \cong \omega_C
\end{equation}
Unfortunately, condition \eqref{eq:classical theta} forces $\CM$ to be a line bundle, while a family of smooth curves with theta-characteristic can degenerate flatly to a singular curve with a torsion sheaf on it.
On a Gorenstein curve one can still make sense of the condition
\begin{equation} \label{eq:generalized theta}
\CM \cong \CRHom(\CM,\omega_C)
\end{equation}
which in the smooth case is equivalent to \eqref{eq:classical theta}.
\begin{definition}
A generalized theta-characteristic on a Gorenstein curve $C$ is a sheaf $\CM$ on $C$ such that condition \eqref{eq:generalized theta} holds.
\end{definition}
Generalized theta-characteristics can be not locally free and can have rank greater than one (as condition \eqref{eq:generalized theta} is for example invariant under direct sums).
Given an instanton $E$ and provided that it is trivial on the generic line, theorem \ref{thm:jumping sheaf} associates a theta-characteristic $\theta_E$ with $E$.
Moreover, all theta-characteristics $\theta_E$ arising from instantons satisfy $H^0(\theta_E) = 0$: such generalized theta-characteristics are called non-degenerate.

There are several open questions about how degenerate can $\theta_E$ be for an instanton $E$: can the support of $\theta_E$ be singular? Can it be reducible? Can the sheaf $\theta_E$ be not locally free?
It turns out the answer to all the above questions is positive, as it is proved in the next propositions.

\begin{proposition}\label{prop:non locally free theta}
Let $E$ be an instanton of charge $3$, then $\theta_E$ is not locally free if and only if $E$ is special.
\end{proposition}
\begin{proof}
By corollary \ref{prop:2-jump again} an instanton of charge $3$ is special if and only if it has a line $L$ where its order of jump is greater than $1$.
By proposition \ref{cor:jumping order} this is equivalent to $\theta_E$ having rank $2$ at some point of its support, i.e. to the fact that $\theta_E$ is not locally free on its schematic support.
\end{proof}

\begin{proposition}\label{prop:reducible theta}
Let $E$ be a non-special instanton such that $\beta(E) \subset \PP(V)$ is a tangent plane to $\Imsigma$. 
Then the support of $\theta_E$ is reducible.
\end{proposition}
\begin{proof}
First, note that by \ref{thm:gamma injective} the support of $\theta_E$ is the intersection of $\PP(A)$ and of the symmetric discriminant cubic fourfold $\Delta_H$ inside $\PP^(S^2H)$.
More precisely, the sheaf $\theta_E$ is the pullback of the sheaf of cokernels of a symmetric form in $H$, which is a sheaf of rank $1$ over the generic point of $\Delta_H$ and of rank $2$ over a surface.
If $\PP(A) \subset \Delta_H$, then it intersects the locus of corank $2$ symmetric forms, so that the rank of $\theta_E$ is $2$ at least at one point of $\PP(A)$.
As the instanton is non-special, this is impossible by proposition \ref{prop:2-jump again}, so that the support of $\theta_E$ is not the whole $\PP(A)$.

By corollary \ref{cor:bisecant jump}, all tangent lines to $\Imsigma$ which are contained in $\beta(E)$ are jumping lines for $E$.
As $\beta(E)$ is a tangent plane to $\Imsigma$ at some point $a$, the $\PP^1$ of lines in $\beta(E)$ through $a$ corresponds to a $\PP^1$ of lines in $Y$ all of them intersecting $L_a$.
By proposition \ref{prop:intersecting lines}, the family of lines intersecting $L_a$ is a line in $\PP(A)$, so that we have shown that the support of $\theta_E$ contains a line.
\end{proof}

\begin{remark}
One can construct explicitly a theta-characteristic coming from an instanton and supported on a reducible curve.
This can be done by taking coordinates $x,y,z$ on $\PP(A)$ of weight $-2,2,0$ with respect to the action of $\SSL_2$.
Then the object
\begin{equation}\label{eq:reducible theta 3}
\left\{ \CO_{\PP(A)}(-2)^{\oplus 3} \xrightarrow{\Theta}  \CO_{\PP(A)}(-1)^{\oplus 3} \right\}
\end{equation}
where $\Theta$ is given by the matrix
\[
\Theta = 
\begin{pmatrix}
 0 	& x 	& y \\
 -x 	& y 	& z \\
 -y 	& -z 	& 0
\end{pmatrix}
\]
The object \eqref{eq:reducible theta 3} is $\CC^*$-equivariant, and this fact allows to prove that the conditions of theorem \ref{maininstantontheorem} hold for it.
\end{remark}

\section{A desingularization of $\CMIbar_3$} \label{sec:A desingularization}
The compactification of $\CMI_3$ which we have constructed in section \ref{sec:Instantons of charge 3} is singular on a subvariety of codimension $3$. 
We will show it by constructing a desingularization for it, which we will denote by $\CMItil_3$, and a small contraction from $\CMItil_3$ to $\CMIbar_3$.

There is a commutative diagram
\begin{equation}\label{diag:flag}
\begin{diagram}
\Flag_{\Bg}(3,4, \CK \otimes A)	& \rTo^{\Wpi_3}		& \Gr_{\Bg}(4, \CK \otimes A)	\\
\dTo^{\Wpi_4}					&				& \dTo_{\pi_4}					\\
\Gr_{\Bg}(3, \CK \otimes A)		& \rTo_{\pi_3} 		& \Bg
\end{diagram}
\end{equation}
On the relative Grassmannians there are relative tautological injections
\begin{equation}
\CRelTaut_3 \to \pi_3^*\CK \otimes A
\end{equation}
and
\begin{equation}
\CRelTaut_4 \to \pi_4^*\CK \otimes A
\end{equation}
The bundle $\CK \otimes A$ is a subbundle of $V \otimes A \otimes \CO_{\Bg}$, so that the natural evaluation of $A$ on $V$ induces a map
\begin{equation}\label{eq:bundle Ker}
\CK \otimes A \to V^* \otimes \CO_{\Bg}
\end{equation}
extending to an exact sequence
\[
0 \to \CKer \to \CK \otimes A \to V^* \otimes \CO_{\Bg} \to \kappa_{*}\CO_{\PP(A)}(2) \to 0
\]
which is discussed in proposition \ref{prop:embed P(A)}.

\begin{lemma}\label{lemma:from 9 to 8}
There is a unique $\SSL_2$-equivariant line subbundle
\begin{equation}\label{eq:from 9 to 8}
\CO_{\Bg}(-1) \to A \otimes \CK
\end{equation}
and it is factors via $\CKer \to A \otimes \CK$.
\end{lemma}
\begin{proof}
Note that $\CK(1)$ is isomorphic to $\Lambda^2 \CK^*$.
The space of sections of $\Lambda^2 \CK^*$ is $\Lambda^2V^*$ and it contains a unique $\SSL_2$-invariant subspace of dimension $3$.
It follows that there is a unique $\SSL_2$-invariant choice of the map \eqref{eq:from 9 to 8}.

In order to check that it factors via 
\[
\CKer \to A \otimes \CK
\]
we check that the composition
\[
\CO_{\Bg}(-1) \to A \otimes \CK \to V^* \otimes \CO_{\Bg}
\]
vanishes.
This last fact holds because there are no non-zero $\SSL_2$-invariant elements in $V^* \otimes \Lambda^2 V$.
\end{proof}

Denote by $\CF$ the rank $8$ vector bundle which is the quotient \eqref{eq:from 9 to 8}.
There is a natural embedding of relative Grassmannians
\[
\Gr_{\Bg}(3, \CF) \to \Gr_{\Bg}(4, A \otimes \CK)
\]
and a tautological embedding
\[
\CRelTaut_{\CF} \to \pi_{\CF}^*\CF
\]
where $\CRelTaut_{\CF}$ is the rank $3$ tautological vector bundle and $\pi_{\CF}$ is the canonical projection to $\Bg$.

By lemma \ref{lemma:from 9 to 8} the map \eqref{eq:bundle Ker} factors via the natural projection from $A \otimes \CK$ to $\CF$.
Denote by $\sigma_{\CF}$ the composition
\begin{equation}\label{eq:Z_CF -1}
\sigma_{\CF}: \CRelTaut_{\CF} \to \pi_{\CF}^*\CF \to V^* \otimes \CO_{\Bg}
\end{equation}
and its zero locus by $Z_{\CF}$.
Analogously there are maps
\begin{gather*}
\sigma_3: \CRelTaut_3 \to V^* \otimes \CO \\
\sigma_4: \CRelTaut_4 \to V^* \otimes \CO
\end{gather*}
over $\Gr_{\Bg}(3, \CK \otimes A)$ and $\Gr_{\Bg}(4, \CK \otimes A)$, whose zero loci will be denoted by $Z_3$ and $Z_4$.
By its definition, $\CMIbar_3$ is the zero locus $Z_3$ of $\sigma_3$.
Moreover, we will denote by $\WZ_{\CF}, \WZ_4$ the pullbacks of $Z_{\CF}, Z_4$ via $\Wpi_3$.

\begin{lemma}\label{lemma:Z_4}
The section $\sigma_{\CF}$ is regular.
Its zero locus $Z_{\CF}$ is irreducible and is isomorphic to $\Bl_{\Bs}\Bg$.
The intersection of $\pi_4^{-1}(\Bs)$ and $Z_{\CF}$ in $\Gr_{\Bg}(4,A \otimes \CK)$ embeds into $\pi_4^{-1}(\Bs)$ as
\[
\Gr_{\Bs}\left( 3, \CKer^{\sfs}/\CO_{\Bs}(-3) \right) \subset \Gr_{\Bs}\left(4, A \otimes \CK \right)
\]
and into $Z_{\CF}$ as
\[
E \subset \Bl_{\Bs}\Bg
\]
where $E$ is the exceptional divisor of the blow up.
\end{lemma}
\begin{proof}
First, the zero locus $Z_{\CF}$ has a stratification into $Z_{\CF}^{\sfs}$ and $Z_{\CF}^{\sfn}$ induced by pulling back the stratification $\Bs, \Bn$.
Note that $Z_{\CF}^{\sfn}$ is isomorphic to $\Bn$, as the kernel of
\begin{equation}\label{eq:Z_CF 1}
\pi_{\CF}^*\CF \to V^* \otimes \CO_{\Bg}
\end{equation}
is locally free of rank $3$ over $\Bn$.
On the other hand, for the same reason, $Z_{\CF}^{\sfn}$ is a $\PP^3$ fibration over $\Bs$, so that its dimension is $5$.
It follows that $Z_{\CF}$ is $6$-dimensional inside a $21$-dimensional space where it is cut by a section $\sigma_{\CF}$ of a rank $15$ vector bundle.
As the ambient space is smooth, the section is regular and $Z_{\CF}$ is irreducible. 

Next, let us construct a map from $Z_{\CF}$ to the blow up $\Bl_{\Bs}\Bg$.
Note that by the definition of $Z_{\CF}$ the map \eqref{eq:Z_CF -1} vanishes over $Z_{\CF}$, so that there is an induced
\begin{equation}\label{eq:Z_CF 3}
\CF/\CRelTaut_{\CF} \to V^* \otimes \CO_{\Gr_{\CF}}
\end{equation}
By proposition \ref{prop:embed P(A)} its cokernel is a line bundle supported on $\pi_{\CF}^{-1}(\Bs)$, so that the rank of \eqref{eq:Z_CF 3} drops at most by $1$.
It follows that the schematic support of the cokernel of \eqref{eq:Z_CF 3} is cut by the determinant of \eqref{eq:Z_CF 3}.
This shows that $\pi_{\CF}^{-1}(\Bs)$ is a Cartier divisor and by universal property induces a unique map to the blow up of $\Bg$ in $\Bs$.

In the other direction, again by proposition \ref{prop:embed P(A)} and by definition of blow up, the cokernel of the pullback to $\Bl_{\Bs}\Bg$ of \eqref{eq:bundle Ker} is a line bundle supported on a Cartier divisor.
It follows by $\Tor$-dimension that the kernel of the pullback of \eqref{eq:bundle Ker} is a rank $4$ subbundle of $A \otimes \CK$.
Moreover, by lemma \ref{lemma:from 9 to 8}, the line subbundle
\[
\pi_4^*\CO_{\Bg}(-1) \to A \otimes \CK
\]
factors via this kernel, inducing a rank $3$ subbundle of $\CF$ such that the composition with
\[
\pi_{\CF}^*\CF \to V^* \otimes \CO_{\Bg}
\]
vanishes.
By universal property of $Z_{\CF}$ this induces a map from $\Bl_{\Bs}\Bg$ to it.

Moreover, by their construction, these two maps are relative to $\Bg$ and are isomorphisms out of the preimage of $\Bs$.
Note that there is a unique map from $\Bl_{\Bs}\Bg$ to itself which is relative to $\Bg$ and it is the identity.
It follows that
\begin{equation}\label{eq:Z_CF 2}
\Bl_{\Bs}\Bg \to Z_{\CF}
\end{equation}
is a closed embedding.
Note that $Z_{\CF}$ is generically reduced and that it cannot have embedded components as it is cut by a regular section.
It follows that $Z_{\CF}$ is reduced, so that  \eqref{eq:Z_CF 2} is a closed embedding of irreducible reduced varieties of the same dimension.
This is enough to conclude that it is an isomorphism.

As for the embedding statements, they follow directly from the definitions of $Z_{\CF}$ and of $\Bl_{\Bs}$.
\end{proof}

\begin{lemma}\label{lemma:Z_4 again}
The section $\sigma_4$ is regular and its zero locus $Z_4$ has $2$ components: $Z_{\CF}$ and
\[
Z_4^v := \Gr_{\Bs}(3, \CKer^{\sfs})
\]
\end{lemma}
\begin{proof}
As in lemma \ref{lemma:Z_4} one can check that $\sigma_4$ is regular.
Clearly there is an inclusion of $Z_{\CF}$ in $Z_4$ which is an isomorphism out of the preimage of $\Bs$.
On the other hand the restriction of $Z_4$ to $\Bs$ is by definition $\Gr_{\Bs}(3, \CKer^{\sfs})$.
\end{proof}

Recall that $\WZ_{\CF}$ is the result of the fiber product
\begin{equation}\label{eq:desing 1}
\begin{diagram}
\WZ_{\CF} 			& \rTo^{\Wpi_3}	& Z_{\CF}		\\
\dTo^{\Wpi_4}			&			& \dTo^{\pi_4}	\\
\CMIbar_3				& \rTo^{\pi_3}	& \Bg
\end{diagram}
\end{equation}
where we abuse the notation of \eqref{diag:flag} and use the same letters for the restricted projections. 

Our next claim is that $\CMIbar_3$ is singular on a $6$-dimensional locus and that $\WZ_{\CF}$ is a small resolution of singularities for it.
We also show that $\WZ_{\CF}$ is isomorphic to the blow up of $\CMIbar_3$ in the (non-Cartier) divisor $\CMIbar_3^{\sfs}$, which is the closure of $\CMI_3^{\sfs}$ inside $\CMIbar_3$.
\begin{theorem}\label{thm:resolution of singularity}
There is an isomorphism 
\[
\phi : \WZ_{\CF} \to \Bl_{\CMIbar_3^{\sfs}}\CMIbar_3
\] fitting in a commutative diagram
\begin{equation} \label{diag:resolution of MI3}
\begin{diagram}
\WCMI_3:=\WZ_{\CF}	& 			& 	\rTo^{\Wpi_3}			&		& \Bl_{\Bs}\Bg	\\
			& \rdTo_{\phi}			&		*				& \ruTo	& 					\\
 \dTo^{\Wpi_4}	&	\quad			&\Bl_{\CMIbar_3^{\sfs}}\CMIbar_3	&\quad \,\,	&\dTo_{\pi_4 = b}		\\
			& \ldTo				&						& \rdTo	&					\\
 \CMIbar_3	& 					& \rTo_{\pi_3 = \beta}				&		& \Bg
\end{diagram}
\end{equation}
The singular locus of $\CMIbar_3$ is 
\[
\Gr_{\Bs}(2,\CKer^{\sfs}/\CO_{\Bs}(-3)) \subset \Gr_{\Bs}(3,\CKer^{\sfs}) \cong \CMIbar_3^{\sfs}
\]
and the map $\Wpi_4$ is a small resolution of singularities for $\CMIbar_3$ such that
\[
\CO_{\CMIbar_3} \to R\Wpi_{4*} \CO_{\WCMI_3}
\]
is an isomorphism, i.e. the singularity of $\CMIbar_3$ is rational.
\end{theorem}
\begin{proof}
First, we construct the map
\[
\Bl_{\CMIbar_3^{\sfs}}\CMIbar_3 \to \Bl_{\Bs}\Bg
\]
By universal property of the blow-up, given a map from $\Bl_{\CMIbar_3^{\sfs}}\CMIbar_3$ to $\Bg$ it is enough to show that the pullback of $\Bs$ is a Cartier divisor.
As the rightmost square is cartesian by proposition \ref{prop:special as fiber product} and as the composition of cartesian squares is cartesian, the outer square of
\begin{equation}
\begin{diagram}
\Bl_{\CMIbar_3^{\sfs}}\CMIbar_3 \times_{\CMIbar_3} \CMIbar_3^{\sfs}	& \rTo			& \CMIbar_3^{\sfs}	& \rTo	& \PP(A)		\\
\dTo								&\square\qquad\quad	& \dTo		&\qquad \quad \square \qquad	& \dTo^{\kappa}	\\
\Bl_{\CMIbar_3^{\sfs}}\CMIbar_3							& \rTo			& \CMIbar_3		& \rTo	& \Bg
\end{diagram}
\end{equation}
is cartesian.
Moreover, by definition of the blow up of $\CMIbar_3^{\sfs} \subset \CMIbar_3$, the leftmost vertical arrow is the embedding of the exceptional divisor, which is always a Cartier divisor.

Now we construct the map 
\[
\phi: \WZ_{\CF} \to \Bl_{\CMIbar_3^{\sfs}}\CMIbar_3
\]
Again, by universal property of blow up we want to check that the leftmost vertical arrow in
\[
\begin{diagram}
\WZ_{\CF} \times_{\CMIbar_3} \CMIbar_3^{\sfs}	& \rTo				& \CMIbar_3^{\sfs}	& \rTo		& \PP(A)		\\
\dTo									&\square\qquad		& \dTo		&\quad \square	& \dTo_{\kappa}	\\
\WZ_{\CF}								& \rTo_{\Wpi_4 \qquad}	& \CMIbar_3	& \rTo_{\beta}	& \Bg
\end{diagram}
\]
is the embedding of a Cartier divisor. By commutativity of square \eqref{diag:resolution of MI3}, this is equivalent to the fact that the leftmost vertical arrow in
\[
\begin{diagram}
\WZ_{\CF} \times_{\Bl_{\Bs}\Bg} E	& \rTo	& E						& \rTo	& \PP(A)		\\
\dTo									& \square	& \dTo					& \square	& \dTo_{\kappa}	\\
\WZ_{\CF}					&\rTo_{\Wpi_3 \qquad}	& \Bl_{\Bs}\Bg		& \rTo_{b}	& \Bg
\end{diagram}
\]
is the embedding of a Cartier divisor. 
This last fact is true as $E$ is a Cartier divisor in $\Bl_{\Bs}\Bg$ and as $\Wpi_3$ is a dominant map from an irreducible scheme.

Next, we prove commutativity of all triangles in \eqref{diag:resolution of MI3}. 
The triangles which are not marked with $*$ are commutative by construction. 
The triangle marked with $*$ is commutative because by universal property of the blow up there is a unique map from $\WZ_{\CF}$ to $\Bl_{\Bs}\Bg$ which commutes with the projections to $\Bg$, and both $\Wpi_3$ and
\[
\WZ_{\CF} \xrightarrow{\phi} \Bl_{\CMIbar_3^{\sfs}}\CMIbar_3 \to  \Bl_{\Bs}\Bg
\]
commute with the projections $\pi_3 \circ \Wpi_4$ and $\pi_4$.


We still have to prove that $\phi$ is an isomorphism. 
As $\CMIbar_3$ is integral, also its blow-up in $\CMIbar_3^{\sfs}$ is integral. 
As both $\WZ_{\CF}$ and $\Bl_{\CMIbar_3^{\sfs}}\CMIbar_3$ are integral proper schemes of dimension $9$, so that in order to check that $\phi$ is an isomorphism it is enough to check that it is a closed embedding.
By their definition, the maps $\Wpi_3$ and $\Wpi_4$ induce a closed embedding
\[
\WZ_{\CF} \xrightarrow{(\Wpi_3, \Wpi_4)} \Gr(3,\CK \otimes A) \times_{\Bg} \Gr(4, \CK \otimes A)
\]
As we have just proved that diagram \eqref{diag:resolution of MI3} is commutative, there is an induced map from $\Bl_{\CMIbar_3^{\sfs}}\CMIbar_3$ to the fiber product of $\pi_3$ and $\pi_4$ such that the triangle
\[
\begin{diagram}
\WZ_{\CF} 				&		& 	\\
\dTo^{\phi}				& \rdInto	&	\\
\Bl_{\CMIbar_3^{\sfs}}\CMIbar_3	& \rTo	&\Gr(3,\CK \otimes A) \times_{\Bg} \Gr(4, \CK \otimes A)
\end{diagram}
\]
commutes.
As $\phi$ is the first factor of a closed embedding, it is itself a closed embedding, and therefore an isomorphism.

As for the last part of the statement, $\Wpi_3$ is a smooth map as it is the projection associated with the projectivization of a vector bundle, so that $\WZ_{\CF}$ is smooth as $\Bl_{\Bs} \Bg $ is smooth.

To describe the singular locus of $\CMIbar_3$, recall that by lemma \ref{lemma:from 9 to 8} there is an $\SSL_2$-equivariant exact sequence
\[
0 \to \CO_{\Bg}(-1) \to A \otimes \CK \to \CF \to 0
\]
inducing an embedding
\[
\Gr_{\Bs}(2, \CKer^{\sfs}/\CO_{\Bs}(-3)) \to \Gr_{\Bs}(3, \CKer^{\sfs}) \cong \CMIbar_3^{\sfs} 
\]
Over points in $\Gr_{\Bs}(2, \CKer^{\sfs}/\CO_{\Bs}(-3))$ the fiber of $\Wpi_4$ is $\PP^1$, while over the complement of $\Gr_{\Bs}(2, \CKer^{\sfs}/\CO(-3))$ inside $\CMIbar_3$ it has relative dimension $0$.

As a consequence, each point in $\Gr_{\Bs}(2, \CKer^{\sfs}/\CO(-3))$ is in the singular locus, as otherwise, by choosing a small enough neighborhood of the point, one would find a non-trivial small contraction between smooth varieties, which is impossible.

Now we are going to check that over the complement of $\Gr_{\Bs}(2, \CKer^{\sfs}/\CO(-3))$ in $\CMIbar_3$ the map $\Wpi_4$ is an isomorphism. 
Note that the restriction of
\[
\Wpi_4 : \WCMI_3 \to \CMIbar_3
\]
to
\[
\CMIbar_3 \setminus \Gr_{\Bs}\left(2, \CKer^{\sfs}/\CO(-3)\right)
\]
is proper of relative dimension $0$. 
If we prove that
\[
\CO_{\CMIbar_3} \cong R\Wpi_{4*}\CO_{\WCMI_3}
\]
then by lemma \ref{lemma:closed embedding and pushforward} we have proved that the restricted projection $\Wpi_4$ is an isomorphism.
This will prove at the same time that the singularity of $\CMIbar_3$ is rational and is the content of lemma \ref{lemma:Koszul of Z_4} and lemma \ref{lemma:rational sing}.
\end{proof}

\begin{lemma}\label{lemma:Koszul of Z_4}
The adjunction map
\begin{equation}\label{eq:Koszul of Z_4 bis 1}
\CO_{Z_3} \to R\Wpi_{4*}\CO_{\WZ_4}
\end{equation}
is an isomorphism.
\end{lemma}
\begin{proof}
Take the locally free resolution of $\CO_{Z_4}$ given by the Koszul complex of
\begin{equation}\label{eq:Koszul of Z_4 bis}
V \otimes \CRelTaut_4 \xrightarrow{\sigma_4} \CO_{\Gr_{\Bg}(4,\CK \otimes A)}
\end{equation}
Pull it back via the flat map $\Wpi_3$ to find a Koszul resolution for $\CO_{\WZ_4}$.

We are now going to pushforward the Koszul resolution for $\CO_{\WZ_4}$ via $\Wpi_{4}$.
Clearly,
\[
R\Wpi_{4*}\CO_{\Flag_{\Bg}(3,4,A\otimes \CK)} \cong \CO_{\Gr_{\Bg}(3,A \otimes \CK)}
\]
The tautological flag fits in an exact sequence
\begin{equation}\label{eq:tautological flag}
0 \to \Wpi_4^* \CRelTaut_3 \to \Wpi_3^* \CRelTaut_4 \to \CO_{\Wpi_4}(-1) \to 0
\end{equation}
so that
\[
R\Wpi_{4*}\Wpi_3^* \CRelTaut_4 \cong R\Wpi_{4*} \Wpi_4^* \CRelTaut_3
\]
By projection formula
\[
R\Wpi_{4*} \Wpi_4^* \CRelTaut_3 \cong \CRelTaut_3
\]

As the projections $\Wpi_4$ and $\Wpi_3$ are $SL_2$-equivariant, the $SL_2$-invariant section \eqref{eq:Koszul of Z_4 bis} is transformed into the unique $SL_2$-invariant section
\[
V \otimes \CRelTaut_3 \xrightarrow{\sigma_3} \CO_{\Gr_{\Bg}(3,\CK \otimes A)}
\]
In order to prove \eqref{eq:Koszul of Z_4 bis 1}, it is enough to prove that
\[
R^i \Wpi_{4*} \Wpi_3^* \Lambda^i(\CRelTaut_4 \otimes V)
\]
vanish for $i > 0$.
The exterior powers
\[
\Lambda^i(\CRelTaut_4 \otimes V)
\]
admit a filtration (see e.g. \cite{weyman2003cohomology})
with subquotients
\begin{equation}\label{eq:Koszul of Z_4 bis 2}
\Sigma^\alpha \CRelTaut_4 \otimes \Sigma^{\alpha^T} V
\end{equation}
where $\Sigma^\alpha$ is the Schur functor associated with a Young diagram $\alpha$ such that $| \alpha | = i$.
Note that \eqref{eq:Koszul of Z_4 bis 2} is not zero only when $\alpha$ is contained in a 5-by-4 rectangle, as the ranks of $\CRelTaut_4$ and $V$ are respectively $4$ and $5$.

The Schur powers of $\CRelTaut_4$ can be rewritten in terms of Schur powers of $\CRelTaut_3$ and $\CO_{\Wpi_4}(-1)$. 
This can be done via a natural filtration for  
\[
\Sigma^\alpha \CRelTaut_4
\]
induced by the tautological flag \eqref{eq:tautological flag}.
The associated graded object of this natural filtration is
\begin{equation}\label{eq:Koszul of Z_4 bis 3}
\bigoplus_{\beta \subseteq \alpha} \left( \Sigma^\beta \CRelTaut_3 \otimes \Sigma^{\alpha / \beta} \CO_{\Wpi_4}(-1) \right)
\end{equation}
where $\Sigma^{\alpha / \beta}$ are the skew Schur functors.
As $\alpha$ is contained in a 5-by-4 rectangle
\begin{equation}
\Sigma^{\alpha / \beta}(\CO_{\Wpi_4}(-1)) = \CO_{\Wpi_4}(-j)
\end{equation}
for $j \in [0,5]$.
As $\Wpi_4$ is the projectivization of a rank $6$ vector bundle, the line bundles have vanishing higher pushforward.
Finally, by projection formula the pushforward of summands of \eqref{eq:Koszul of Z_4 bis 3} is
\[
\Sigma^\beta \CRelTaut_3 \otimes R \Wpi_{4*} \CO_{\Wpi_4}(-j)
\]
which is always concentrated in degree $0$.
\end{proof}

\begin{lemma}\label{lemma:rational sing}
The adjunction map
\begin{equation}
\CO_{\CMIbar_3} \to R\Wpi_{4*}\CO_{\WZ_{\CF}}
\end{equation}
is an isomorphism.
\end{lemma}
\begin{proof}
On $\WZ_4$ there are exact sequences
\[
0 \to \CI_{\WZ_{\CF}/\WZ_4} \to \CO_{\WZ_4} \to \CO_{\WZ_{\CF}} \to 0
\]
and
\[
0 \to \CI_{\WZ_4^v \cap \WZ_{\CF}/ \WZ_4^v} \to \CO_{\WZ_4^v} \to \CO_{\WZ_4^v \cap \WZ_{\CF}} \to 0
\]
By lemma \ref{lemma:Koszul of Z_4}
\[
R\Wpi_{4*} \CO_{\WZ_4} \cong \CO_{Z_3}
\]
so that our claim becomes that
$R\Wpi_{4*} \CI_{\WZ_{\CF}/\WZ_4}$ vanishes,
or equivalently
\begin{equation}\label{eq:rational sing}
R\Wpi_{4*}\CI_{\WZ_4^v \cap \WZ_{\CF}/ \WZ_4^v} = 0
\end{equation}

We are now going to prove that \eqref{eq:rational sing} is an isomorphism. 
To this purpose, recall that the intersection of $Z_4^v$ and $Z_{\CF}$ inside $\Gr_{\Bs}(4, \CKer^{\sfs})$ 
is described in lemma \ref{lemma:Z_4} and \ref{lemma:Z_4 again}.
As a consequence, the commutative diagram
\[
\begin{diagram}
\WZ_4^v \cap \WZ_{\CF} 	& \rInto	& \WZ_4^v	\\
					& \rdTo	& \dTo		\\
					&		& \CMI_3^{\sfs}
\end{diagram}
\]
is isomorphic to the commutative diagram
\begin{equation}\label{eq:rational sing 1}
\begin{diagram}
\Flag_{\Bs}(3, \CO_{\Bs}(-3) \subset 4,\CKer^{\sfs})	& \rInto	&\Flag_{\Bs}(3,4,\CKer^{\sfs}) 	\\
						& \rdTo	& \dTo		\\
						&		& \Gr_{\Bs}(3,\CKer^{\sfs})
\end{diagram}
\end{equation}
where the projections are the natural ones to $\Gr_{\Bs}(3,\CKer^{\sfs})$ and where the upper left corner is the flag variety of subbundles of $\CKer^{\sfs}$ such that the injection \eqref{eq:from 9 to 8} restricted to $\Bs$ factors via the $4$-dimensional subbundle.

The ideal of the top left corner of diagram \eqref{eq:rational sing 1} is the pullback of $\CO(-1)$ from $\Gr(4,\CKer^{\sfs})$ (possibly with a twist by a line bundle on $\Bs$). 
It follows that it is acyclic on the fibers of the projection to $\Gr_{\Bs}(3,\CKer^{\sfs})$, so that \eqref{eq:rational sing} holds.
\end{proof}


\bibliographystyle{alpha}
\bibliography{arXiv_Sanna_Thesis.bbl}

\end{document}